\documentclass{amsart}

\usepackage{latexsym}

\usepackage{amsfonts,amssymb,euscript,epsfig,color}

\usepackage[all]{xy}

\newtheorem{theo}{Theorem}[section]
\newtheorem{defi}[theo]{Definition}

\newtheorem{lem}[theo]{Lemma}

\newtheorem{prop}[theo]{Proposition}

\newtheorem{rem}[theo]{Remark}

\newtheorem{exam}[theo]{Example}

\newcommand{\agot}{\ensuremath{\mathfrak{a}}}
\newcommand{\bgot}{\ensuremath{\mathfrak{b}}}
\newcommand{\cgot}{\ensuremath{\mathfrak{c}}}

\newcommand{\hgot}{\ensuremath{\mathfrak{h}}}
\newcommand{\ggot}{\ensuremath{\mathfrak{g}}}
\newcommand{\rgot}{\ensuremath{\mathfrak{r}}}

\newcommand{\Ecal}{\ensuremath{\mathcal{E}}}
\newcommand{\Fcal}{\ensuremath{\mathcal{F}}}

\newcommand{\Hcal}{\ensuremath{\mathcal{H}}}

\newcommand{\Lcal}{\ensuremath{\mathcal{L}}}
\newcommand{\Mcal}{\ensuremath{\mathcal{M}}}
\newcommand{\Ncal}{\ensuremath{\mathcal{N}}}

\newcommand{\Scal}{\ensuremath{\mathcal{S}}}
\newcommand{\Ucal}{\ensuremath{\mathcal{U}}}
\newcommand{\Vcal}{\ensuremath{\mathcal{V}}}
\newcommand{\Wcal}{\ensuremath{\mathcal{W}}}
\newcommand{\Xcal}{\ensuremath{\mathcal{X}}}

\newcommand{\Cbb}{\ensuremath{\mathbb{C}}}
\newcommand{\Rbb}{\ensuremath{\mathbb{R}}}
\newcommand{\Zbb}{\ensuremath{\mathbb{Z}}}

\newcommand{\Tbb}{\ensuremath{\mathbb{T}}}

\newcommand{\Ko}{\ensuremath{{\mathbf K}^0}}
\newcommand{\Kun}{\ensuremath{\mathbf{K}^1}}
\newcommand{\KK}{\ensuremath{\mathbf{K}^*}}

\newcommand{\croc}{\ensuremath{\hookrightarrow}}
\newcommand{\T}{\ensuremath{\hbox{\bf T}}}

\newcommand{\sigmad}{\sigma_{\overline{\partial}}}

\newcommand{\cu}{\underline{\mathbb{C}}}

\newcommand{\indice}{\mathrm{Index}}

\newcommand{\indGa}{\mathrm{Ind}^{G}_{G_\chi}}

\newcommand{\bott}{\operatorname{Bott}}

\newcommand{\thom}{\operatorname{Thom}}

\def \Rgene {{R^{-\infty}}}

\newcommand{\supp}{\operatorname{\hbox{\rm \small Supp}}}

\renewcommand{\index}{\operatorname{index}}

\newcommand{\teta}{\operatorname{\bf \theta}}
\newcommand{\br}{\operatorname{\bf R}}
\newcommand{\bs}{\operatorname{\bf S}}
\newcommand{\bj}{\operatorname{\bf J}}

\def \dm  {{\rm DM}}
\def \Clif {{\rm Cl}}

\def \codim {{\rm codim}}

\begin{document}

\title{On the structure of $\KK_G(\T_G^* M)$}

\author{Paul-Emile  PARADAN}

\address{Institut de Math\'ematiques et de Mod\'elisation de Montpellier (I3M), 
Universit\'e Montpellier 2} 

\email{Paul-Emile.Paradan@math.univ-montp2.fr}

\date\today

%%%%%%%%%%%%%%%%%%%%%%%%%%%%%%%%%%%%%%%%%%%%%%%%%%%%%%%%%%%%%%%%%%%%%

\begin{abstract}
In this expository paper, we revisit the results of Atiyah-Singer and de Concini-Procesi-Vergne 
concerning the structure of the $K$-theory groups $\KK_G(\T_G^* M)$.
\end{abstract}

%%%%%%%%%%%%%%%%%%%%%%%%%%%%%%%%%%%%%%%%%%%%%%%%%%%%%%%%%%%%%%%%%%%%%

\maketitle

{\def\thefootnote{\relax}
\footnote{{\em Keywords} : K-theory, Lie group, equivariant index,
transversally elliptic symbol.\\
{\em 2010 Mathematics Subject Classification} :  19K56, 19L47, 57S15   }
\addtocounter{footnote}{-1}
}

{\small
\tableofcontents}

%%%%%%%%%%%%%%%%%%%%%%%%%%%%%%%%%%%%%%%%%%%%%%%%%%%%%
%%%%%%%%%%%%%%%%%%%%%%%%%%%%%%%%%%%%%%%%%%%%%%%%%%%%%
%%%%%%%%%%%%%%%%%%%%%%%%%%%%%%%%%%%%%%%%%%%%%%%%%%%%%
\section{Introduction}
%%%%%%%%%%%%%%%%%%%%%%%%%%%%%%%%%%%%%%%%%%%%%%%%%%%%%
%%%%%%%%%%%%%%%%%%%%%%%%%%%%%%%%%%%%%%%%%%%%%%%%%%%%%
%%%%%%%%%%%%%%%%%%%%%%%%%%%%%%%%%%%%%%%%%%%%%%%%%%%%%

When a compact Lie group $G$ acts on a compact manifold $M$, the $K$-theory group $\Ko_G(\T^*M)$ 
is the natural receptacle for the principal symbol of any $G$-invariant elliptic pseudo-differential operators on $M$. 
One important point of Atiyah-Singer's Index Theory \cite{Atiyah-Singer-1,Atiyah-Segal68,Atiyah-Singer-2,Atiyah-Singer-3} 
is that the equivariant index map $\indice^G_M:\Ko_G(\T^*M)\to R(G)$ 
can be defined as the composition of a pushforward map $i_!:\Ko_G(\T^*M)\to\Ko_G(\T^*V)$ associated to an 
embedding $M\stackrel{i}{\croc} V$ in a $G$-vector space, with the index map $\indice^G_V:\Ko_G(\T^*V)\to R(G)$ 
which is the inverse of the Bott-Thom isomorphism \cite{Segal68}. 

In his Lecture Notes \cite{Atiyah74} describing joint work with I.M. Singer, Atiyah extends the index theory 
to the case of transversally elliptic operators. If we denote by $\T^*_G M$ the closed subset of $\T^*M$, 
union of the conormals to the $G$-orbits, Atiyah explains how the principal symbol of a pseudo-differential 
transversally elliptic operator on $M$ determines an element of the equivariant $K$-theory
group $\Ko_G(\T^*_G M)$, and how the analytic index induces a map 
\begin{equation}\label{eq:indice-transverse}
\indice^G_M:\Ko_G(\T^*_G M)\to R^{-\infty}(G),
\end{equation} 
where $R^{-\infty}(G):=\hom(R(G),\Zbb)$.

Like in the elliptic case the map (\ref{eq:indice-transverse}) can be seen as the composition of a push-forward map $i_!:\Ko_G(\T_G^*M)\to\Ko_G(\T_G^*V)$ with the index map $\indice^G_V:\Ko_G(\T^*_G V)\to R^{-\infty}(G)$. Hence the comprehension of the $R(G)$-module 
\begin{equation}\label{eq:K_G-T_G-intro}
\KK_G(\T^*_G V)
\end{equation} 
is fundamental in this context. For example, in \cite{B-V.inventiones.96.2,pep-vergne:bismut} the authors gave a cohomological 
formula for the index and the knowledge of the generators of $\Ko_{U(1)}(\T^*_{U(1)} \Cbb)$ was used to establish the formula. In \cite{CPV-2012}, de Concini-Procesi-Vergne proved a formula for the multiplicities of the index by checking it on the generators of 
(\ref{eq:K_G-T_G-intro}).

\medskip

When $G$ is abelian, Atiyah-Singer succeeded to find a set of generators for (\ref{eq:K_G-T_G-intro}), and recently de Concini-Procesi-Vergne have shown that the index map identifies (\ref{eq:K_G-T_G-intro}) with a generalized Dahmen-Michelli space \cite{CPV-2010-1,CPV-2010-2}.  Let us explain their result. 

Let $\widehat{G}$ be the set of characters of the abelian compact Lie group $G$ : for any $\chi\in\widehat{G}$ we denote $\Cbb_\chi$ the corresponding complex one dimensional representation of $G$. We associate to any element $\Phi:=\sum_{\chi\in \widehat{G}}m_\chi \Cbb_\chi \in R^{-\infty}(G)$ its support $\supp(\Phi)=\{\chi\ \vert\  m_\chi\neq 0\}\subset \widehat{G}$.

For any real $G$-module $V$, we denote $\Delta_G(V)$ the set formed by the infinitesimal stabilizer of points in $V$: 
we denote $\hgot_{min}$ the minimal stabilizer.  For any $\hgot\in\Delta_G(V)$, we denote $H:=\exp(\hgot)$ the corresponding torus and we 
denote $\pi_H : \widehat{G}\to \widehat{H}$ the restriction map.

We denote $R^{-\infty}(G/H)\subset R^{-\infty}(G)$ the subgroup formed by the elements $\Phi \in R^{-\infty}(G)$ such that $\pi_H(\supp(\Phi))\subset \widehat{H}$ is reduced to the trivial representation. Let 
$$
\langle R^{-\infty}(G/H)\rangle\subset R^{-\infty}(G)
$$
be the $R(G)$-submodule generated by $R^{-\infty}(G/H)$. We  have $\Phi\in \langle R^{-\infty}(G/H)\rangle$ if and only if $\pi_H(\supp(\Phi))\subset \widehat{H}$ is finite.

For any subspace $\agot\subset\ggot$, we denote $V^\agot\subset V$ the subspace formed by the vectors fixed by the infinitesimal action of $\agot$. We fix an invariant complex structure on $V/V^\ggot$, hence the vector space 
$V/V^\hgot\subset V/V^\ggot$ is equipped with a complex structure for any 
$\hgot\in\Delta_G(V)$. Following \cite{CPV-2012}, we introduce the following submodule 
of $R^{-\infty}(G)$: the Dahmen-Michelli submodule 
\begin{eqnarray*}
   \lefteqn{\dm_G(V):= }\\
   & & \langle R^{-\infty}(G/H_{min})\rangle \bigcap 
  \left\{  \Phi\in R^{-\infty}(G) \, \vert\, \wedge^\bullet\overline{V/V^\hgot}\otimes \Phi=0,\  
   \forall \hgot\neq \hgot_{min}\in \Delta_G(V)   \right\},
\end{eqnarray*}
and the generalized Dahmen-Michelli submodule
\begin{equation*}
\Fcal_G(V):=\left\{ \Phi\in R^{-\infty}(G)\  \vert\ \wedge^\bullet\overline{V/V^\hgot}\otimes \Phi\in \langle R^{-\infty}(G/H)\rangle
,\ \forall \hgot\in \Delta_G(V) \right\}.
\end{equation*}
Note that the relation $\wedge^\bullet\overline{V/V^\hgot}\otimes \Phi\in \langle R^{-\infty}(G/H)\rangle$ becomes 
$\Phi\in \langle R^{-\infty}(G/H_{min})\rangle$ when $\hgot=\hgot_{min}$. Hence $\dm_G(V)$ is contained in $\Fcal_G(V)$.
We have the following remarkable result \cite{CPV-2010-2}.

\begin{theo}[de Concini-Procesi-Vergne]\label{theo-intro-CPV}
Let $G$ be an abelian compact Lie group, and  let  $V$ be a real $G$-module. Let $V^{gen}\subset V$ be its open subset formed by the $G$-orbits of maximal dimension. The index map defines 
\begin{itemize}
\item an isomorphism between  $\Ko_G(\T^*_G V)$ and $\Fcal_G(V)$,
\item an isomorphism between $\Ko_G(\T^*_G V^{gen})$ and $\dm_G(V)$.
\end{itemize}
\end{theo}

\medskip

The purpose of this note is to give a comprehensive account on the work of Atiyah-Singer and de Concini-Procesi-Vergne concerning the structure of (\ref{eq:K_G-T_G-intro}) when $G$ is a compact abelian Lie group. We will explain in details the following facts :
\begin{itemize}
\item The decomposition of $\KK_G(\T^*_G M)$ relatively to the stratification of the manifold $M$ relatively to the type of infinitesimal stabilizers.
\item A set of generators of $\KK_G(\T^*_G V)$.
\item A set of generators of $\KK_G(\T^*_G V^{gen})$.
\item The injectivness of the index map $\indice^G_V:\Ko_G(\T^*_G V)\to R^{-\infty}(G)$.
\item The isomorphisms $\Ko_G(\T^*_G V)\simeq\Fcal_G(V)$ and  $\Ko_G(\T^*_G V^{gen})\simeq\dm_G(V)$.
\end{itemize}

\medskip

\medskip

{\bf Acknowledgements.}  I wish to thank Mich\`ele Vergne for various comments on this text.

%%%%%%%%%%%%%%%%%%%%%%%%%%%%%%%%%%%%%%%%%%%%%%%%%%%%%
%%%%%%%%%%%%%%%%%%%%%%%%%%%%%%%%%%%%%%%%%%%%%%%%%%%%%
%%%%%%%%%%%%%%%%%%%%%%%%%%%%%%%%%%%%%%%%%%%%%%%%%%%%%
\section{Preliminary on $K$-theory}
%%%%%%%%%%%%%%%%%%%%%%%%%%%%%%%%%%%%%%%%%%%%%%%%%%%%%
%%%%%%%%%%%%%%%%%%%%%%%%%%%%%%%%%%%%%%%%%%%%%%%%%%%%%
%%%%%%%%%%%%%%%%%%%%%%%%%%%%%%%%%%%%%%%%%%%%%%%%%%%%%

In this section, $G$ denotes a compact Lie group. Let $R(G)$ be the representation ring of $G$ and let 
$R^{-\infty}(G)=\hom(R(G),\Zbb)$. 

%%%%%%%%%%%%%%%%%%%%%%%%%%%%%%%%%%%%%%%%%%%%%%%%%%%%%
%%%%%%%%%%%%%%%%%%%%%%%%%%%%%%%%%%%%%%%%%%%%%%%%%%%%%
\subsection{Equivariant $K$-theory}
%%%%%%%%%%%%%%%%%%%%%%%%%%%%%%%%%%%%%%%%%%%%%%%%%%%%%
%%%%%%%%%%%%%%%%%%%%%%%%%%%%%%%%%%%%%%%%%%%%%%%%%%%%%
We briefly review the notations for K-theory that we will use, for a systematic treatment
see Atiyah \cite{Atiyah89} and Segal \cite{Segal68}.

Let $N$ be a locally compact topological space equipped with a continuous action of $G$. Let $E^{\pm}\to N$ be two $G$-equivariant complex vector bundles. An equivariant morphism $\sigma$ on $N$ is defined by a vector bundle map 
$\sigma\in \Gamma(N,\hom(E^+,E^-))$, that we denote also $\sigma:E^+\to E^-$: at each point $n\in N$, we have a linear map $\sigma(n): E^+_{n}\to E^-_{n}$. The support of the morphism $\sigma$ is the closed set formed by the point $n\in N$ where $\sigma(n)$ is not an isomorphism. We denote it  $\mathrm{Support}(\sigma)\subset N$. 

A morphism $\sigma$ is elliptic when its support is compact, and then it defines a class 
$$
[\sigma]\in \Ko_G(N)
$$
in the equivariant $\bf{K}$-group \cite{Segal68}. The group $\Kun_G(N)$ is by definition the group $\Ko_G(N\times \Rbb)$ where $G$ acts trivially on $\Rbb$.

%%With the choice of an invariant hermitian products on the bundles $E^\pm$, we can define the adjoint morphism  
%%$\sigma^*:E^-\to E^+$. Note that, when $\sigma$ is elliptic, the opposite of the class $[\sigma]$ is 
%%given by the class $[\sigma^*]$. 

Let $j:U\croc N $ be an invariant open subset, and let us denote by $r: N\setminus U\croc N$ the inclusion of the closed complement. 
We have a push-forward morphism $j_*: \KK_G(U)\to \KK_G(N)$ and a restriction morphism $r^*:  \KK_G(N)\to \KK_G(N\setminus U)$ 
that fit in a six terms exact sequence :
\begin{equation}\label{exact.sequence}
\xymatrix{
\Ko_G(U)\ar[r]^{j_*} & \Ko_G(N) \ar[r]^{r^*} & \Ko_G(N\setminus U)\ar[d]_{\delta}\\
\Kun_G(N\setminus U)\ar[u]^{\delta}& \ar[l]^{r^*}   \Kun_G(N) &\ar[l]^{j_*} \Kun_G(U)
  }
\end{equation}

In the next sections we will use the following basic lemma which is a direct consequence of (\ref{exact.sequence}).

\begin{lem}
Suppose that we have a morphism $S:\KK_G(N\setminus U)\to \KK_G(N)$ of $R(G)$-module such that $r^*\circ S$ is the identity on 
$\KK_G(N\setminus U)$. 
Then 
$$
\KK_G(N)\simeq \KK_G(U)\oplus\KK_G(N\setminus U)
$$
as $R(G)$-module.
\end{lem}

\medskip

We finish this section by considering the case of torus $\Tbb$ belonging to the center of $G$. Let 
$i: \Tbb\croc G$ be the inclusion map. We still denote $i: {\rm Lie}(\Tbb)\to\ggot$ the map of Lie algebra, and 
$i^*:\ggot^*\to  {\rm Lie}(\Tbb)^*$ the dual map. Note that the restriction to $\Tbb$ of an irreducible representation 
$V^G_\lambda$ is isomorphic to $(\Cbb_{i^*(\lambda)})^{p}$ with $p=\dim(V^G_\lambda)$. 
The representation ring $R(G)$ contains as a subring $R(G/ \Tbb)$. At each character $\mu$ of $\Tbb$, we 
associate the $R(G/ \Tbb)$-submodule of $R(G)$ defined by
$$
R(G)_\mu=\sum_{i^*(\lambda)=\mu} \Zbb V^G_\lambda.
$$
Note that $R(G)_0=R(G/ \Tbb)$. 

We have then a grading $R(G)=\bigoplus_{\mu\in\widehat{\Tbb}}R(G)_\mu$ since $R(G)_\mu\cdot R(G)_{\mu'}\subset R(G)_{\mu+\mu'}$. 
If we work now with the $R(G)$-module $\Rgene(G)$, we have also a decomposition\footnote{The sign  $\widehat{\bigoplus}$ means that
one can take infinite sum.} $\Rgene(G)=\widehat{\bigoplus}_{\mu\in\widehat{\Tbb}}\Rgene(G)_\mu$ such that 
$R(G)_\mu\cdot \Rgene(G)_{\mu'}\subset \Rgene(G)_{\mu+\mu'}$.

\medskip

Let us consider now  the case of a $G$-space $N$, connected,  such that the action of the subgroup $\Tbb$ is trivial. Each $G$-equivariant complex vector bundle $\Ecal\to N$ decomposes as a finite sum
\begin{equation}\label{eq:E-mu}
\Ecal=\bigoplus_{\mu\in \Xcal}\Ecal_\mu
\end{equation}
 where $\Ecal_\mu\simeq \hom_\Tbb(\Cbb_\mu,\Ecal)$ is the $G$-sub-bundle where $\Tbb$ acts trough the character $t\mapsto t^\mu$. 
 Note that a $G$-equivariant morphism $\sigma:\Ecal^+\to\Ecal^-$ is equal to the sum of morphisms  
 $\sigma_\mu:\Ecal^+_\mu\to\Ecal^-_\mu$. Hence, at the level of {\bf K}-theory we have also a decomposition
\begin{equation}\label{eq:K-mu}
\KK_G(N)=\bigoplus_{\mu\in \widehat{\Tbb}}\KK_G(N)_\mu
\end{equation}
such that $R(G)_\mu\cdot \KK_G(N)_{\mu'}\subset \KK_G(N)_{\mu+\mu'}$.

\begin{defi}\label{def:hat.K}
We denote $\widehat{\KK_G}^\Tbb(N)$ or simply $\widehat{\KK_G}(N)$ the $R(G)$-module formed by the {\em infinite sum} 
$\sum_{\mu\in \widehat{\Tbb}}\sigma_\mu$ with $\sigma_\mu\in\KK_G(N)_\mu$. When $N=\{\bullet\}$, $\widehat{\Ko_G}(\bullet)=\widehat{R(G)}$ is a $R(G)$-submodule of $\Rgene(G)$.
\end{defi}

%%%%%%%%%%%%%%%%%%%%%%%%%%%%%%%%%%%%%%%%%%%%%%%%%%%%%
%%%%%%%%%%%%%%%%%%%%%%%%%%%%%%%%%%%%%%%%%%%%%%%%%%%%%
\subsection{Index morphism : excision and free action}\label{sec:index}
%%%%%%%%%%%%%%%%%%%%%%%%%%%%%%%%%%%%%%%%%%%%%%%%%%%%%
%%%%%%%%%%%%%%%%%%%%%%%%%%%%%%%%%%%%%%%%%%%%%%%%%%%%%

When $M$ is a compact $G$-manifold, an equivariant morphism $\sigma$ on the cotangent bundle $\T^*M$ is called a symbol on $M$. We denote by $\T^*_G M$ the following subset of $\T^*M$
$$
\T^*_G M:=\{(m,\xi)\in \T^*M\ \vert\ \langle \xi, X_M(m)\rangle=0\  \mathrm{for \ all}\ X\in\ggot\}.
$$
where $X_M(m):=\frac{d}{dt} e^{-tX}\cdot m\vert_{t=0}$ is the vector field generated by the infinitesimal action of $X\in \ggot$. More generally, if $D\subset G$ is a distinguished subgroup, we can consider the $G$-invariant subset 
\begin{equation}\label{eq:T-Z-M}
\T^*_D M\supset \T^*_G M.
\end{equation}

 An elliptic symbol $\sigma$ on $M$ defines an element of $\Ko_G(\T^*M)$, and the index of $\sigma$ is a virtual finite dimensional
representation of $G$ that we denote $\indice_M^G(\sigma)$ \cite{Atiyah-Singer-1,Atiyah-Segal68,Atiyah-Singer-2,Atiyah-Singer-3}. An equivariant symbol $\sigma$ on $M$ is transversally elliptic when $\mathrm{Support}(\sigma)\cap \T^*_G M$ is compact: in this case Atiyah and Singer have shown that its index, still denoted  $\indice_M^G(\sigma)$, is well defined in $R^{-\infty}(G)$ and its depends only of the class $[\sigma]\in \Ko_G(\T^*_G M)$ (see \cite{Atiyah74} for the analytic index and \cite{pep-vergne:bismut} for the cohomological one). It is interesting to look at the index map as a pairing 
\begin{equation}\label{eq:pairing}
\indice_M^G: \Ko_G(\T_{G}^*M)\times \Ko_G(M)\to \Rgene(G).
\end{equation}

%% Note that we have a natural restriction map $\Ko_G(\T^* M)\to\Ko_G(\T^*_G M)$ which makes
%%the following diagram
%%\begin{equation}\label{indice.generalise}
%%\xymatrix{
%%\Ko_G(\T^* M)\ar[r]\ar[d]_{\indice_M^G} &
%%\Ko_G(\T_{G}^*M)\ar[d]^{\indice_M^G}\\
%%R(G)\ar[r] & R^{-\infty}(G)
%%   }
%%\end{equation}
%%commutative. 

\medskip

Let $\sigma$ be a $G_1\times G_2$-equivariant symbol $\sigma$ on a manifold $M$. If $\sigma$ is $G_1$-transversally elliptic it defines a class 
$$
[\sigma]\in \Ko_{G_1\times G_2}(\T_{G_1}^*M), 
$$
and its index is {\em smooth} relatively to $G_2$. It means that $\indice_M^{G_1\times G_2}(\sigma)=\sum_{\mu\in\widehat{G_1}} \theta_\mu \otimes V_\mu^{G_1}$ where $\theta_\mu\in R(G_2)$ for any $\mu$. Hence 
\begin{itemize}
\item the $G_1$-index $\indice_M^{G_1}(\sigma)= \sum_{\mu\in\widehat{G_1}} \dim(\theta_\mu) \otimes V_\mu^{G_1}$ is equal to the restriction of $\indice_M^{G_1\times G_2}(\sigma)$ to $g=1\in G_2$.
\item the product of $\indice_M^{G_1\times G_2}(\sigma)$ with any element $\Theta\in \Rgene(G_1)$ is a well defined element 
$\Theta\cdot \indice_M^{G_1\times G_2}(\sigma)\in \Rgene(G_1\times G_2)$.
\end{itemize}

\begin{rem}\label{rem:indice.mu}
Suppose that a torus $\Tbb$ belonging to the center of $G$ acts trivially on the manifold $M$. 
Since the index map $\indice_M^G$ is a morphism of $R(G)$-module, the pairing (\ref{eq:pairing}) specializes in a map 
from $\Ko_G(\T_{G}^*M)_\mu\times \Ko_G(M)_{\mu'}$ into $\Rgene(G)_{\mu+\mu'}$. Hence on can extend the pairing (\ref{eq:pairing}) to
\begin{equation}\label{eq:pairing.hat}
\indice_M^G: \Ko_G(\T_{G}^*M)\times \widehat{\Ko_G}(M)\to \Rgene(G).
\end{equation}
See Definition \ref{def:hat.K}, for the notation $\widehat{\Ko_G}(M)$.
\end{rem}

Let $U$ be a non-compact $K$-manifold. Lemma 3.6 of \cite{Atiyah74} tell us that, for any open
$K$-embedding $j:U\croc M$ into a compact manifold, we have a push-forward map
$j_{*}:\KK_{G}(\T_{G}^* U)\to \KK_{G}(\T_{G}^* M)$.

Let us rephrase Theorem 3.7 of \cite{Atiyah74}.
\begin{theo}[Excision property]
The composition
$$
\Ko_{G}(\T_{G}^*U)\stackrel{j_{*}}{\longrightarrow} \Ko_{G}(\T_{G}^*M)
\stackrel{\index^{G}_{M}}{\longrightarrow} R^{-\infty}(G)
$$
is independent of the choice of $j:U\croc M$: we denote this map $\index^{G}_{U}$.
\end{theo}

Note that a {\em relatively compact} $G$-invariant open subset $U$ of a
$G$-manifold admits an open $G$-embedding $j:U\croc M$ into a compact
$G$-manifold. So the index map $\index^{G}_{U}$ is defined in this case.
Another important example is when $U\to N$ is a $G$-equivariant vector bundle
over a compact manifold $N$ : we can imbed $U$ as an open subset
of the real projective bundle $\mathbb{P}(U\oplus \Rbb)$.

\medskip

Let $K$ be another compact Lie group. Let $P$ be a compact manifold provided with an action of
$K\times G$. We assume that the action of $K$ is free. Then the manifold $M:=P/K$ is provided with an
action of $G$ and the quotient map $q:P\to M$ is $G$-equivariant. Note that we have the natural
identification of
$\T^*_K P$ with $q^*\T^* M$, hence $(\T^*_K P)/K\simeq \T^*M$ and more generally
$$
(\T^*_{K\times G} P)/K\simeq \T_G^*M.
$$
This isomorphism induces an isomorphism
$$
Q^*: \Ko_{G}(\T^*_{G} M)\to \Ko_{K\times G}(\T^*_{K\times G} P).
$$
%%Let $\Ecal^\pm$ be two $G$-equivariant complex vector bundles on $M$ and $\sigma:p^*\Ecal^+\to p^*\Ecal^-$
%%be a $G$-transversally elliptic symbol. For any finite dimensional irreducible
%%representation $(\tau,V_\tau)$ of $K$, we form the $G$-equivariant complex vector bundle
%%$\Vcal_\tau:= P\times_K V_\tau$ on $M$. We consider the morphism
%%$$
%%\sigma_\tau:=\sigma\otimes {\rm Id}_{V_\tau}:p^*(\Ecal^+\otimes \Vcal_\tau)\to p^*(\Ecal^-\otimes \Vcal_\tau)
%%$$
%%which is $G$-transversally elliptic.

The following theorem  was obtained by Atiyah-Singer in \cite{Atiyah74}. For any $\Theta\in \Rgene(K\times G)$, we denote 
$[\Theta]^K\in \Rgene(G)$ its $K$-invariant part.

\begin{theo}[Free action property] \label{theo:free.action}
 For any $[\sigma]\in \Ko_{G}(\T^*_{G} M)$, we have the following equality in
$\Rgene(K)$: 
$$
\left[\index^{K\times G}_P(Q^*[\sigma])\right]^K=\index^{G}_M([\sigma]).
$$
\end{theo}

%%%%%%%%%%%%%%%%%%%%%%%%%%%%%%%%%%%%%%%%%%%%%%%%%%%%%
%%%%%%%%%%%%%%%%%%%%%%%%%%%%%%%%%%%%%%%%%%%%%%%%%%%%%
\subsection{Product }\label{sec:product}
%%%%%%%%%%%%%%%%%%%%%%%%%%%%%%%%%%%%%%%%%%%%%%%%%%%%%
%%%%%%%%%%%%%%%%%%%%%%%%%%%%%%%%%%%%%%%%%%%%%%%%%%%%%

Suppose that we have two $G$-locally compact topological spaces  $N_k, k=1,2$. For $j\in \Zbb/2\Zbb$, we have a product 
\begin{equation}\label{product=odot}
\odot_{ext}: \mathbf{K}_G^{0}(N_1)\times  \KK_G(N_2)\longrightarrow  \KK_G(N_1\times N_2)
\end{equation}
which is defined as follows \cite{Atiyah74}.  Suppose first that $*=0$. For $k=1,2$, let $\sigma_k: E^+_k\to E^-_k$ be a morphism on $N_k$. Let $E^\pm$ be the vector bundles on 
$N_1\times N_2$ defined as $E^+= E_1^+\otimes E_2^+ \bigoplus  E_1^-\otimes E_2^-$ and 
$E^-= E_1^-\otimes E_2^+ \bigoplus  E_1^+\otimes E_2^-$.  On $N_1\times N_2$, the morphism
$\sigma_1\odot_{ext}\sigma_2: E^+\to E^-$, is defined by the matrix
$$
\sigma_1\odot_{ext}\sigma_2(a,b)=
\left(\begin{array}{cc} 
\sigma_1(a)\otimes Id &  -Id \otimes \sigma_2(b)^* \\ 
Id \otimes \sigma_2(b)& \sigma_1(a)^*\otimes Id 
\end{array}\right).
$$
for $(a,b)\in N_1\times N_2$. Note that $\mathrm{Support}(\sigma_1\odot_{ext}\sigma_2)=\mathrm{Support}(\sigma_1)\times\mathrm{Support}(\sigma_2)$. Hence the product $\sigma_1\odot_{ext}\sigma_2$ is elliptic when each $\sigma_k$ is elliptic, and the product $[\sigma_1]\odot_{ext}[\sigma_2]$ is defined as the class $[\sigma_1\odot_{ext}\sigma_2]$. When $*=1$, we make the same construction with the spaces $N_1$ and $N_2\times \Rbb$. 

Two particular cases of this product are noteworthy:

 - When $N_1=N_2=N$, the inner product on $\Ko_G(N)$ is  defined as $a\odot b=\Delta^*(a\odot_{ext}b)$, where $\Delta^*: \Ko_G(N\times N)\to\Ko_G(N)$ is the restriction morphism associated to the diagonal mapping $\Delta:N\to N\times N$. 
 
 - The structure of $R(G)$-module of $\KK_G(N_2)$ can be understood as a particular case of the exterior product, when 
$N_1$ is reduced to a point.

\medskip

%%\begin{rem}
%%In order to simplify the notation, we will use the same notation $\odot$ for any of the product in the $\mathbf{K}$-group.
%%\end{rem}

Let us recall the multiplicative property of the index for the product
of manifolds. Consider a compact Lie group $G_2$ acting on two manifolds $M_1$ and $M_2$, 
and assume that another compact Lie group $G_1$ acts on $M_1$ commuting with the action of $G_2$. 
The external product of complexes on $\T^*M_1$ and $\T^*M_2$ induces
a multiplication (see (\ref{product=odot})):
$$
\odot_{ext}:\Ko_{G_1\times G_2}(\T^*_{G_1} M_1)\times \KK_{G_2}(\T^*_{G_2} M_2)
\longrightarrow\KK_{G_1\times G_2}(\T^*_{G_1\times G_2} (M_1\times M_2)).
$$

\medskip

Since $\T^*_{G_1\times G_2}(M_1\times M_2)\neq \T^*_{G_1} M_1\times\T^*_{G_2} M_2$ in general, the product 
$[\sigma_1]\odot_{\rm ext}[\sigma_2]$ of transversally elliptic symbols need some care: we have to take representative 
$\sigma_2$ that are almost homogeneous (see Lemma 4.9 in \cite{pep-vergneIII}).

\begin{theo}[Multiplicative property] \label{theo:multiplicative-property}
For any $[\sigma_1]\in \Ko_{G_1\times G_2}(\T^*_{G_1} M_1)$ and
any $[\sigma_2]\in\Ko_{G_2}(\T^*_{G_2} M_2)$ we have
$$
\index^{G_1\times G_2}_{M_1\times M_2}([\sigma_1]\odot_{\rm ext}[\sigma_2])
=\index^{G_1\times G_2}_{M_1}([\sigma_1])\index^{G_2}_{M_2}([\sigma_2]).
$$
\end{theo}

In the last theorem, the product of  $\index^{G_1\times G_2}_{M_1}([\sigma_1])\in \Rgene(G_1\times G_2)$ and 
$\index^{G_2}_{M_2}([\sigma_2])\in \Rgene(G_2)$ is well defined since $\index^{G_1\times G_2}_{M_1}([\sigma_1])$ 
is {\em smooth} relatively to $G_2$ (see Section \ref{sec:index}).

\medskip

Suppose now that $G$ is {\bf abelian}. For a generalized character $\Phi\in R^{-\infty}(G)$, we consider its support $\supp(\Phi)\subset\widehat{G}$ and the corresponding subset  $\overline{\supp(\Phi)}\subset\ggot^*$ formed by the differentials. 

Let $\agot\subset \ggot$ a rational\footnote{A subspace $\agot\subset \ggot$ is rational when it is the Lie algebra of a closed subgroup.} subspace, and let $\pi_\agot:\ggot^*\to \agot^*$ be the projection.  We will be interested to the $K$-groups $\KK_{G}(\T^*_{\agot} M)$ associated to the $G$-spaces $\T^*_\agot M:=\{(m,\xi)\in \T^*M\ \vert\ \langle \xi, X_M(m)\rangle=0\  \mathrm{for \ all}\ X\in\agot\}$.  
We can prove that if $\sigma\in \Ko_{G}(\T^*_{\agot} M)$, then its index $\Phi:=\indice^M_G(\sigma)\in R^{-\infty}(G)$ has the following property : the projection $\pi_\agot$, when restricted to $\overline{\supp(\Phi)}$, is {\em proper} (see \cite{B-V.inventiones.96.2}).

We have another version of Theorem \ref{theo:multiplicative-property}.

\begin{theo}[Multiplicative property - Abelian case] \label{theo:multiplicative-property-abelian}
Let $M_1$ and $M_2$ be two $G$-manifolds (with $G$ abelian), and let $\agot_1,\agot_2$ be two rationnal subspaces of $\ggot$ 
such that $\agot_1\cap\agot_2=\{0\}$. If the infinitesimal action of 
$\agot_1$ is {\bf trivial} on $M_2$, we have an external product
$$
\odot_{ext}:\Ko_{G}(\T^*_{\agot_1} M_1)\times \KK_{G}(\T^*_{\agot_2} M_2)
\longrightarrow\KK_{G}(\T^*_{\agot_1\oplus \agot_2} (M_1\times M_2)),
$$
and for any $[\sigma_k]\in \Ko_{G}(\T^*_{\agot_k} M_k)$  we have
$$
\index^{G}_{M_1\times M_2}([\sigma_1]\odot_{\rm ext}[\sigma_2])
=\index^{G}_{M_1}([\sigma_1])\index^{G}_{M_2}([\sigma_2]).
$$
\end{theo}
 
Let us briefly explain why the product of the generalized characters $\Phi_k:=\index^{G}_{M_k}([\sigma_k])\in \Rgene(G)$ is 
well-defined. We know that the projection $\pi_k:\ggot^*\to \agot_k^*$ is proper when restricted to the infinitesimal 
support $\overline{\supp(\Phi_k)}\subset\ggot^*$. Since the infinitesimal action of $\agot_1$ is trivial on $M_2$, 
we know also that the image of $\overline{\supp(\Phi_2)}$ by $\pi_1$ is finite (see Remark \ref{rem:indice-H-trivial}). 
These three facts insure that for any $\chi \in \widehat{G}$ the set 
$\{(\chi_1,\chi_2)\in \supp(\Phi_1)\times \supp(\Phi_2)\ \vert \chi_1+\chi_2=\chi \}$
 is finite. Hence we can define the product $\Phi_1\otimes\Phi_2$ as the restriction of 
 $(\Phi_1,\Phi_2)\in R^{-\infty}(G\times G)$ to the diagonal.
 
\medskip

\begin{rem}\label{rem:indice-H-trivial}
Consider an action of a compact {\bf abelian} Lie group $G$ on a manifold $M$. Suppose that a torus subgroup 
$H\subset G$ acts {\bf trivially} on $M$. Let $H'$ be a closed subgroup of $G$ such that $G\simeq H\times H'$. 
In this case we have an isomorphism $\KK_{G}(\T^*_{G} M)\simeq R(H)\otimes \KK_{H'}(\T^*_{H'} M)$
and we see that the index map sends $\Ko_{G}(\T^*_{G} M)$ into $R(H)\otimes R^{-\infty}(H')\simeq 
\langle R^{-\infty}(G/H)\rangle$.  
See the introduction where the submodule $\langle R^{-\infty}(G/H)\rangle$ is defined without using a decomposition 
$G\simeq H\times H'$.
\end{rem}
 
\medskip

%%%%%%%%%%%%%%%%%%%%%%%%%%%%%%%%%%%%%%%%%%%%%%%%%%%%%
%%%%%%%%%%%%%%%%%%%%%%%%%%%%%%%%%%%%%%%%%%%%%%%%%%%%%
\subsection{Direct images and Bott symbols}\label{section:direct.image}
%%%%%%%%%%%%%%%%%%%%%%%%%%%%%%%%%%%%%%%%%%%%%%%%%%%%%
%%%%%%%%%%%%%%%%%%%%%%%%%%%%%%%%%%%%%%%%%%%%%%%%%%%%%

Let $\pi:\Ecal\to N$ be a $G$-equivariant  complex vector bundle. We define the Bott morphism on $\Ecal$ 
$$
\bott(\Ecal):\wedge^+\pi^*\Ecal\to\wedge^- \pi^*\Ecal,
$$
by the relation $\bott(\Ecal)(n,v)=\Clif(v):\wedge^+\Ecal_{n}\to\wedge^- \Ecal_{n}$. Here the Clifford map is defined after the choice of a
$G$-invariant Hermitian product on $\Ecal$. 

Let $s:N\to \Ecal$ be the $0$-section map. Since the support of $\bott(\Ecal)$ is the zero section, we have a push-forward morphism 
\begin{eqnarray}\label{thom.iso}
s_!: \KK_G(N)&\longrightarrow& \KK_G(\Ecal)\\
\sigma &\longmapsto& \bott(\Ecal)\odot_{\rm ext} \pi^*(\sigma) \nonumber
\end{eqnarray}
which is bijective: it is the Bott-Thom isomorphism \cite{Segal68}.

\medskip

Consider now an Euclidean vector space $V$. Then its complexification $V_\Cbb$ is an Hermitian
vector space. The cotangent bundle $\T^*V$ is identified with $V_\Cbb$: we associate
to the covector $\xi\in \T^*_v V$ the element $v+i\hat{\xi}\in V_\Cbb$, where
$\xi\in V^*\to\hat{\xi}\in V$ is the identification given by the Euclidean structure.

Then $\bott(V_\Cbb)$ defines an elliptic symbol on $V$ which is equivariant relative to the action of
the orthogonal group $O(V)$.  Its analytic index is computed in \cite{Atiyah74}. We have the equality
\begin{equation}\label{eq.indice.bott}
  \index^{O(V)}_V(\bott(V_\Cbb))=1
 \end{equation}
in $R(O(V))$.

Let $\pi : \Vcal\to M$ be a $G$-equivariant {\em real} vector bundle over a compact manifold. We have the fundamental fact
\begin{prop}\label{prop.push.vector.bundle}
We  have a push-forward morphism 
\begin{equation}\label{thom.tangent}
s_!: \KK_G(\T^*_G M)\longrightarrow \KK_G(\T^*_G \Vcal)
\end{equation}
such that $\index^G_{\Vcal}\circ\ s_! =\index^G_{M}$ on $\Ko_G(\T^*_G M)$.
\end{prop}

\begin{proof} 
We fix a $G$-invariant euclidean structure on $\Vcal$. Let $n={\rm rank}\ \Vcal$. Let $P$ be the associated orthogonal frame bundle. We have $M=P/O$ and $\Vcal=P\times_{O} V$ where $V=\Rbb^n$ and $O$ is the orthogonal group of $V$.  For the cotangent bundle we have canonical isomorphisms 
$$
\T^*_G M\simeq \T^*_{G\times O}(P/O)\quad \mathrm{and}\quad \T^*_G\Vcal\simeq \T^*_{G\times O}(P\times V)/O
$$
which induces isomorphisms between $\mathbf{K}$-groups
\begin{eqnarray*}
Q^*_1: \KK_G(\T^*_G M)&\longrightarrow&\KK_{G\times O}(\T^*_{G\times O} P),\\
Q^*_2: \KK_G(\T^*_G \Vcal) &\longrightarrow&\KK_{G\times O}(\T^*_{G\times O} (P\times V)).
 \end{eqnarray*}
 
 Let us use the multiplicative property (see Section \ref{sec:product}) with the groups $G_2=G\times O, G_1=\{1\}$ and the manifolds
 $M_1=V, M_2= P$. We have a map
 \begin{eqnarray}\label{thom.transversally.bott}
s'_!:  \KK_{G\times O}(\T^*_{G\times O} P)&\longrightarrow&  \KK_{G\times O}(\T^*_{G\times O} (P\times V))\\
\sigma &\longmapsto& \bott(V_\Cbb)\odot_{\rm ext} \sigma \nonumber
\end{eqnarray}

The map $s_!: \KK_G(\T^*_G M)\longrightarrow \KK_G(\T^*_G \Vcal)$ is defined by the relation $s_!=  Q^*_1\circ s'_!\circ (Q^*_2)^{-1}$. 

Thanks to Theorem \ref{theo:multiplicative-property}, the relation (\ref{eq.indice.bott})  implies that
$\index^{G\times O}_{P\times V}\circ s'_!= \index^{G\times O}_{P}$ on $\Ko_{G\times O}(\T^*_{G\times O} P)$. Thanks to Theorem \ref{theo:free.action} we have
 \begin{eqnarray*}
\index^{G}_{\Vcal}( s_!(\sigma))&=&\left[\index^{G\times O}_{P\times V}( s'_!\circ (Q^*_2)^{-1}(\sigma))\right]^{O}\\
&=& \left[\index^{G\times O}_{P}(  (Q^*_2)^{-1}(\sigma))\right]^{O}\\
&=& \index^{G}_{M}( \sigma).
\end{eqnarray*}
for any $\sigma\in \Ko_{G\times O}(\T^*_{G\times O} P)$.
\end{proof}

\medskip

We finish this section by considering the case of a $G$-equivariant embedding $i:Z\croc M$ between $G$-manifolds.

\begin{prop}\label{push=immersion}
We  have a push-forward morphism 
\begin{equation}\label{thom.embedding}
i_!: \KK_G(\T^*_G Z)\longrightarrow \KK_G(\T^*_G M)
\end{equation}
such that $\index^G_{M}\circ\ i_! =\index^G_{Z}$ on $\Ko_G(\T^*_G Z)$.
\end{prop}

\begin{proof}Let $\Ncal=\T M\vert_{Z} /\T Z$ be the normal bundle. We know that an  open $G$-invariant tubular 
neighborhood $U$ of $Z$ is equivariantly diffeomorphic with $\Ncal$: let us denote by $\varphi:U\to\Ncal$ this equivariant diffeomorphism. Let $j:U\croc M$ be the inclusion. We consider the morphism $s_!: \KK_G(\T^*_G Z)\longrightarrow \KK_G(\T^*_G \Ncal)$ 
defined in Proposition \ref{prop.push.vector.bundle}, the isomorphism
$\varphi^*:\KK_G(\T^*_G \Ncal)\longrightarrow \KK_G(\T^*_G U)$ and the push-forward morphism 
$j_*:\KK_G(\T^*_G U)\longrightarrow \KK_G(\T^*_G M)$. Thanks to Proposition \ref{prop.push.vector.bundle}, one sees that the composition $i_!= j_*\circ\varphi^*\circ s_!$ satisfies 
$\index^G_{M}\circ\ i_! =\index^G_{Z}$ on  $\Ko_G(\T^*_G Z)$.
\end{proof}

%%%%%%%%%%%%%%%%%%%%%%%%%%%%%%%%%%%%%%%%%%%%%%%%%%%%%
%%%%%%%%%%%%%%%%%%%%%%%%%%%%%%%%%%%%%%%%%%%%%%%%%%%%%
\subsection{Restriction : the vector bundle case}
%%%%%%%%%%%%%%%%%%%%%%%%%%%%%%%%%%%%%%%%%%%%%%%%%%%%%
%%%%%%%%%%%%%%%%%%%%%%%%%%%%%%%%%%%%%%%%%%%%%%%%%%%%%

Let $\Ecal\to M$ be a $G$-equivariant complex vector bundle. Let us introduce the invariant open subset 
$\T^*_G (\Ecal\setminus\{0\})$ of $\T^*_G \Ecal$ and its complement  $\T^*_G 
\Ecal\vert_{0-\mathrm{section}}=\T^*_G M\times \Ecal^*$. We denote
\begin{equation}\label{eq:morphism-R}
\br : \KK_G(\T^*_G \Ecal)\longrightarrow \KK_G(\T^*_G M)
\end{equation}
the composition of the restriction morphism 
$\KK_G(\T^*_G \Ecal)\to\KK_G(\T^*_G M\times \Ecal^*)$ with the Bott-Thom isomorphism 
$\KK_G(\T^*_G M\times \Ecal^*)\simeq \KK_G(\T^*_G M)$. Note that the morphism 
\begin{equation}\label{eq:morphism-R-distinguished}
\br : \KK_G(\T^*_D \Ecal)\longrightarrow \KK_G(\T^*_D M)
\end{equation} 
is also defined when $D\subset G$ is a {\bf distinguished} subgroup.

If $\Scal=\{v\in \Ecal\ \vert \ \|v\|^2=1\}$ is the sphere bundle, we have $\Ecal\setminus\{0\}\simeq 
\Scal\times\Rbb$ and then $\T^*_G (\Ecal\setminus\{0\})\simeq \T^*_G \Scal\times \T^*\Rbb$. 
 Let $i:\Scal\croc \Ecal$ be the canonical immersion. 
The composition of the Bott-Thom isomorphism $\KK_G(\T^*_G \Scal)\simeq \KK_G(\T^*_G (\Ecal\setminus\{0\}))$ 
with the morphism $j_*:\KK_G(\T^*_G (\Ecal\setminus\{0\}))\to\KK_G(\T^*_G \Ecal)$ correspond to the push-forward 
map $i_!$ defined in Proposition \ref{push=immersion}. The six term exact sequence (\ref{exact.sequence}) becomes
\begin{equation}\label{exact.sequence.cotangent}
\xymatrix{
\Ko_G(\T^*_G \Scal)\ar[r]^{i_!} & \Ko_G(\T^*_G \Ecal) \ar[r]^{\br} & \Ko_G(\T^*_G M)\ar[d]_{\delta}\\
\Kun_G(\T^*_G M)\ar[u]^{\delta}& \ar[l]^{\br}  \Kun_G(\T^*_G \Ecal) &\ar[l]^{i_!} \Kun_G(\T^*_G \Scal).
  }
\end{equation}

Let  $s_!: \KK_G(\T^*_G  M)\longrightarrow \KK_G(\T^*_G  \Ecal)$ be the push-forward morphism associated to the zero section 
$s:M\croc\Ecal$ (see Proposition \ref{push=immersion}). We have the fundamental
\begin{prop}\label{prop:restriction}
\begin{itemize}
\item The composition $\br \circ \, s_!: \KK_G(\T^*_G  M)\to \KK_G(\T^*_G  M)$ is the map 
$\sigma \longrightarrow \sigma\otimes \wedge^\bullet \overline{\Ecal}$.
\item The composition $s_!\circ\, \br : \KK_G(\T^*_G  \Ecal)\to \KK_G(\T^*_G  \Ecal)$ is defined by 
$\sigma \longrightarrow \sigma\otimes \wedge^\bullet \pi^*\overline{\Ecal}$.
\item We have $\index^G_M(\br(\sigma))= \index^G_\Ecal(\sigma\otimes\wedge^\bullet \pi^*\overline{\Ecal})$
for any  $\sigma\in \Ko_G(\T^*_G \Ecal)$.
\end{itemize}
\end{prop}

\begin{proof} The third point is a consequence of second point. Let us check the first two points. 

We use the notations of the proof of proposition \ref{prop.push.vector.bundle}: we have a principal bundle $P\to M=P/O$ 
and $\Ecal$ coincides as a real vector bundle with $P\times_O E$. Since $\Ecal$ has an invariant complex structure, we can consider the frame bundle $Q\subset P$ formed by the unitary basis of $\Ecal$. Here $E=\Rbb^{2n}=\Cbb^n$. Let $U\subset O$ be the unitary group of $E$. Here the map $s_!$ and $\br$ can be defined with the reduced data $(Q,U)$ through the maps
\begin{eqnarray*}
s'_!:  \KK_{G\times U}(\T^*_{G\times U} Q)&\longrightarrow&  \KK_{G\times U}(\T^*_{G\times U} (Q\times E))\\
\sigma &\longmapsto& \bott(E_\Cbb)\odot_{\rm ext} \sigma \nonumber
\end{eqnarray*}
and $\br': \KK_{G\times U}(\T^*_{G\times U} (Q\times E))\longrightarrow\KK_{G\times U}(\T^*_{G\times U} Q)$. Since $E$ admits a complex structure $J$, the map 
$w\oplus iv\mapsto (w+Jv,w-Jv$) is an isomorphism between $E_\Cbb$ and  the orthogonal sum $E\oplus \overline{E}$. Hence on $E_\Cbb$ the 
Bott morphism $\Clif(w\oplus iv):\wedge^+ E_\Cbb\to\wedge^- E_\Cbb$ 
is  equal to the product of the morphisms $\Clif(w+ Jv):\wedge^+ E\to\wedge^- E$ and 
$\Clif(w- Jv):\wedge^+ \overline{E}\to\wedge^- \overline{E}$. When we restrict the 
Bott symbol $\bott(E_\Cbb)\in \Ko_{U}(\T^* E)$ to the $0$-section, we get 
$$
\left(\wedge^+ E\stackrel{\Clif(w)}{\longrightarrow}\wedge^- E\right)\odot 
\left(\wedge^+ \overline{E}\stackrel{\Clif(w)}{\longrightarrow}\wedge^- \overline{E}\right)
$$
which is equal to the class $\bott(E)\otimes\wedge^\bullet  \overline{E}$ in $\Ko_{U}(E)$. Finally the composition 
$\br' \circ \, s'_!: \KK_{G\times U}(\T^*_{G\times U} Q)\to \KK_{G\times U}(\T^*_{G\times U} Q)$ is equal to the map 
$\sigma \longrightarrow \sigma\otimes \wedge^\bullet \overline{E}$. We get the first point through the isomorphism
$\KK_{G\times U}(\T^*_{G\times U} Q)\simeq\KK_G(\T^*_G  M)$.

\medskip

Let $\sigma\in \KK_{G\times U}(\T_{G\times U}^*(Q\times E))$. For $(x,\xi; v,w)\in\T ^*Q\times \T^* E$, the transversally elliptic symbols 
\begin{eqnarray*}
\sigma(x,\xi;v,w)&\otimes&\wedge^\bullet\overline{E}\\
\sigma(x,\xi;v,w)&\odot&\Clif(v)\\
\sigma(x,\xi;0,w)&\odot&\Clif(v)\\
\br'(\sigma)(x,\xi)&\odot&\Clif(w)\odot\Clif(v)\\
\br'(\sigma)(x,\xi)&\odot&\Clif(w+ Jv)\odot(w-Jv)\\
s'_!&\circ& \br'(\sigma)(x,\xi; v,w)
\end{eqnarray*}
define the same class in $\KK_{G\times U}(\T_{G\times U}^*(Q\times E))$. We have proved that $s'_!\circ \br'(\sigma)=\sigma\otimes \wedge^\bullet\overline{E}$, and we get the second point through the isomorphism
$\KK_{G\times U}(\T^*_{G\times U} Q)\simeq\KK_G(\T^*_G  M)$. 
\end{proof}

%%%%%%%%%%%%%%%%%%%%%%%%%%%%%%%%%%%%%%%%%%%%%%%%%%
%%%%%%%%%%%%%%%%%%%%%%%%%%%%%%%%%%%%%%%%%%%%%%%%%%
\subsection{Restriction to a sub-manifold}\label{sec:restriction-Z}
%%%%%%%%%%%%%%%%%%%%%%%%%%%%%%%%%%%%%%%%%%%%%%%%%%
%%%%%%%%%%%%%%%%%%%%%%%%%%%%%%%%%%%%%%%%%%%%%%%%%%

Let $M$ be a $G$-manifold and let $Z$ be a closed $G$-invariant sub-manifold of $M$. Let us consider the open subset 
$\T^*_G(M\setminus Z)$ of $\T^*_G M$. Its complement is the closed subset $\T^*_G M\vert_Z$. 
Let $\Ncal$ be the normal bundle of $Z$ in $M$. We have $\T^*M\vert_Z=\T^*Z \times \Ncal^*$ and then 
$\T^*_G M\vert_Z = \T^*_G Z\times \Ncal^*$.

\medskip 

We make the following {\bf  hypothesis} : the real vector bundle $\Ncal^*\to Z$ has a $G$-equivariant complex structure. 
Then we can define the map 
\begin{equation}\label{eq:morphism-R-Z}
\br_Z : \KK_G(\T^*_G M)\longrightarrow \KK_G(\T^*_G Z)
\end{equation}
as the composition of the restriction $\KK_G(\T^*_G M)\to \KK(\T^*_G M\vert_Z)=$  \break $\KK_G(\T^*_G Z\times \Ncal^*)$ with the Bott-Thom isomorphism $\KK_G(\T^*_G Z\times \Ncal^*)\to \KK_G(\T^*_G Z)$.

%%%%%%%%%%%%%%%%%%%%%%%%%%%%%%%%%%%%%%%%%%%%%%%%%%%%%
%%%%%%%%%%%%%%%%%%%%%%%%%%%%%%%%%%%%%%%%%%%%%%%%%%%%%
%%%%%%%%%%%%%%%%%%%%%%%%%%%%%%%%%%%%%%%%%%%%%%%%%%%%%
\section{Localization}\label{sec:localization}
%%%%%%%%%%%%%%%%%%%%%%%%%%%%%%%%%%%%%%%%%%%%%%%%%%%%%
%%%%%%%%%%%%%%%%%%%%%%%%%%%%%%%%%%%%%%%%%%%%%%%%%%%%%
%%%%%%%%%%%%%%%%%%%%%%%%%%%%%%%%%%%%%%%%%%%%%%%%%%%%%

In this section, $\beta\in\ggot$ denotes a non-zero $G$-invariant element, and $\pi:\Ecal\to M$ is  a $G$-equivariant {\em hermitian} vector bundle such that 
\begin{equation}\label{E=beta}
\Ecal^\beta=M.
\end{equation}

\begin{rem}\label{rem-J-beta}
Note that (\ref{E=beta}) imposes the existence of a $G$-invariant complex structure on the fibers of $\Ecal$. 
We can take\footnote{Relatively to a $G$-invariant Euclidean metric on $\Ecal$, the linear map $-\Lcal(\beta)^2$ 
is positive definite, hence one can take its square root.} $J_\beta:=\Lcal(\beta)(-\Lcal(\beta)^2)^{\frac{-1}{2}}$, where 
$\Lcal(\beta)$ denotes the linear action on the fibers of 
$\Ecal$. 
\end{rem}

The aim of this section is the following 

\begin{theo}\label{theo.S.beta}
There exists a morphism $\bs_\beta: \KK_G(\T^*_G  M)\longrightarrow \KK_G(\T^*_G  \Ecal)$ satisfying the following properties:
\begin{enumerate}
\item The composition $\br\circ \bs_\beta$ is the identity on $\KK_G(\T^*_G  M)$. 
\item For any $a\in \KK_G(\T^*_G  M)$, we have $\bs_\beta(a)\otimes \wedge^\bullet \pi^*\overline{\Ecal}=s_!(a)$.
\item For any $\sigma\in \Ko_G(\T^*_G  M)$, we have the following equality
$$
\indice^G_{\Ecal}(\bs_\beta(\sigma))=\indice^G_{M}(\sigma\otimes[\wedge^\bullet  \overline{\Ecal}]^{-1}_{\beta})
$$
in $\Rgene(G)$, where $[\wedge^\bullet  \overline{\Ecal}]^{-1}_{\beta}$ is a polarized inverse of  $\wedge^\bullet  \overline{\Ecal}$.
\end{enumerate}
\end{theo}

\medskip

\begin{rem} \label{rem:S-beta-complex-structure}
The maps $\br$ and  $\bs_\beta$ depend on the choice of the $G$-invariant complex structure on $\Ecal$.
\end{rem}

Theorem \ref{theo.S.beta} tells us that (\ref{exact.sequence.cotangent}) breaks in an exact sequence
$$
0\to\KK_G(\T^*_G \Scal)\stackrel{i_!}{\longrightarrow} \KK_G(\T^*_G \Ecal) \stackrel{\br}{\longrightarrow} \KK_G(\T^*_G M)\to 0.
$$
Since $\br\circ \bs_\beta=\br\circ \bs_{-\beta}$ the image of the map $\bs_\beta- \bs_{-\beta}:\KK_G(\T^*_G M)\to\KK_G(\T^*_G \Ecal)$ 
belongs to the image of the push-forward map $i_!:\KK_G(\T^*_G \Scal)\to\KK_G(\T^*_G \Ecal)$.

Let us work now with the complex structure $J_\beta$ on $\Ecal$. We denote $\bs_{\pm\beta}^o$ the corresponding morphism. In Section \ref{proof:theta-beta} we will prove the following

\begin{theo}\label{theo:theta-beta}
There exists a morphism $\teta_\beta: \KK_G(\T^*_G  M)\longrightarrow \KK_G(\T^*_G  \Scal)$ such that
$$
\bs^o_{-\beta}- \bs^o_{\beta}=i_! \circ \teta_\beta.
$$
\end{theo}

%%%%%%%%%%%%%%%%%%%%%%%%%%%%%%%%%%%%%%%%%%%%%%%%%%%%%
%%%%%%%%%%%%%%%%%%%%%%%%%%%%%%%%%%%%%%%%%%%%%%%%%%%%%
\subsection{Atiyah-Singer pushed symbols}\label{sec:pushed}
%%%%%%%%%%%%%%%%%%%%%%%%%%%%%%%%%%%%%%%%%%%%%%%%%%%%%
%%%%%%%%%%%%%%%%%%%%%%%%%%%%%%%%%%%%%%%%%%%%%%%%%%%%%

Let $M$ be a $G$-manifold with an invariant almost complex structure $J$. Then the cotangent bundle $\T^* M$ is canonically equipped with a complex structure, still denoted $J$. The Bott morphism on $\T^* M$ associated to the complex vector bundle $(\T^* M, J)\to M$, is called the Thom symbol of $M$, and is denoted\footnote{When the almost complex structure is understood, we will use the notation $\thom(M)$.}  $\thom(M,J)$. Note that the product by the Thom symbol induces an isomorphism 
$\KK_G(M)\simeq\KK_G(\T^* M)$. 

For any $X\in\ggot$, we denote $X_M(m):=\frac{d}{dt}\vert_0 e^{-t X}\cdot m$ the corresponding vector field on 
$M$. Thanks to an invariant Riemmannian metric on $M$, we define the $1$-form 
$$
\widetilde{X}_M(m)=(X_M(m), - ).
$$

From now on, we take $X=\beta$ a non-zero $G$-invariant element. Then the corresponding $1$-form $\widetilde{\beta}_M$ is $G$-invariant, and we define following Atiyah-Singer the equivariant morphism  
$$
\thom_\beta(M,J)(m,\xi):=\thom(M,J)(m,\xi-\widetilde{\beta}_M(m)), \quad (\xi,m)\in \T^* M.
$$
 We check easily that 
$$
{\rm Support}\left(\thom_\beta(M,J)\right)\bigcap \T_{\Tbb_\beta}^* M= \{(m,0);\ m\in M^\beta\}.
$$
where $\Tbb_\beta=\overline{\exp(\Rbb\beta)}$ is the torus generated by $\beta$. In particular, we get a class \begin{equation}
\thom_\beta(M,J)\in \Ko_G(\T_{\Tbb_\beta}^* M)
\end{equation} 
when $M^\beta$ is compact. 

%%%%%%%%%%%%%%%%%%%%%%%%%%%%%%%%%%%%%%%%%%%%%%%%%%%%%
%%%%%%%%%%%%%%%%%%%%%%%%%%%%%%%%%%%%%%%%%%%%%%%%%%%%%
\subsection{Atiyah-Singer pushed symbols : the linear case}\label{sec:pushed-linear}
%%%%%%%%%%%%%%%%%%%%%%%%%%%%%%%%%%%%%%%%%%%%%%%%%%%%%
%%%%%%%%%%%%%%%%%%%%%%%%%%%%%%%%%%%%%%%%%%%%%%%%%%%%%

Let us consider the case of a $G$-Hermitian  vector space $E$ such that $E^\beta=\{0\}$.  

Let $i_!: \Ko_G(\T^*_{\Tbb_\beta} S)\to \Ko_G(\T^*_{\Tbb_\beta} E)$ be the push-forward morphism associated to 
the inclusion $i: S\croc E$ of the sphere of radius one. Let $\br: \KK_G(\T^*_{\Tbb_\beta} E)\to \KK_G(\{\bullet\})$ 
be the restriction morphism.  Since $\Kun_G(\{\bullet\})=0$, the six term exact sequence 
(\ref{exact.sequence.cotangent}) becomes
\begin{equation}\label{exact.sequence.linear}
0\longrightarrow\Ko_G(\T^*_{\Tbb_\beta} S_E)\stackrel{i_!}{\longrightarrow}  \Ko_G(\T^*_{\Tbb_\beta} E) \stackrel{\br}{\longrightarrow}R(G).
\end{equation}

The pushed Thom symbol on $E$ defines a class $\thom_\beta(E)\in \Ko_{G}(\T_{\Tbb_\beta}^* E)$. 

\begin{prop}\label{prop:thom-beta-linear}
\begin{itemize}
\item We have $\br(\thom_\beta(E))=1$ in $R(G)$.
\item The sequence (\ref{exact.sequence.linear}) breaks down: we have a decomposition
$$
\Ko_G(\T^*_{\Tbb_\beta} E)=\Ko_G(\T^*_{\Tbb_\beta} S_E)\oplus \langle \thom_\beta(E)\rangle,
$$
where $\langle \thom_\beta(E)\rangle$ denotes the free $R(G)$-module generated by $\thom_\beta(E)$.
\end{itemize}
\end{prop}

\begin{proof}
At $(x,\xi)\in\T^*E$ the map $\thom_\beta(E)(x,\xi): \wedge^+E\to \wedge^- E$ is equal to $\Clif(\hat{\xi}-\beta_E(x))$, where $\xi\in E^*\mapsto \hat{\xi}\in E$ is the identification given by the Euclidean structure. We see that the restriction of $\thom_\beta(E)$ to $\T^* _{\Tbb_\beta}E\vert_0= E^*$ is equal to $\bott(E^*)$ and then the first point follows. The second point is a direct consequence of the first one.
\end{proof}

\bigskip

Let $\widehat{\Tbb}_\beta$ the group of characters of the torus $\Tbb_\beta$. The complex $G$-module $E$ decomposes into weight spaces $E=\sum_{\alpha\in \widehat{\Tbb}_\beta} E_\alpha$ where each $E_\alpha=\{v\in E\, \vert\, t\cdot v=t^\alpha v\}$ are $G$-submodules.  We define the $\beta$-positive and negative part of $E$, 
$$
E^{+,\beta}=\sum_{\stackrel{\alpha\in \widehat{\Tbb}_\beta}{\langle \alpha, \beta\rangle>0}} E_\alpha,\qquad 
E^{-,\beta}=\sum_{\stackrel{\alpha\in \widehat{\Tbb}_\beta}{\langle \alpha, \beta\rangle<0}} E_\alpha
$$
and the $\beta$-polarized module $|E|^{\beta}= E^{+,\beta}\oplus \overline{E^{-,\beta}}$. It is important to note that the complex $G$-module $|E|^{\beta}$ is isomorphic to\footnote{With $J_\beta= \Lcal(\beta)(- \Lcal(\beta)^2)^{-1/2}$.} $(E,J_\beta)$, and so it does no depend on the initial complex structure of $E$. 

Let $\widehat{R(G)}$ be the $R(G)$-submodule of $\Rgene(G)$ defined by the torus $\Tbb_\beta$ (see Definition \ref{def:hat.K}). 
Since all the $\widehat{\Tbb}_\beta$-weights in $|E|^{\beta}$ satisfy the condition $\langle \alpha, \beta\rangle>0$, the symmetric space $S^\bullet(|E|^{\beta})$ decomposes as a sum $\sum_{\mu\in \widehat{\Tbb}_\beta}S^\bullet(|E|^{\beta})_\mu$ with 
$S^\bullet(|E|^{\beta})_\mu\in R(G)_\mu$. Hence $S^\bullet(|E|^{\beta})$ defines an element of $\widehat{R(G)}$.

The following computation is done in \cite{Atiyah74}[Lecture 5] (see also \cite{pep-RR}[Section 5.1]).

\begin{prop}\label{prop.index.Thom.beta.E}
We have the following equality in $R^{-\infty}(G)$ :
    \begin{equation}\label{eq.indice.thom.beta.E}
    \indice^{G}_E(\thom_{\beta}(E))=(-1)^{\dim_{\Cbb}E^{+,\beta}}\
    \det(E^{+,\beta})\otimes S^\bullet(|E|^{\beta}),
    \end{equation}
    where $\det(E^{+,\beta})$ is a character of $G$.
\end{prop}

\begin{exam} Let $V=\Cbb$ with the canonical action of $G=S^1$. Let $\beta=\pm 1$ in $\mathrm{Lie}(S^1)=\Rbb$. The class $\thom_{\pm 1}(\Cbb)\in \Ko_{S^1}(\T_{S^1}\Cbb)$ are represented
by the symbols
$$
\Clif(\xi\pm i x):\Cbb\longrightarrow \Cbb,\quad (x,\xi)\in \T^*\Cbb\simeq\Cbb^2.
$$
We have $\indice^{S^1}_\Cbb(\Clif(\xi+i x))=-\sum_{k\geq 1} t^k$, and $\indice^{S^1}_\Cbb(\Clif(\xi-i x))=\sum_{k\leq 0} t^k$ in $\Rgene(S^1)$.
\end{exam}

\begin{rem}\label{rem:thom-J-0-1}
Let $J_k, k=0,1$ be two invariants complex structures on $E$, and let $\thom_{\beta}(E, J_k)$ be the corresponding pushed symbols. There exists an invertible element $\Phi\in R(G)$ such that 
$$
\indice^{G}_E(\thom_{\beta}(E,J_0))=\Phi\cdot\indice^{G}_E(\thom_{\beta}(E,J_1)).
$$
\end{rem}

%%%%%%%%%%%%%%%%%%%%%%%%%%%%%%%%%%%%%%%%%%%%%%%%%%%%%
%%%%%%%%%%%%%%%%%%%%%%%%%%%%%%%%%%%%%%%%%%%%%%%%%%%%%
\subsection{Pushed symbols : functoriality}\label{sec:pushed-functorial}
%%%%%%%%%%%%%%%%%%%%%%%%%%%%%%%%%%%%%%%%%%%%%%%%%%%%%
%%%%%%%%%%%%%%%%%%%%%%%%%%%%%%%%%%%%%%%%%%%%%%%%%%%%%

Suppose now that we have a decomposition $V=W\oplus E$ of $G$-complex vector spaces such that 
$V^\beta=\{0\}$.

\begin{prop}\label{prop:thom-beta-bott}  
In $\Ko_G(\T^*_{G} V)$, we have\footnote{These equalities holds also in $\Ko_G(\T^*_{\Tbb_\beta} V)$.} the equalities
\begin{eqnarray*}
\thom_\beta(V)\otimes\wedge^\bullet_\Cbb \overline{V}&=& \bott(V_\Cbb),\\
\thom_\beta(V)\otimes \wedge^\bullet_\Cbb \overline{E}&= &\thom_\beta(W)\odot\bott(E_\Cbb).
\end{eqnarray*}
\end{prop}

\begin{proof} Note that the first relation is a particular case of the second one when $W=0$. 

A covector $(x,\xi)\in\T^* V$ decomposes in $x=x_W\oplus x_E$, and $\xi=\xi_W\oplus \xi_E$. 
The morphism $\sigma:=\thom_\beta(W)\odot\bott(E_\Cbb)$ defines at $(x,\xi)$ the map
$$
\Clif(\hat{\xi}_W-\gamma_W(x_W))\odot \Clif(x_E + i\hat{\xi}_E)
$$
from $(\wedge W\otimes\wedge E_\Cbb)^+$ to $(\wedge W\otimes\wedge E_\Cbb)^-$.

We have an isomorphism of complex $G$-modules : $E_\Cbb\simeq E\times\overline{E}$. We have two classes 
$\bott(E)$ and $\bott(\overline{E})$ in $\Ko_G(E)$  and $\bott(E_\Cbb)=\bott(E)\odot \bott(\overline{E})$.  At the level of endomorphism on $\wedge E_\Cbb\simeq \wedge E\otimes \wedge \overline{E}$, one has 
\begin{equation}\label{eq:clif}
\Clif(x_E+i\xi_E)=\Clif(\xi_E-J_E x_E)\odot \Clif(\xi_E+J_E x_E)
\end{equation}
where $J_E$ is the complex structure on $E$. We consider the family of maps 
$\sigma_s(x,\xi):(\wedge W\otimes\wedge E\otimes\wedge \overline{E})^+\longrightarrow
(\wedge W\otimes\wedge E\otimes\wedge \overline{E})^-$ 
defined by \break $\Clif(\xi_W-\beta_W(x_W))\odot \Clif(\xi_E- \theta_s(x_E))\odot \Clif(\xi_E+J_E x_E)$
where $\theta_s= (1-s) J_E +s\beta_E$. One checks easily that 
$\mathrm{Support}(\sigma_s)\bigcap \T_G V=\{x_W=x_E=\xi_W=\xi_E=0\}$
for any $s\in [0,1]$. Hence $\sigma=\sigma_0$ is equal to $\sigma_1$ in $\Ko(\T_G^* E)$. Finally we check that 
$\sigma_1(x,\xi)=\Clif(\xi-\beta_V(x))\odot \Clif(\xi_E+J_E x_E)$ can be deformed in
$$
\Clif(\xi-\beta_V(x))\odot \Clif(0)=\thom_\beta(V)\otimes \wedge^\bullet_\Cbb \overline{E},
$$
without changing its class in $\Ko_G(\T_G^* V)$.
\end{proof}

\medskip

\medskip

Since $\indice^G_V(\bott(V_\Cbb))=1$, the first relation of Proposition \ref{prop:thom-beta-bott} gives that 
\begin{equation}
\indice^{G}_V(\thom_{\beta}(V))\cdot \wedge^\bullet \overline{V}=1
\end{equation}
 in $R^{-\infty}(G)$.

\begin{defi}\label{def:wedge.inverse}
Let $V$ be a complex $G$-vector space such that $V^\beta=\{0\}$. We denote 
$[\wedge^\bullet V]^{-1}_\beta\in \Rgene(G)$ the element 
$(-1)^{\dim_{\Cbb}V^{-,\beta}}\  \det(V^{-,\beta})\otimes S^\bullet(|V|^{\beta})$. 
\end{defi}

\medskip

\medskip

We come back to the morphism
\begin{equation}\label{eq:map-R-V-W}
\br: \Ko_G(\T_G^* V)\longrightarrow \Ko_G(\T_G^* W)
\end{equation}
which is the composition of the restriction morphism $\Ko_G(\T_G^* V)\to \Ko_G(\T_G^* W\times E^*)$ with the 
Thom isomorphism $\Ko_G(\T_G^* W\times E^*)\simeq\Ko_G(\T_G^* W)$. We are interested by the image 
of the transversally elliptic symbols $\thom_{\beta}(V)\in \Ko_G(\T_G^* V)$ by the morphism $\br$.

\begin{prop}\label{prop:R-thom-beta} 
We have the following equality in $\Ko_G(\T_G^* W)$
$$
\br\left(\thom_{\beta}(V)\right)=\thom_{\beta}(W).
$$
\end{prop}

\begin{proof}The class $\thom_{\beta}(V)$ are defined by the symbols
$\Clif(\xi- \tilde{\beta}(x)):\wedge^+ V\to\wedge^- V$, for $(x,\xi)\in \T V$. Relatively to the decomposition 
$V=W\oplus E$, we write $x=x_W\oplus x_E$ and $\xi=\xi_W\oplus \xi_E$. If we restrict $\Clif(\xi- \tilde{\beta}(x))$ 
to $\T^* V\vert_W=\T^* W\times E^*$ we get $\Clif(\xi_W- \tilde{\beta}(x_W))\odot \Clif(\xi_E)$ 
acting from $\left(\wedge W\otimes\wedge E\right)^+$ to $\left(\wedge W\otimes\wedge E\right)^-$. 
By definition of the map $\br$ we find that $\br\left(\thom_{\beta}(V)\right)=\thom_{\beta}(W)$. 
\end{proof}

\medskip

\medskip

We consider now the case of a product of pushed symbols. Suppose that we have an invariant 
decomposition $E=E_1\oplus E_2$ and invariant elements $\beta_1,\beta_2\in\ggot$ such that 
\begin{itemize}
\item $E_1^{\beta_1}=E_2^{\beta_2}=\{0\}$, 
\item $\beta_2$ acts trivially on $E_1$.
\end{itemize}

We consider then $\beta^t=t\beta_1+ \beta_2$ with $t>0$. We have 
$V_1^{\beta^t}=\{0\}$ for any $t>0$ and $V_2^{\beta^t}=\{0\}$  if $t>0$ is small enough.

\begin{lem}\label{lem:produit-thom-beta}
Let $J=J_1\oplus J_2$ be an invariant complex structure on $V=V_1\oplus V_2$. Then if $t>0$ is small enough, we have the following equality in $\Ko_G(\T_G^* V)$:
$$
\thom_{\beta^t}(V, J)=\thom_{\beta_1}(V_1, J_1)\odot\thom_{\beta_2}(V_2, J_2).
$$
\end{lem}

\begin{proof}Both symbols are maps from $(\wedge V_1\otimes \wedge V_2)^+$ into $(\wedge V_1\otimes \wedge V_2)^-$. 
We write a tangent vector $(\xi, x)\in \T V$ as $\xi=\xi_1\oplus\xi_2$ and $x=x_1\oplus x_2$. The symbol 
$\thom_{\beta^t}(V, J)$ is equal to 
$$
\Clif(\xi_1+ \widetilde{\beta^t}(x_1))\odot\Clif(\xi_2+ \widetilde{\beta^t}(x_2))= 
\Clif(\xi_1+ t\tilde{\beta_1}(x_1))\odot\Clif(\xi_2+ (t\tilde{\beta_1}+\tilde{\beta_2})(x_2))
$$

Note that $\tilde{\beta_2}: V_2\to V_2$ is invertible, so there exist $c>0$ such that $t\tilde{\beta_1}+\tilde{\beta_2}$ 
is invertible for any $t\in [0,c]$. Hence $\thom_{\beta^t}(V, J)$ is transversally elliptic for $0<t\leq c$.
We consider the deformation 
$$
\sigma_s=\Clif(\xi_1+ (st+(1-s))\tilde{\beta_1}(x_1))\odot\Clif(\xi_2+ (st\tilde{\beta_1}+\tilde{\beta_2})(x_2))
$$
for $s\in[0,1]$. We check easily that $\mathrm{Support}(\sigma_s)\cap \T_G V=\{(0,0)\}$ for any $s\in[0,1]$. Hence 
$\sigma_1=\thom_{\beta^t}(V, J)$ and $\sigma_0=\thom_{\beta_1}(V_1, J_1)\odot\thom_{\beta_2}(V_2, J_2)$ 
defines the same class in $\Ko_G(\T_G^* V)$.
\end{proof}

\medskip

%%%%%%%%%%%%%%%%%%%%%%%%%%%%%%%%%%%%%%%%%%%%%%%%%%%%%
%%%%%%%%%%%%%%%%%%%%%%%%%%%%%%%%%%%%%%%%%%%%%%%%%%%%%
\subsection{The map $\bs_\beta$}\label{sec:S-beta}
%%%%%%%%%%%%%%%%%%%%%%%%%%%%%%%%%%%%%%%%%%%%%%%%%%%%%
%%%%%%%%%%%%%%%%%%%%%%%%%%%%%%%%%%%%%%%%%%%%%%%%%%%%%

We come back to the situation of a $G$-equivariant  complex vector bundle $\pi:\Ecal\to M$ such that  $\Ecal^\beta=M$. Since the torus $\Tbb_\beta$ acts trivially on $M$, we have a decomposition $\Ecal=\oplus_{\alpha\in \Xcal}\Ecal_\alpha$ where $\Xcal$ is a finite set of character of  $\Tbb_\beta$, and $\Ecal_\alpha$ is the complex sub-bundle of $\Ecal$ where $\Tbb_\beta$ acts trough the character $t\mapsto t^\alpha$. Definition \ref{def:wedge.inverse} can be extended as follows. We denote 
\begin{equation}\label{eq.inverse.wedge}
[\wedge^\bullet \Ecal]^{-1}_\beta=(-1)^{\dim_{\Cbb}\Ecal^{-,\beta}}\  \det(\Ecal^{-,\beta})\otimes S^\bullet(|\Ecal|^{\beta}). 
\end{equation}
where $\Ecal^{\pm,\beta}=\sum_{\pm\langle \alpha, \beta\rangle>0} \Ecal_\alpha$ and $|\Ecal|^{\beta}=\Ecal^{+,\beta}\oplus \overline{\Ecal^{-,\beta}}$. Note that $[\wedge^\bullet \Ecal]^{-1}_\beta$ belongs to $\widehat{\Ko_G}(M)$ (see Definition \ref{def:hat.K}).

Let $n_\alpha$ be the complex rank of $\Ecal_\alpha$, and let $E$ be the following $\Tbb_\beta$-complex vector space 
$$
E=\bigoplus_{\alpha\in \Xcal}(\Cbb_\alpha)^{n_\alpha},
$$
which is equipped with the standard Hermitian structure.

Let $U$ be the unitary group of $E$, and let $U'$ be the subgroup of elements that commute with the action of $\Tbb_\beta$: we have $U'\simeq \Pi_{\alpha\in\Xcal}U(\Cbb^{n_\alpha})$. Let $P'\to M$ be the $U'$-principal bundle defined as follows: for $m\in M$, the fiber $P'_m$ is defined as the set of maps $f: E\to \Ecal_m$ preserving the Hermitian structures and which are $\Tbb_\beta$-equivariant. By definition, the bundle $P'\to M$ is $G$-equivariant. We consider the following groups action: 
\begin{itemize}
\item $G\times U'$ acts on $P'$,
\item $U'\times \Tbb_\beta$ acts on $E$,
\item $\Tbb_\beta$ and $G$ acts trivially respectively on $P'$ and on $E$.
\end{itemize}

Let us use the multiplicative property (see Section \ref{sec:product}) with the groups $G_2=G\times U', G_1=\Tbb_\beta$ and the manifolds
 $M_1=E, M_2= P'$. We have a product
 $$
 \Ko_{\Tbb_\beta\times G\times U' }(\T^*_{\Tbb_\beta} E)\times \KK_{G\times U' }(\T^*_{G\times U' } P')\longrightarrow  
 \KK_{\Tbb_\beta\times G\times U'}(\T^*_{\Tbb_\beta \times G\times U'} (P'\times E)), 
 $$
and the Thom class $\thom_\beta(E)\in \Ko_{\Tbb_\beta\times U' }(\T^*_{\Tbb_\beta} E)$ induces the map 
 \begin{eqnarray}\label{thom.transversally}
\bs_\beta'':  \KK_{G\times U'}(\T^*_{G\times U'} P)&\longrightarrow&  
\KK_{\Tbb_\beta\times G\times U'}(\T^*_{\Tbb_\beta\times G\times U'} (P'\times E))\\
\sigma &\longmapsto& \thom_\beta(E)\odot_{\rm ext} \sigma \nonumber
\end{eqnarray} 
After taking the quotient by $U'$, we get a map 
$$
\bs_\beta':  \KK_{G}(\T^*_{G} M)\longrightarrow
\KK_{\Tbb_\beta\times G}(\T^*_{\Tbb_\beta\times G} \Ecal)
$$
Finally, since $\T^*_{\Tbb_\beta\times G} \Ecal=\T^*_{G} \Ecal$, we can compose $\bs_\beta'$ with the forgetful map \break 
$\KK_{\Tbb_\beta\times G}(\T^*_{G} \Ecal)\to \KK_{G}(\T^*_{G} \Ecal)$ to get 
$$
\bs_\beta:  \KK_{G}(\T^*_{G} M)\longrightarrow
\KK_{G}(\T^*_{G} \Ecal).
$$

\bigskip

Now we see that in Theorem \ref{theo.S.beta} :
\begin{itemize}
\item The relation $\br\circ \bs_\beta={\rm Id}$ is induced by the relation $\br(\thom_\beta(E))=1$, where 
$\br:\Ko_{\Tbb_\beta\times U' }(\T^*_{\Tbb_\beta} E)\to R(\Tbb_\beta\times U')$ (see Proposition \ref{prop:thom-beta-linear}).
\item The relation $\bs_\beta(a)\otimes \wedge^\bullet \pi^*\overline{\Ecal}=s_!(a)$ is induced by the relation 
$\thom_\beta(E)\otimes \wedge^\bullet \overline{E}=\bott(E_\Cbb)$ proved in Proposition \ref{prop:thom-beta-bott}.
\end{itemize}

Let us prove the last point of Theorem \ref{theo.S.beta}. Let $\sigma\in \Ko_{G}(\T^*_{G} M)$ and let $\widetilde{\sigma}$ be the corresponding element in $\Ko_{G\times U'}(\T^*_{G\times U'} P)$. The index $\indice^G_{\Ecal}(\bs_\beta(\sigma))\in \Rgene(G)$ is equal to the restriction of $\indice^{G\times\Tbb_\beta}_{\Ecal}(\bs'_\beta(\sigma))\in \Rgene(G\times\Tbb)$ at $t=1\in\Tbb_\beta$ (see Section \ref{sec:index}). By definition we have the following equalities in $\Rgene (G\times\Tbb_\beta)$
\begin{eqnarray*}
\indice^{G\times\Tbb_\beta}_{\Ecal}(\bs'_\beta(\sigma))&=&
\left[\indice^{U'\times G\times\Tbb_\beta}_{P'\times E}(\bs''_\beta(\widetilde{\sigma}))\right]^{U'}\\
&=&\left[\indice^{U'\times G}_{P'}(\widetilde{\sigma})\cdot\indice^{U'\times \Tbb_\beta}_{E}(\thom_\beta(E))\right]^{U'}\\
&=&\sum_{\mu\in\widehat{\Tbb}_\beta}\indice^{G}_M(\sigma\otimes \Wcal_\mu)\otimes \Cbb_\mu
\end{eqnarray*}
where $\indice^{U'\times \Tbb_\beta}_{E}(\thom_\beta(E))=[\wedge^\bullet \overline{E}]^{-1}_\beta =\sum_{\mu\in\widehat{\Tbb}} W_\mu\otimes \Cbb_\mu$ with $W_\mu\in R(U')$. 
We denote $\Wcal_\mu=P'\times_{U'} W_\mu$ the corresponding element in $\Ko_G(M)_\mu$. Finally we get 
\begin{eqnarray*}
\indice^{G}_{\Ecal}(\bs_\beta(\sigma))&=&\sum_{\mu\in\widehat{\Tbb}_\beta}\indice^{G}_M(\sigma\otimes \Wcal_\mu)\\
&=&\indice^{G}_M\left(\sigma\otimes[\wedge^\bullet \overline{\Ecal}]^{-1}_\beta\right),
\end{eqnarray*}
where $[\wedge^\bullet \overline{\Ecal}]^{-1}_\beta=\sum_{\mu\in\widehat{\Tbb}} \Wcal_\mu\in \widehat{\Ko_G}(M)$.

%%%%%%%%%%%%%%%%%%%%%%%%%%%%%%%%%%%%%%%%%%%%%%%%%%%%%
%%%%%%%%%%%%%%%%%%%%%%%%%%%%%%%%%%%%%%%%%%%%%%%%%%%%%
\subsection{The map $\theta_\beta$}
%%%%%%%%%%%%%%%%%%%%%%%%%%%%%%%%%%%%%%%%%%%%%%%%%%%%%
%%%%%%%%%%%%%%%%%%%%%%%%%%%%%%%%%%%%%%%%%%%%%%%%%%%%%

We keep the same notation than the previous section: $\pi:\Ecal\to M$ is a $G$-equivariant  complex vector bundle  
such that  $\Ecal^\beta=M$, but here we work with the complex structure $J_\beta$ on $\Ecal$. Since the map 
$\bs^o_{\pm\beta}$ are defined through the pushed Thom classes $\thom_\beta(E)\in \Ko_G(\T^*_{\Tbb_\beta} E)$ 
(see (\ref{thom.transversally})), we have to study the class $\thom_{-\beta}(E)-\thom_{\beta}(E)$ in order 
to understand how the map $\bs^o_{-\beta}-\bs_{\beta}^o:\KK_{G}(\T^*_{G} M)\to \KK_{G}(\T^*_{G} \Ecal)$ 
factorizes through the push-forward morphism $i_!:\KK_{G}(\T^*_{G} \Scal)\to \KK_{G}(\T^*_{G} \Ecal)$.

%%%%%%%%%%%%%%%%%%%%%%%%%%%%%%%%%%%%%%%%%%%%%%%%%%%%%
\subsubsection{The tangential Cauchy Riemann operator}
%%%%%%%%%%%%%%%%%%%%%%%%%%%%%%%%%%%%%%%%%%%%%%%%%%%%%

Let $E$ be a Euclidean $G$-module such that $E^\beta=\{0\}$. We equipped $E$ with the invariant complex 
structure $J_\beta$ (see Remark \ref{rem-J-beta}). Let $S\subset E$ be the sphere of radius one. Let us defined the tangential Cauchy 
Riemann operator on $S$. For $y\in S$, we have 
\begin{eqnarray*}
\T_y S&=&\{\xi\ \vert\ (\xi,y)=0\}\\
&=&\Hcal_y\oplus \Rbb J_\beta y,
\end{eqnarray*}
where $\Hcal_y=(\Cbb y)^\perp$ is a complex invariant subspace of $(E,J_\beta)$. Let 
$\Hcal\to S$ be the corresponding Hermitian vector bundle. For $\xi\in\T_y S $, we denote $\xi'$ its component in $\Hcal_y$. 
Since $(\beta_E(y),J_\beta y)\neq 0$ for $y\neq 0$, we see that for $\xi\in\T_G S\vert_y$, we have 
$\xi'=0\Leftrightarrow\xi=0$.

\begin{defi}\label{def:sigmad} The Cauchy Riemann symbol\footnote{Here we use an identification $\T^* S\simeq \T S$ 
given by the Euclidean structure.} $\sigmad^{E}: \wedge^+ \Hcal\to \wedge^-\Hcal$ is defined by 
$\sigmad^{E}(y,\xi)=\Clif(\xi'): \wedge^+ \Hcal_y\to \wedge^-\Hcal_y$. It defines\footnote{Note that $\sigmad^{E}$ 
defines also a class in $\Ko_G(\T_{\Tbb_\beta}^* S)$.} a class $\sigmad^{E}\in \Ko_G(\T_{G}^* S)$.
\end{defi}

\medskip 

The Thom isomorphism tells us that $\Ko_G(\T_G^* S)\simeq \Ko_G(\T_G^* (E\setminus\{0\}))$ and we know that 
$i_!: \Ko_G(\T_G^* S)\longrightarrow \Ko_G(\T_G^* E)$ is injective. Hence, it will be convenient to use the same notations 
for $\sigmad^{E}$ and $i_!(\sigmad^E)$ and to consider them as a class in $\Ko_G(\T_G^* (E\setminus\{0\}))$ 
or in $\Ko_G(\T_G^* E)$.

\medskip 

\begin{exam} Consider the Cauchy Riemann symbol $\sigmad^{\Cbb_\chi}\in\Ko_G(\T_{G}^*\Cbb_\chi)$ associated to the one dimensional representation $\Cbb_\chi$ of $G$. We check that $\sigmad^{\Cbb_\chi}$ is represented by the map $\rho:\T^*\Cbb_\chi\to\Cbb$ defined by:
$\rho(w,z)=\Re(w\bar{z}) + \imath (\|z\|-1)$.
\end{exam}

\medskip

We come back to the setting of Section \ref{sec:pushed}. We have an exact sequence 
$0\to\KK_G(\T^*_{\Tbb_\beta} S)\stackrel{i_!}{\longrightarrow} \KK_G(\T^*_{\Tbb_\beta}E) 
\stackrel{\br}{\longrightarrow} R(G)\to 0$, and we know that $\br(\thom_{\pm\beta}(E))= 1$. Then $\thom_\beta(E)-\thom_{-\beta}(E)$ belongs to $\ker(\br)={\rm Im}(i_!)$.

The following result is due to Atiyah-Singer when $G$ is the circle group (see \cite{Atiyah74}[Lemma 6.3]). The proof in the general case is given in Appendix B. 

\begin{prop}\label{prop:thom-pm} 
Let $E$ be a $G$-module equipped with the invariant complex structure $J_\beta$. 
We have the following equality 
$$
\thom_{-\beta}(E)-\thom_{\beta}(E)= i_!(\sigmad^{E}).
$$
in $\Ko_G(\T_G^* E)$.
\end{prop}

%%%%%%%%%%%%%%%%%%%%%%%%%%%%%%%%%%%%%%%%%%%%%%%%%%%%%
\subsubsection{Functoriality}
%%%%%%%%%%%%%%%%%%%%%%%%%%%%%%%%%%%%%%%%%%%%%%%%%%%%%

Suppose that $V=W\oplus E$ with $W^\beta= E^\beta=\{0\}$. We equipped $V,W,E$ 
by the invariant complex structures  defined by $\beta$. Let $\sigmad^{V}\in \Ko_G(\T^*_G (V\setminus\{0\}))$, $\sigmad^{W}\in \Ko_G(\T_G^* (W\setminus\{0\}))$ be the corresponding Cauchy Riemann classes. We have a natural product 
$$
\Ko_G(\T_G^* (W\setminus\{0\}))\times \Ko_G(\T^* E)\longrightarrow \Ko_G(\T_G^* (V\setminus\{0\})).
$$
and a restriction morphism $\br: \Ko_G(\T_G^* V)\longrightarrow \Ko_G(\T_G^* W)$ (see (\ref{eq:map-R-V-W})).

\begin{prop}\label{prop:sigmad-bott} We have
\begin{itemize}
\item $\sigmad^{W,\beta}\otimes \bott(E_\Cbb)= \sigmad^{V,\beta}\otimes \wedge^\bullet \overline{E}$\quad  in\  $\Ko_G(\T_G^* (V\setminus\{0\}))$,
\item  $\br(\sigmad^{V,\beta})=\sigmad^{W,\beta}$\quad in\  $\Ko_G(\T_G^* W)$.
\end{itemize}
\end{prop}

\begin{proof} These are direct consequences of Proposition \ref{prop:thom-pm}. For the first point, we use it together 
with Proposition \ref{prop:thom-beta-bott}, and for the second one we use it together with Proposition \ref{prop:R-thom-beta}. 
\end{proof}

\medskip

The element $\sigmad^{V,\beta}$ belongs to the subspace $\Ko_G(\T_G^* (V\setminus\{0\}))\croc \Ko_G(\T_G^* V)$, and  
the restriction map $\br$ sends $\Ko_G(\T_G^* (V\setminus\{0\}))$ into $\Ko_G(\T_G^* (W\setminus\{0\}))$ (see remark \ref{rem-JR}). We can precise the last statement of Proposition \ref{prop:sigmad-bott}, by saying that the equality 
$\br(\sigmad^{V,\beta})=\sigmad^{W,\beta}$ holds in $\Ko_G(\T_G^* (W\setminus\{0\}))$.

%%%%%%%%%%%%%%%%%%%%%%%%%%%%%%%%%%%%%%%%%%%%%%%%%%%%%%%%%%%%%%%%%%
\subsubsection{Definition of the map $\theta_\beta$}\label{proof:theta-beta}
%%%%%%%%%%%%%%%%%%%%%%%%%%%%%%%%%%%%%%%%%%%%%%%%%%%%%%%%%%%%%%%%%%

We come back to the setting of Section \ref{sec:S-beta}. The complex vector bundle $\Ecal\to M$ corresponds to $P'\times_{U'}E\to P'/U'$, and the  sphere bundle is $\Scal=P'\times_{U'} S_E$.

Let us use the multiplicative property (see Section \ref{sec:product}) with the groups $G_2=G\times U', G_1=\Tbb_\beta$ and the manifolds
 $M_1=S_E, M_2= P'$. Thanks to the product 
 $$
 \Ko_{\Tbb_\beta\times G\times U' }(\T^*_{\Tbb_\beta} S_E)\times \KK_{G\times U' }(\T^*_{G\times U' } P')\longrightarrow  
 \KK_{\Tbb_\beta\times G\times U'}(\T^*_{\Tbb_\beta \times G\times U'} (P'\times S_E))
 $$
 we can define
 \begin{eqnarray}\label{thom.transversally.sigma}
\theta_\beta'':  \KK_{G\times U'}(\T^*_{G\times U'} P')&\longrightarrow&  
\KK_{\Tbb_\beta\times G\times U'}(\T^*_{\Tbb_\beta\times G\times U'} (P'\times S_E))\\
\sigma &\longmapsto& \sigmad^{E,\beta}\odot_{\rm ext} \sigma. \nonumber
\end{eqnarray} 
After taking the quotient by $U'$, we get a map 
$$
\theta_\beta':  \KK_{G}(\T^*_{G} M)\longrightarrow
\KK_{\Tbb_\beta\times G}(\T^*_{\Tbb_\beta\times G} \Scal)
$$
Finally, since $\T^*_{\Tbb_\beta\times G} \Scal=\T^*_{G} \Scal$, we can compose $\theta_\beta'$ with the forgetful map \break 
$\KK_{\Tbb_\beta\times G}(\T^*_{G} \Scal)\to \KK_{G}(\T^*_{G} \Scal)$ to get $\theta_\beta:  \KK_{G}(\T^*_{G} M)\longrightarrow
\KK_{G}(\T^*_{G} \Scal)$.

The identity $\thom_{-\beta}(E)-\thom_{\beta}(E)= i_!(\sigmad^{E,\beta})$ shows that we have a commutative diagram
$$
\xymatrix@C=2cm{
\KK_{G\times U' }(\T^*_{G\times U' } P') \ar[r]^{\theta_\beta''} \ar[dr]_{\bs_{-\beta}''-\bs_\beta''} &
 \KK_{\Tbb_\beta\times G\times U'}(\T^*_{\Tbb_\beta \times G\times U'} (P'\times S_E) \ar[d]^{i_!}\\
     & \KK_{\Tbb_\beta\times G\times U'}(\T^*_{\Tbb_\beta \times G\times U'} (P'\times E))\ .
}
$$
After taking the quotient by $U'$, we get  the commutative diagram
$$
\xymatrix@C=2cm{
\KK_{G }(\T^*_{G} M) \ar[r]^{\theta_\beta} \ar[dr]_{\bs^o_{-\beta}-\bs^o_\beta} &
 \KK_{G}(\T^*_{G} \Scal) \ar[d]^{i_!}\\
     & \KK_{G}(\T^*_{G} \Ecal)\ .
}
$$
which is the content of Theorem \ref{theo:theta-beta}.

%%%%%%%%%%%%%%%%%%%%%%%%%%%%%%%%%%%%%%%%%%%%%%%%%%%%%%%%%%%
%%%%%%%%%%%%%%%%%%%%%%%%%%%%%%%%%%%%%%%%%%%%%%%%%%%%%%%%%%%
\subsection{Restriction to a fixed point sub-manifold}\label{sec:restriction-Z-beta}
%%%%%%%%%%%%%%%%%%%%%%%%%%%%%%%%%%%%%%%%%%%%%%%%%%%%%%%%%%%%
%%%%%%%%%%%%%%%%%%%%%%%%%%%%%%%%%%%%%%%%%%%%%%%%%%%%%%%%%%%

Let $M$ be a $G$-manifold and let $\beta\in\ggot$ be a $G$-invariant element. 
Let $Z$ be a connected component of the fixed point set $M^\beta$. Note that $\beta$ 
defines a complex structure $J_\beta$ on the normal bundle of $Z$ in $M$. Following Section \ref{sec:restriction-Z} we have a restriction morphism $\br_Z$ that fits in the six term exact sequence 
$$
\xymatrix{
\Ko_G(\T^*_G (M\setminus Z))\ar[r]^{j_*} & \Ko_G(\T^*_G M) \ar[r]^{\br_Z} & \Ko_G(\T^*_G Z)\ar[d]_{\delta}\\
\Kun_G(\T^*_G Z)\ar[u]^{\delta}& \ar[l]^{\br_Z}  \Kun_G(\T^*_G M) &\ar[l]^{j_*} \Kun_G(\T^*_G (M\setminus Z)).
  }
$$

\begin{prop}
\begin{itemize}
\item There exists a morphism $\bs_{\beta,Z}: \KK_G(\T^*_G Z)\to \KK_G(\T^*_G M)$ such that $\br_Z\circ \bs_{\beta,Z}$ is the identity on $\KK_G(\T^*_G Z)$.
\item We have an isomorphism of $R(G)$-modules :
$$
\KK_G(\T^*_G M)\simeq \KK_G(\T^*_G Z)\oplus \KK_G(\T^*_G (M\setminus Z)).
$$
\end{itemize}
\end{prop}
\begin{proof}
Let $\Ncal$ be the normal bundle of $Z$ in $M$. Let $\Ucal$ be an invariant tubular neighborhood of $Z$, which is small enough so that we have an equivariant diffeomorphism $\phi: \Ucal\to \Ncal$ which is the identity on $Z$. Let $\bs_{\beta,\Ncal}: \KK_G(\T^*_G Z)\to \KK_G(\T^*_G \Ncal)$ the map that we have constructed in Section \ref{sec:S-beta}. Let $j_*:\KK_G(\T^*_G \Ucal)\to \KK_G(\T^*_G M)$
be the push-forward map associated to the inclusion $j:\Ucal\croc M$. Let $\phi^*: :\KK_G(\T^*_G \Ncal)\to \KK_G(\T^*_G \Ucal)$ be the isomorphism associated to $\phi$. We can consider the composition
$$
\bs_{\beta,Z}:=j_*\circ\phi^*\circ\bs_{\beta,\Ncal},
$$
and we leave to the reader the verification that $\br_Z\circ \bs_{\beta,Z}={\rm Id}$. The last point is a direct consequence of the first one.
\end{proof}

%%%%%%%%%%%%%%%%%%%%%%%%%%%%%%%%%%%%%%%%%%%%%%%%%%%%%
%%%%%%%%%%%%%%%%%%%%%%%%%%%%%%%%%%%%%%%%%%%%%%%%%%%%%
%%%%%%%%%%%%%%%%%%%%%%%%%%%%%%%%%%%%%%%%%%%%%%%%%%%%%
\section{Decomposition of $\KK_G(\T^*_G M)$ when $G$ is abelian}\label{sec:decomposition}
%%%%%%%%%%%%%%%%%%%%%%%%%%%%%%%%%%%%%%%%%%%%%%%%%%%%%
%%%%%%%%%%%%%%%%%%%%%%%%%%%%%%%%%%%%%%%%%%%%%%%%%%%%%
%%%%%%%%%%%%%%%%%%%%%%%%%%%%%%%%%%%%%%%%%%%%%%%%%%%%%

In this section $G$ denotes a compact {\bf abelian} Lie group, with Lie algebra $\ggot$. Let $M$ be a (connected) manifold equipped with an action of $G$. For any $m\in M$, we denote $\ggot_m\subset \ggot$ its infinitesimal stabilizer.  

Let $\Delta_G(M)$ be the set formed by the infinitesimal stabilizer of points in $M$. During this section, we suppose that $\Delta_G(M)$ is {\bf finite}: it is the case if $M$ is compact or when $M$ is embedded equivariantly in a $G$-module. We have a partition 
$$
M=\bigsqcup_{\hgot\in\Delta_G(M)} M_\hgot
$$
where $M_\hgot:=\{m\in M\ \vert\ \hgot=\ggot_m\}$ is an invariant open subset of the smooth sub-manifold 
$M^\hgot:=\{m\in M\ \vert\ \hgot\subset\ggot_m\}$.

\medskip

On the other hand, we consider for $0\leq k\leq s=\dim G$ the {\em closed} subset 
$$
M^{\leq k}\subset M
$$ 
formed by the points $m\in M$ such that $\dim (G\cdot m)=\codim(\ggot_m)\leq k$. We 
have 
$$
M^{\leq k}=\bigsqcup_{\codim\hgot\leq k}M_\hgot= \bigcup_{\codim\hgot\leq k}M^\hgot
$$
Let $M^{=k}=M^{\leq k}\setminus M^{\leq k-1}$ and $M^{> k}=M\setminus M^{\leq k-1}$. We note that
$$
M^{=k}=\bigsqcup_{\codim\hgot = k} M_\hgot
$$

Let $s_o$ be the maximal dimension of the $G$-orbit in $M$. We will use the increasing
sequence of invariant open subsets 
$$
M^{>s_o-1}\subset\cdots\subset M^{>1}\subset M^{>0}\subset M.
$$
Here $M^{>0}=M\setminus M^\ggot$, and $M^{>s_o-1}=M^{gen}$ is the dense open subset formed by the $G$-orbits of maximal dimension.  Note also that $M^{gen}$ corresponds to $M_{\hgot_{min}}$ where $\hgot_{min}$ is the minimal stabilizer.

Let us consider the related sequences of open subspaces 
$$
\T_G^* M^{>s_o-1}\subset\cdots\subset \T_G^*  M^{>1}\subset \T_G^*  M^{>0}\subset \T_G^* M.
$$
At level of $K$-theory the inclusion $j_k:  M^{>k}\croc M^{>k-1}$ gives rise to the map
$$
(j_k)_*: \KK_G(\T_G M^{>k})\longrightarrow \KK_G(\T_G^* M^{>k-1}).
$$

Let $0\leq k\leq s_o-1$. We have the decomposition
\begin{eqnarray*}
\T_G^*  M^{>k-1}&=&\T_G  M^{>k}\bigsqcup \T_G^*  M^{>k-1}\vert_{M^{=k}}\\
&=&\T_G^*  M^{>k}\bigsqcup \bigsqcup_{\codim\hgot=k} \T_G^*  M^{>k-1}\vert_{M_\hgot}\\
&=&  \T_G^*  M^{>k}\bigsqcup \bigsqcup_{\codim\hgot=k} \T_G^*  M_\hgot\times \Ncal_\hgot^*
\end{eqnarray*}
where $\Ncal_\hgot$ is the normal bundle of $M_\hgot$ in $M$. Note that $M_\hgot$ is a closed sub-manifold of the open subset
$M^{>k-1}$, when $\codim\hgot=k$. 

\begin{lem} \label{lem-gamma-hgot} Let $\hgot\in\Delta_G(M)$ with $\codim\hgot=k$. There exists $\gamma_\hgot\in\hgot$ so that $M_\hgot$ is equal to the fixed point set $(M^{>k-1})^{\gamma_\hgot}:=\left\{m\in M^{>k-1}\ \vert\  \gamma_\hgot\in\ggot_m  \right\}$. The element $\gamma_\hgot$ defines then a complex structure $J_{\gamma_\hgot}$ on the normal bundle $\Ncal_\hgot$.
\end{lem}

\begin{proof} Let $H$ be the closed connected subgroup of $G$ with Lie algebra $\hgot$. Let $\gamma_\hgot\in\hgot$ generic so that 
the closure of $\{\exp(t\gamma_\hgot), t\in \Rbb\}$ is equal to $H$. Then for any $m\in M$, $\gamma_\hgot\in\ggot_m\Leftrightarrow
\hgot\subset\ggot_m$. Then 
$$
\left\{m\in M^{>k-1}\, \vert\,  \gamma_\hgot\in\ggot_m  \right\}= \left\{m\in M^{>k-1}\, \vert\,  \hgot\subset\ggot_m  \right\}
=\left\{m\in M\, \vert\,  \hgot=\ggot_m  \right\}= M_\hgot.
$$
\end{proof}

Thanks to Lemma \ref{lem-gamma-hgot}, we can exploit Section \ref{sec:restriction-Z-beta}. For any 
$\hgot\in\Delta_G(M)$ of codimension $k$, we have a restriction morphism 
\begin{equation}\label{eq:restriction-hgot}
\br_\hgot: \KK_G(\T_G^* M^{>k-1})\longrightarrow \KK_G(\T_G^* M_\hgot)
\end{equation}
and a section
$$
\bs_\hgot:= \bs_{\gamma_\hgot,M_\hgot} : \KK_G(\T_G^* M_\hgot)\longrightarrow \KK_G(\T_G^* M^{>k-1})
$$
such that $\br_\hgot\circ \bs_\hgot$ is the identity on $\KK_G(\T_G^* M_\hgot)$.

We have also a long exact sequence 
$$
\xymatrix@C=12mm{ 
\Ko_{G}(\T_G^* M^{>k}) \ar[r]^{(j_k)_*}&
\Ko(\T_G^* M^{>k-1})\ar[r]^{\br_k\qquad }&
\bigoplus_{\codim\hgot=k} \Ko_G(\T_G^* M_\hgot)\ar[d]^{\delta}\\
\bigoplus_{\codim\hgot=k} \Kun_G(\T_G^* M_\hgot)\ar[u]^{\delta}&
\Kun_G(\T_G^* M^{>k-1})\ar[l]^{\qquad \br_k }&
\Kun_G(\T_G^* M^{>k})\ar[l]^{(j_k)_*}  ,
}
$$
where $\br_k=\oplus_{\codim\hgot=k} \br_\hgot$.  We define $\bs_k:\oplus_{\codim\hgot=k} 
\KK_G(\T_G^* M_\hgot)\longrightarrow \KK_G(\T_G^* M^{>k-1})$ by 
$$
\bs_k(\oplus_{\codim\hgot=k}\sigma_\hgot)=\sum_{\codim\hgot=k}\bs_\hgot(\sigma_\hgot).
$$

\begin{lem} Let $\agot,\bgot\in\Delta_G(M)$.
\begin{itemize}
\item We have $\br_{\agot}\circ\bs_{\agot}={\rm Id}$ in $\KK_G(\T_G^* M_\agot)$.
\item We have $\br_{\agot}\circ\bs_{\bgot}=0$ if $\agot\neq\bgot$. 
\item The map $\br_k\circ\bs_k$ is the identity on $\oplus_{\codim\hgot=k} \KK_G(\T_G^* M_\hgot)$.
\end{itemize}
\end{lem}

\begin{proof}The last point is a direct consequence of the firsts one. The first point is known, 
and the second assertion is due to the fact that $M_\agot \cap M_{\bgot}=\emptyset$ when $\agot\neq \bgot$. 
\end{proof}

\medskip

The previous lemma shows that the map 
\begin{equation}\label{eq:iso-k}
((j_k)_*, \bs_{k}): \KK(\T_G^* M^{>k})\times\oplus_{\codim\hgot=k} 
\KK_G(\T_G^* M_\hgot) \longrightarrow\KK(\T_G^* M^{>k-1})
\end{equation}
is an isomorphism of $R(G)$-module. In particular the maps $(j_k)_*$ are injective.

\medskip

\begin{rem}\label{rem:j-injective}
If we consider the open subset $j:M^{gen}\croc M$ formed by the $G$-orbits of maximal dimension, we know then that 
$$
j_*: \KK_G(\T_G^* M^{gen})\longrightarrow \KK_G(\T_G^* M)
$$
 is injective, since $j$ is the composition of all the $j_k$.
\end{rem}

\bigskip

The isomorphisms (\ref{eq:iso-k}) all together give the following Theorem (which was given in a less precise version in  \cite{Atiyah74}[Theorem 8.4]).

\begin{theo}[Atiyah-Singer]
\label{theo:decomposition-K_G-T_G}
Let $\gamma:=\{\gamma_\hgot,\hgot\in\Delta_G(M)\}$ such that $M_\hgot=\{m\in M^{>\codim\hgot-1}\ \vert\ \gamma_\hgot\in\ggot_m\}$. We have an isomorphism 
\begin{equation}\label{eq:decomposition-K_G-T_G}
\Phi_\gamma: \bigoplus_{\hgot\in\Delta_G(M)}\KK_G(\T_G^* M_\hgot)\longrightarrow\KK_G(\T_G^* M)
\end{equation}
of $R(G)$-module such that 
$$
\indice^G_M\left(\Phi_\gamma(\oplus_\hgot \sigma_\hgot)\right)= 
\sum_{\hgot\in\Delta}\indice^G_{M_\hgot}\left(\sigma_\hgot\otimes S^\bullet(\Ncal_\hgot)\right)
$$
for any $\oplus_\hgot \sigma_\hgot\in \bigoplus_{\hgot\in\Delta}\Ko_G(\T_G^* M_\hgot)$. Here 
$\Ncal_\hgot$ is the normal bundle of $M_\hgot$ in $M$ which is equipped with the complex structure 
defined by $-\gamma_\hgot$
\end{theo}

\bigskip

For any $\hgot\in\Delta_G(M)$ we denote $H\subset G$ the closed connected subgroup with Lie algebra $\hgot$. Let us denote 
$H'\subset G$ be a Lie subgroup such that $G\simeq H\times H'$. Then the $R(G)$-module $\KK_G(\T_G^* M_\hgot)$ is equal to 
$$
\KK_{H'}(\T_{H'}^* M_\hgot)\otimes R(H).
$$
Thus Theorem \ref{theo:decomposition-K_G-T_G} says that $\KK_G(\T_G^* M)$ is isomorphic to 
$$
\bigoplus_{\hgot\in\Delta_G(M)}\KK_{H'}(\T_{H'}^* M_\hgot)\otimes R(H).
$$

Note that the action of $H'$ on $M_\hgot$ has finite stabilizers, hence the group $\KK_{H'}(\T_{H'}^* M_\hgot)$ is equal to 
$\mathbf{K}^*_{orb}(\T^*\Mcal_\hgot)$, where $\Mcal_\hgot=M_\hgot/H'$ is an orbifold.

%%%%%%%%%%%%%%%%%%%%%%%%%%%%%%%%%%%%%%%%%%%%%%%%%%%%%
%%%%%%%%%%%%%%%%%%%%%%%%%%%%%%%%%%%%%%%%%%%%%%%%%%%%%
\section{The linear case}\label{sec:linear-case}
%%%%%%%%%%%%%%%%%%%%%%%%%%%%%%%%%%%%%%%%%%%%%%%%%%%%%
%%%%%%%%%%%%%%%%%%%%%%%%%%%%%%%%%%%%%%%%%%%%%%%%%%%%%

In this section, the group $G$ is a compact {\bf abelian} Lie group. Let $V$ be a real $G$-module. Let $V^{gen}$ be the open subset formed by the $G$-orbits of maximal dimension. We equip $V/V^\ggot$ with an invariant complex structure. For any $\gamma\in\ggot$ such that $V^\gamma=V^\ggot$ we associate the class
$$
\thom_\gamma(V/V^\ggot)\odot\bott(V^\ggot)\in \Ko_G(\T_G^* V).
$$

\medskip 

Let $H_{min}\subset G$ be the minimal stabilizer for the $G$-action on $V$. Let $s:=\dim G -\dim H_{min}$.

\begin{defi}\label{def:flag}
A $(G,V)$-flag $\varphi$  corresponds to a decomposition $V/V^\ggot=V^\varphi_1\oplus\cdots\oplus V^\varphi_s$ in complex $G$-subspaces, and
a decomposition $\ggot=\hgot_{min}\oplus\Rbb\beta^\varphi_1\oplus\cdots\Rbb\beta^\varphi_s$ such that for any $1\leq k\leq s$ 
\begin{itemize}
\item[\bf c1] $\beta^\varphi_k$ acts trivially on  $V^\varphi_j$ when $j<k$,
\item[\bf c2] $\beta^\varphi_k$ acts bijectively\footnote{$\beta$ acts bijectively on a vector space $V$ if $V^\beta=\{0\}$.} on $V^\varphi_k$.
\end{itemize} 
\end{defi}

We can associate to the data $\varphi$ above , the flags 
$V^\ggot=V^{[0],\varphi} \subset  V^{[1],\varphi}\subset \ldots \subset V^{[s],\varphi}=V$ and
$\hgot_{min}=\ggot^{[0],\varphi} \subset \ggot^{[1],\varphi}\subset \ldots \subset \ggot^{[s],\varphi}=\ggot$
where 
$$
V^{[j],\varphi} =V^\ggot\oplus \sum_{1\leq k\leq j} V^\varphi_k,\quad \mathrm{and} \quad 
\ggot^{[j],\varphi} = \hgot_{min}\oplus\Rbb\beta^\varphi_{j+1}\oplus\cdots\oplus\Rbb\beta^\varphi_s.
$$
We see that conditions {\bf c1} and {\bf c2} are equivalent to saying that the generic infinitesimal stabilizer of the $G$-action on the vector space $V^{[j],\varphi}$ is equal to 
$\ggot^{[j],\varphi}$.

\medskip

Thanks to {\bf c2}, the Cauchy-Riemann symbol 
$$
\sigmad^{\varphi,k}\in \Ko_{G}(\T_{\Rbb\beta_k}^*(V^\varphi_k\setminus\{0\})), 
$$
is well defined. Conditions  {\bf c1} and {\bf c2}  tell us also that $(V^\varphi_1\setminus\{0\})\times\cdots\times (V^\varphi_s\setminus\{0\})$ 
is an open subset of $(V/V^\ggot)^{gen}$, and thanks to Theorem \ref{theo:multiplicative-property-abelian} we know that the following product
$$
\sigmad^{V/V^\ggot,\varphi}:=\sigmad^{\varphi,1}\odot\cdots\odot\sigmad^{\varphi,s}
$$
is a well defined class in $\Ko_G(\T_G^* (V/V^\ggot)^{gen})$. 

We need the following submodule of $R^{-\infty}(G)$ defined by the relations
\begin{eqnarray*}
\Phi\in \Fcal_G(V)&\Longleftrightarrow & \wedge^\bullet\overline{V/V^\hgot}\otimes \Phi\in \langle R^{-\infty}(G/H)\rangle
\  , \forall \hgot\in \Delta_G(V),\\
\Phi\in \dm_G(V)&\Longleftrightarrow & \wedge^\bullet\overline{V/V^\hgot}\otimes \Phi=0, \forall \hgot\neq\hgot_{min}
\quad  \mathrm{and}\quad \Phi\in \langle R^{-\infty}(G/H_{min})\rangle.
\end{eqnarray*}

The purpose of this section is to give a detailled proof of the following theorem. 

\begin{theo}\label{theo:generateur} Let $G$ a compact {\em abelian} Lie group and let $V$ be a real $G$-module. We have
\begin{itemize}
\item[\bf a.] $\Kun_G(\T_G^* V)=\Kun_G(\T_G^* V^{gen})=0$
\item[\bf b.] The index map $\indice^G_V:\Ko_G(\T_G^* V)\longrightarrow R^{-\infty}(G)$ is one to one.
\item[\bf c.] The elements $\bott(V^\ggot_\Cbb)\odot\thom_\gamma(V/V^\ggot)$ generate $\Ko_G(\T_G^* V)$, when $\gamma$ runs over the elements such that $V^\gamma=V^\ggot$.
\item[\bf d.] The elements $\bott(V^\ggot_\Cbb)\odot\sigmad^{V/V^\ggot,\varphi}$ generate $\Ko_G(\T_G^* V^{gen})$, when $\varphi$ runs over the $(G,V)$-flag.
\item[\bf e.] The image $\KK_G(\T^*_G V)$ by $\indice^G_V$ is equal to $\Fcal_G(V)$.
\item[\bf f.] The image $\KK_G(\T^*_G V^{gen})$ by $\indice^G_V$ is equal to $\dm_G(V)$.
\end{itemize}
\end{theo}

Hence {\bf b.}, {\bf e.} and {\bf f.} say that the $R(G)$-modules $\KK_G(\T^*_G V)$ and $\KK_G(\T^*_G V^{gen})$ are respectively isomorphic to $\Fcal_G(V)$ and $\dm_G(V)$.

Note that, when $\dim V/V^\ggot=0$, we have $\T_G^* V=\T_G^* V^{gen}=\T^* V$ and all the points are direct consequences of the Bott isomorphism. Point {\bf d.} is proved in \cite{Atiyah74}, and points {\bf a.}, {\bf e.} and {\bf f.} are due to de Concini-Procesi-Vergne \cite{CPV-2010-1,CPV-2010-2}. Point {\bf b.} is proved in \cite{Atiyah74} for the circle group, and in \cite{CPV-2010-1,CPV-2010-2} for the general case. In  \cite{CPV-2010-1,CPV-2010-2}, {\bf c.} is obtained as a consequence of {\bf d.} together with the decomposition formula (\ref{eq:decomposition-K_G-T_G}).

We will give a proof by induction on $\dim V/V^\ggot$ that is based on the work of \cite{CPV-2010-1,CPV-2010-2}. But here our treatment differs from those of \cite{Atiyah74,CPV-2010-1,CPV-2010-2}, since the proof of all points of Theorem \ref{theo:generateur} follows directly by a careful analysis of the exact sequence  
$$
0\longrightarrow \Ko_{G_\chi}(\T_{G_\chi}^*W) \stackrel{\bj}{\longrightarrow}
\Ko_G(\T_G^* V)\stackrel{\br}{\longrightarrow} \Ko_G(\T_G^* W)\longrightarrow 0.
$$
associated to an invariant decomposition $V=\Cbb_\chi\oplus W$.

%%%%%%%%%%%%%%%%%%%%%%%%%%%%%%%%%%%%%%%%%%%%%%%%%%%%%
%%%%%%%%%%%%%%%%%%%%%%%%%%%%%%%%%%%%%%%%%%%%%%%%%%%%%
\subsection{Restriction to a subspace}\label{subsec:restriction}
%%%%%%%%%%%%%%%%%%%%%%%%%%%%%%%%%%%%%%%%%%%%%%%%%%%%%
%%%%%%%%%%%%%%%%%%%%%%%%%%%%%%%%%%%%%%%%%%%%%%%%%%%%%

Suppose that $V\neq V^\ggot$. Then $V$ contains a complex representation $\Cbb_\chi$ attached to a 
{\bf surjective} character $\chi: G\to S^1$. Let $G_\chi=\ker(\chi)$ with Lie algebra $\ggot_\chi$. The differential of $\chi$ is  $i\bar{\chi}$ with $\bar{\chi}\in\ggot^*$. Here 
$\Cbb_\chi\cap V^\ggot=\{0\}$ since  $\bar{\chi}\neq 0$. 

Let us consider an invariant decomposition $V=W\oplus \Cbb_\chi$. 
\begin{rem}\label{rem:induction}
We check  that $\dim W/W^\ggot=\dim V/V^\ggot -1$, and $\dim W/W^{\ggot_\chi}\leq \dim V/V^\ggot -1$.
\end{rem}
We look at the open subset 
$j:\T_G^*(W\times \Cbb_\chi\setminus\{0\})\croc \T_G^* V$. Its complement is the closed subset 
$ \T_G^* V\vert_{W\times \{0\}}\simeq \T_G^* W\times \Cbb_\chi$. We have the six term exact sequence
\begin{equation}\label{eq:six-term}
\xymatrix@C=15mm{
\Ko_G(\T_G^*(W\times \Cbb_\chi\setminus\{0\}))\ar[r]^{j_*} & \Ko_G(\T^*_G V) \ar[r]^{r} & 
\Ko_G(\T_G^* W\times \Cbb_\chi)\ar[d]_{\delta}\\
\Kun_G(\T_G^* W\times \Cbb_\chi)\ar[u]^{\delta}& \ar[l]^{r}  \Kun_G(\T^*_G V) &\ar[l]^{j_*} 
\Kun_G(\T^*_G (W\times \Cbb_\chi\setminus\{0\})).
  }
\end{equation}

Let $\br:\KK_G(\T_G^* V) \to \KK_G(\T_G^* W)$ be the composition of the map $r$ with the Bott isomorphism 
$\KK_G(\T_G^* W\times \Cbb_\chi) \to \KK_G(\T_G^* W)$. Note that $\br$ depends of the choice of the canonical complex structure 
on $\Cbb_\chi$.

The open subset $\Cbb_\chi\setminus\{0\}$ with the $G$-action is isomorphic to $G/G_\chi\times\Rbb$. Hence
$\T_G^*(W\times \Cbb_\chi\setminus\{0\})\simeq \T_G^*(W\times G/G_\chi)\times\T\Rbb$. Since the $G$-manifold 
$W\times G/G_\chi$ is isomorphic to $G\times_{G_\chi} W$, we get finally
\begin{eqnarray*}
\KK_G(\T_G^*(W\times \Cbb_\chi\setminus\{0\}))&=&\KK_G(\T_G^*(W\times G/G_\chi)\times\T\Rbb)\\
&\simeq& \KK_G(\T_G^*(W\times G/G_\chi))\\
&\simeq& \KK_G(\T_G^*(G\times_{G_\chi} W))\\
&\simeq& \KK_{G_\chi}(\T_{G_\chi}^*W).
\end{eqnarray*}

Let $\bj: \KK_{G_\chi}(\T_{G_\chi}^*W)\to \KK_G(\T_G^* V)$ be the composition of the map $j_*$ with the previous isomorphism 
$\KK_{G_\chi}(\T_{G_\chi}^* W)\simeq \KK_G(\T_G^*(W\times \Cbb_\chi\setminus\{0\}))$. The sequence (\ref{eq:six-term}) becomes 
\begin{equation}\label{sequence-JR}
\xymatrix{
\Ko_{G_\chi}(\T_{G_\chi}^*W)\ar[r]^{\bj} & \Ko_G(\T^*_G V) \ar[r]^{\br} & 
\Ko_G(\T_G^* W)\ar[d]_{\delta}\\
\Kun_{G}(\T_{G}^* W)\ar[u]^{\delta}& \ar[l]^{\br}  \Kun_G(\T^*_G V) &\ar[l]^{\bj} 
\Kun_{G_\chi}(\T^*_{G_\chi} W).
  }
\end{equation}

The following description of the morphism $\bj$ will be used in the next sections. Let $\beta\in\ggot$ such that $\ggot=\ggot_\chi\oplus\Rbb\beta$. Since the action of $G_\chi$ is trivial on $\Cbb_\chi$, the  product
$$
\Ko_{G}(\T_{G_\chi}^*W)\times \Ko_{G}(\T_{\Rbb\beta}^*\Cbb_\chi)\stackrel{\odot}{\longrightarrow} \Ko_{G}(\T_{G}^*V)
$$
is well defined. Let $\sigmad^{\Cbb_\chi}\in \Ko_{G}(\T_{\Rbb\beta}^* \Cbb_\chi)$ be the Cauchy-Riemann class.

\medskip

\begin{lem}\label{lem:J-sigma} Let $[\sigma]\in \Ko_{G_\chi}(\T_{G_\chi}^*W)$ be a class that is represented by a 
$G$-equivariant, $G_\chi$-transversally elliptic morphism $\sigma$. Then the product $\sigma\odot \sigmad^{\Cbb_\chi}$ 
is $G$-transversally elliptic and $\bj([\sigma])=\left[\sigma\odot \sigmad^{\Cbb_\chi}\right]$ in $\Ko_{G}(\T_{G}^*V)$.
\end{lem}

\medskip

\begin{proof} The character $\chi$ defines the inclusion $i: G/G_\chi \to \Cbb_\chi, g\mapsto \chi(g)$. Let 
$i_!: \Ko_G(\T_G^*(G/G_\chi\times W))\longrightarrow\Ko_G(\T_G^*V)$ be the push-forward morphism.

The manifold $G\times W$ is equipped with two $G\times G_\chi$-actions: $(g,h)\cdot_1 (x,w):= (gxh^{-1}, h\cdot w)$ and 
 $(g,h)\cdot_2 (x,w):= (gxh^{-1}, g\cdot w)$. The map $\theta(x,w)=(x,x^{-1}\cdot w)$ is an isomorphism between 
 $G\times_2 W$ and $G\times_1 W$. The quotients by $G$ and $G_\chi$ give us the maps 
$\pi_G: G\times_1 W\to W$, and $\pi_{G_\chi}:G\times_2 W\to G/G_\chi\times W$. 

We have 
\begin{equation}\label{eq:map-J}
\bj=  i_!\circ(\pi_{G_\chi}^*)^{-1}\circ \theta^*\circ\pi_G^*
\end{equation} 
where $(\pi_{G_\chi}^*)^{-1}\circ\, \theta^*\circ\pi_G^*:\Ko_{G_\chi}(\T_{G_\chi}^*W)\longrightarrow\Ko_G(\T_G^*(G/G_\chi\times W))$ is an isomorphism.

 It is an easy matter to check that if the class 
$[\sigma]\in \Ko_{G_\chi}(\T_{G_\chi}^*W)$ is represented by a  {\bf $\mathbf{G}$-equivariant}, $G_\chi$-transversally elliptic morphism $\sigma$, then $(\pi_{G_\chi}^*)^{-1}\circ \theta^*\circ\pi_G^*(\sigma)=\sigma\odot [0]$ where 
$[0]:\Cbb\to\{0\}$ is the zero symbol on $G/G_\chi$. Finally $\bj(\sigma)= i_!(\sigma\odot [0])=\sigma\odot \sigmad^{\Cbb_\chi}$.
\end{proof}

\medskip

\begin{rem}\label{rem-JR}
In the next sections, we will use the exact sequence (\ref{sequence-JR}), when $V$ is replaced by an invariant 
open subset $\Ucal_V$. Suppose that there exist invariant open subsets $\Ucal^1_W,\Ucal^2_W\subset W$ such that 
$\Ucal_V=\Ucal^1_W\bigsqcup \Ucal^2_W\times \Cbb_\chi\setminus\{0\}$. Then (\ref{sequence-JR}) becomes 
\begin{equation}\label{sequence-JR-bis}
\xymatrix{
\Ko_{G_\chi}(\T_{G_\chi}^*\Ucal^2_W)\ar[r]^{\bj} & \Ko_G(\T^*_G \Ucal_V) \ar[r]^{\br} & 
\Ko_G(\T_G^* \Ucal_W^1)\ar[d]_{\delta}\\
\Kun_{G}(\T_{G}^* \Ucal_W^1)\ar[u]^{\delta}& \ar[l]^{\br}  \Kun_G(\T^*_G \Ucal_V) &\ar[l]^{\bj} 
\Kun_{G_\chi}(\T^*_{G_\chi} \Ucal^2_W).
  }
\end{equation}
For example, if  $\Ucal_V= V^{gen}$, we take $\Ucal^1_W=W\cap V^{gen}$ and $\Ucal^2_W=W^{gen, G_\chi}$.
\end{rem}

%%%%%%%%%%%%%%%%%%%%%%%%%%%%%%%%%%%%%%%%%%%%%%%%%%%%%
%%%%%%%%%%%%%%%%%%%%%%%%%%%%%%%%%%%%%%%%%%%%%%%%%%%%%
\subsection{The index map is injective}\label{subsec:proof-injective}
%%%%%%%%%%%%%%%%%%%%%%%%%%%%%%%%%%%%%%%%%%%%%%%%%%%%%
%%%%%%%%%%%%%%%%%%%%%%%%%%%%%%%%%%%%%%%%%%%%%%%%%%%%%

Let us prove by induction on $n\geq 0$ the following fact
$$
(\mathrm{H}_n)\qquad \indice^G_V:\Ko_G(\T_G^* V)\longrightarrow R^{-\infty}(G)\ \ \mathrm{is\ one\ to\ one\ if}\ \ \dim V/V^\ggot\leq n.
$$

If $\dim V/V^\ggot=0$, we have $\T_G^* V=\T^* V$ and the index map $\Ko_G(\T^* V)\to R(G)$ is 
the inverse of the Bott isomorphism. 

Suppose now that $(\mathrm{H}_n)$ is true, and consider $G\circlearrowleft V$ such that $\dim V/V^\ggot= n+1$.  
We start with a decomposition $V=W\oplus \Cbb_\chi$ and the exact sequence (\ref{sequence-JR}). The induction map 
$\indGa: R^{-\infty}(G_\chi)\to R^{-\infty}(G)$ is defined by the relation $\indGa(E)=\left[L^2(G)\otimes 
E\right]^{G_\chi}$. We denote $\wedge^\bullet \overline{\Cbb_{\chi}}: R^{-\infty}(G)\to R^{-\infty}(G)$ 
the product by $1-\overline{\Cbb_{\chi}}$.

\begin{prop}
The following diagram is commutative
\begin{equation}\label{diagram-I-J}
\xymatrix@C=10mm{ 
\Ko_{G_\chi}(\T_{G_\chi}^*W) \ar[d]^{\indice^{G_\chi}_W}\ar[r]^{\bj}&
\Ko_G(\T_G^* V)\ar[d]^{\indice^{G}_V}\ar[r]^{\br}&
\Ko_G(\T_G^* W)\ar[d]^{\indice^{G}_W}\\
R^{-\infty}(G_\chi)\ar[r]^{\indGa} &
R^{-\infty}(G)\ar[r]^{\wedge^\bullet \overline{\Cbb_{\chi}}} &
R^{-\infty}(G).
}
\end{equation}
\end{prop}

\begin{proof} Let $\sigma\in \Ko_{G_\chi}(\T_{G_\chi}^*W)$. We have $\pi_{G}^*(\sigma)=\sigma\odot [0]$ where 
$[0]:\Cbb\to\{0\}$ is the zero symbol on $G$. Then the product formula says that 
$\indice^{G\times G_\chi}_{G\times W}(\pi_{G}^*(\sigma))= \indice^{G_\chi}_W(\sigma)\otimes L^2(G)$ and  thanks to (\ref{eq:map-J}), we see that
\begin{eqnarray*}
\indice^G_V(\bj(\sigma))&=&\indice^{G}_{G/G_\chi\times W}((\pi_{G_\chi}^*)^{-1}\circ \theta^*\circ\pi_G^*(\sigma))\\
&=&\left[\indice^{G\times G_\chi}_{G\times W}(\pi_{G}^*(\sigma))\right]^{G_\chi}\\
&=&\indGa\left(\indice^{G_\chi}_W(\sigma)\right).
\end{eqnarray*}

This proved the commutativity of the left part of the diagram, and the commutativity of the right part of the diagram is a particular case of Proposition \ref{prop:restriction}.
\end{proof}

\medskip 

We need now the following result that will be proved in Appendix A
\begin{lem}\label{lem-G-a}
The sequence 
\begin{equation}\label{eq:sequence-G-chi}
0\longrightarrow R^{-\infty}(G_\chi)\stackrel{\indGa}{\longrightarrow}
R^{-\infty}(G)\stackrel{\wedge^\bullet \overline{\Cbb_{\chi}}}{\longrightarrow}
R^{-\infty}(G)
\end{equation}
is exact.
\end{lem}

Lemma \ref{lem-G-a} tells us in particular that $\indGa$ is one to one. We can now finish the proof of the induction. In the commutative diagram (\ref{diagram-I-J}), the maps $\indice^G_W,\indice^{G_\chi}_W$ and $\indGa$ are one to one. It is an easy matter to deduces that $\indice^G_V$ is one to one.

\bigskip 

We end up this section with the following statement which is the direct consequence of the injectivity of $\indice^G_V$ (see Remark \ref{rem:thom-J-0-1}).

\begin{rem}\label{rem:thom-J-K}
Let $J_k, k=0,1$ be two invariants complex structures on $V$, and let $\thom_{\beta}(V, J_k)$ be the corresponding pushed symbols attached to an element $\beta$ satisfying $V^\beta=\{0\}$. There exists an invertible element $\Phi\in R(G)$ such that 
$$
\thom_{\beta}(V,J_0))=\Phi\cdot\thom_{\beta}(V,J_1)
$$
in $\Ko_G(\T_G^* V)$.
\end{rem}

%%%%%%%%%%%%%%%%%%%%%%%%%%%%%%%%%%%%%%%%%%%%%%%%%%%%%
%%%%%%%%%%%%%%%%%%%%%%%%%%%%%%%%%%%%%%%%%%%%%%%%%%%%%
\subsection{Generators of $\KK_G(\T^*_G V)$}\label{subsec:proof-generators-V}
%%%%%%%%%%%%%%%%%%%%%%%%%%%%%%%%%%%%%%%%%%%%%%%%%%%%%
%%%%%%%%%%%%%%%%%%%%%%%%%%%%%%%%%%%%%%%%%%%%%%%%%%%%%

Let $V$ be a real $G$-module : we equip $V/V^\ggot$ with an invariant complex structure. 
Let $A_G(V)\subset\Ko_G(\T_G^* V)$ be the submodule generated by the family $\bott(V^\ggot_\Cbb)\odot\thom_\gamma(V/V^\ggot)$, where $\gamma$ runs over the element of $\ggot$ satisfying $V^\gamma=V^\ggot$. Remark \ref{rem:thom-J-K} tells us that $A_G(V)$ is independent of the choice of the complex structure on $V/V^\ggot$.

 In this section we will prove by induction on $n\geq 0$ the following fact
$$
(\mathrm{H}_n)\qquad \Kun_G(\T_G^* V)=0 \quad \mathrm{and}\quad \Ko_G(\T_G^* V)=A_G(V)\ \ \mathrm{if}\ \ \dim V/V^\ggot\leq n.
$$

If $\dim V/V^\ggot=0$, we have $\T_G^* V=\T^* V$ and assertion $(\mathrm{H}_0)$ is a  direct consequence of the Bott isomorphism. 

Suppose now that $(\mathrm{H}_n)$ and is true, and consider $G\circlearrowleft V$ 
such that $\dim V/V^\ggot= n+1$. We have a decomposition $V=W\oplus \Cbb_\chi$ with $\bar{\chi}\neq 0$.  
If we apply\footnote{See Remark \ref{rem:induction}.}  $(\mathrm{H}_n)$ to  $G\circlearrowleft W$ and 
$G_\chi\circlearrowleft W$, we get first that $\Kun_G(\T_G^* W)=0$ and 
$\Kun_{G_\chi}(\T_{G_\chi}^* W)=0$. 

The long exact sequence (\ref{sequence-JR}) implies then that $\Kun_G(\T^*_G V)=0$, and 
induces the short exact sequence 
\begin{equation}\label{eq:sequence-JR-short}
0\longrightarrow \Ko_{G_\chi}(\T_{G_\chi}^*W) \stackrel{\bj}{\longrightarrow}
\Ko_G(\T_G^* V)\stackrel{\br}{\longrightarrow} \Ko_G(\T_G^* W)\longrightarrow 0.
\end{equation}

The assertion $(\mathrm{H}_n)$ gives also  $\Ko_G(\T_G^* W)=A_G(W)$ and $\Ko_{G_\chi}(\T_{G_\chi}^* W)=A_{G_\chi}(W)$. With the help of (\ref{eq:sequence-JR-short}), the equality $\Ko_G(\T_G^* V)=A_G(V)$ will follows from following Lemma.

\begin{lem}\label{lem:A-G-V} 
We have 
\begin{itemize}
\item $\bj\left( A_{G_\chi}(W)\right)\subset A_G(V)$, 
\item $A_G(W)\subset \br(A_G(V))$.
\end{itemize} 
\end{lem}
\begin{proof} We equip $V/V^\ggot=W/W^\ggot\oplus \Cbb_\chi$ with the complex structure $J:=J_\beta$ 
where $\langle\bar{\chi},\beta\rangle >0$. We will use the decomposition of complex $G$-vector spaces 
$W/W^\ggot\simeq W/W^{\ggot_\chi}\oplus W^{\ggot_\chi}/W^{\ggot}$, and the fact that $V^\ggot=W^\ggot$.

Let $\alpha:=\bott(W^{\ggot_\chi}_\Cbb)\odot\thom_\gamma(W/W^{\ggot_\chi})$ be a generator of $A_{G_\chi}(W)$. It is a $G$-equivariant symbol, hence Lemma \ref{lem:J-sigma} applies: its image by $\bj$ is equal to 
$\bj(\alpha)=\bott(W^{\ggot_\chi}_\Cbb)\odot\thom_\gamma(W/W^{\ggot_\chi})\odot \sigmad^{\Cbb_\chi}$. If we use the fact that 
$\sigmad^{\Cbb_\chi}=\thom_{- \beta}(\Cbb_\chi)-\thom_{\beta}(\Cbb_\chi)$, we see that $\bj(\alpha)=U_{-} - U_{+}$ where 
\begin{eqnarray*}
U_\pm&=&\bott(W^{\ggot_\chi}_\Cbb)\odot\thom_\gamma(W/W^{\ggot_\chi})\odot \thom_{\pm \beta}(\Cbb_\chi)\\
&=&\bott(W^{\ggot_\chi}_\Cbb)\odot\thom_{\gamma_{\pm}}(W/W^{\ggot_\chi}\oplus\Cbb_\chi)\qquad\qquad [1]\\
&=&\bott(V^{\ggot}_\Cbb)\odot\bott((W^{\ggot_\chi}/W^{\ggot})_\Cbb)\odot\thom_{\gamma_{\pm}}(W/W^{\ggot_\chi}\oplus\Cbb_\chi)\qquad [2]\\
&=&\wedge^\bullet\overline{W^{\ggot_\chi}/W^{\ggot}}\otimes\bott(V^{\ggot}_\Cbb)\odot\thom_{\gamma_{\pm}}(V/V^\ggot)\qquad \qquad [3].
\end{eqnarray*}
In $[1]$, the term $\gamma_\pm$ is equal to $\gamma \pm t\beta$ with $0<<t<<1$ (see Lemma \ref{lem:produit-thom-beta}). 
In $[2]$, we use that $W^{\ggot_\chi}\simeq V^\ggot\oplus W^{\ggot_\chi}/W^{\ggot}$. In $[3]$ we use that $V/V^\ggot= W/W^{\ggot_\chi}\oplus W^{\ggot_\chi}/W^{\ggot}\oplus \Cbb_\chi$ (see 
Proposition \ref{prop:thom-beta-bott}).

We have proved that $\bj(\alpha)$ belongs to $A_G(V)$ for any generator $\alpha$ of $A_{G_\chi}(W)$. We get the first point, since the restriction $R(G)\to R(G_\chi)$ is surjective.

Let $\alpha':=\bott(W^{\ggot}_\Cbb)\odot\thom_\gamma(W/W^{\ggot})$ be a generator of $A_{G}(W)$. Thanks to  
Proposition \ref{prop:R-thom-beta}, we see that $\alpha'=\br(\alpha'')$ with $\alpha''=\bott(V^{\ggot}_\Cbb)\odot\thom_\gamma(V/V^{\ggot})$. The second point is then proved.
\end{proof}

%%%%%%%%%%%%%%%%%%%%%%%%%%%%%%%%%%%%%%%%%%%%%%%%%%%%%
%%%%%%%%%%%%%%%%%%%%%%%%%%%%%%%%%%%%%%%%%%%%%%%%%%%%%
\subsection{Generators of $\KK_G(\T_G^* V^{gen})$}\label{subsec:proof-generators-V-gen}
%%%%%%%%%%%%%%%%%%%%%%%%%%%%%%%%%%%%%%%%%%%%%%%%%%%%%
%%%%%%%%%%%%%%%%%%%%%%%%%%%%%%%%%%%%%%%%%%%%%%%%%%%%%

Let $B_G(V)$ be the submodule of $\Ko_G(\T_G^* V^{gen})$ generated by the family 
$\bott(V^\ggot_\Cbb)\odot\sigmad^{V/V^\ggot,\varphi}$ where $\varphi$ runs over the $(G,V)$-flag.

In this section we will prove by induction on $n\geq 0$ the following fact

$$
(\mathrm{H}'_n) \qquad \Kun_G(\T_G^* V^{gen})=0 \quad \mathrm{and} \quad \Ko_G(\T_G^* V^{gen})=B_G(V) \ \ \mathrm{if}\ \  \dim V/V^\ggot\leq n.
$$

\medskip

If $\dim V/V^\ggot=0$, we have $\T_G^* V^{gen}=\T^* V$ and $(\mathrm{H}'_0)$ is a direct consequence 
of the Bott isomorphism. Suppose now that $(\mathrm{H}_n)$ is true, and consider $G\circlearrowleft V$ 
such that $\dim V/V^\ggot= n+1$. We have an invariant decomposition $V=W\oplus\Cbb_\chi$, with 
$\bar{\chi}\neq 0$, and 
$$
V^{gen}= V^{gen}\cap W\bigsqcup W^{gen, G_\chi}\times \Cbb_\chi\setminus\{0\}.
$$
Note that $V^{gen}\cap W$ is either equal to $W^{gen}$ (if the $G$-orbits in $V$ and $W$ have 
the same maximal dimension) or is empty. Following Remark \ref{rem-JR}, we have the exact sequence
\begin{equation}\label{sequence-JR-delta}
\xymatrix{
\Ko_{G_\chi}(\T_{G_\chi}^*W^{gen, G_\chi})\ar[r]^{\bj} & \Ko_G(\T^*_G V^{gen}) \ar[r]^{\br} & 
\Ko_G(\T_G^* W^{gen})\ar[d]_{\delta}\\
\Kun_{G}(\T_{G}^* W^{gen})\ar[u]^{\delta}& \ar[l]^{\br}  \Kun_G(\T^*_G V^{gen}) &\ar[l]^{\bj} 
\Kun_{G_\chi}(\T^*_{G_\chi} W^{gen, G_\chi})
  }
\end{equation}
when $V^{gen}\cap W\neq\emptyset$. On the other hand, when $V^{gen}\cap W= \emptyset$, we have an isomorphism
\begin{equation}\label{eq:J-iso}
\bj:\KK_{G_\chi}(\T_{G_\chi}^*W^{gen,G_\chi})\longrightarrow \KK_{G}(\T_{G}^*V^{gen}).
\end{equation}

If we apply $(\mathrm{H}'_n)$ to  $G\circlearrowleft W$ and 
$G_\chi\circlearrowleft W$, we get first $\Kun_G(\T_G^* W^{gen})=0$ and 
$\Kun_{G_\chi}(\T_{G_\chi}^* W^{gen,G_\chi})=0$. Using the bottom of the diagram (\ref{sequence-JR-delta}) and the 
isomorphism (\ref{eq:J-iso}), we get $\Kun_G(\T^*_G V)=0$. Moreover, the long exact sequence  (\ref{sequence-JR-delta})
induces the short exact sequence 
$$
0\longrightarrow \Ko_{G_\chi}(\T_{G_\chi}^*W^{gen,G_\chi}) \stackrel{\bj}{\longrightarrow}
\Ko_G(\T_G^* V^{gen})\stackrel{\br}{\longrightarrow} \Ko_G(\T_G^* (V^{gen}\cap W))\longrightarrow 0.
$$

Since the assertion $(\mathrm{H}'_n)$  gives also 
$$
\Ko_G(\T_G^* W^{gen})=B_G(W)\quad \mathrm{and}\quad \Ko_{G_\chi}(\T_{G_\chi}^* W^{gen,G_\chi})=B_{G_\chi}(W),
$$
the equality $\Ko_G(\T_G^* V^{gen})=A_G(V)$ will follows from following Lemma. 

\begin{lem}\label{lem:A-G-V-gen} 
We have 
\begin{itemize}
\item $\bj\left( B_{G_\chi}(W)\right)\subset B_G(V)$, 
\item $B_G(W)\subset \br(B_G(V))$, when $V^{gen}\cap W\neq\emptyset$.
\end{itemize} 
\end{lem}
\begin{proof} Let $\beta\in\ggot$ such that $\langle\bar{\chi},\beta\rangle >0$: we have $\ggot=\ggot_\chi\oplus\Rbb\beta$. 
For any $(W,G_\chi)$-flag $\varphi$, we consider the element 
$$
\alpha:= \bott(W^{\ggot_\chi}_\Cbb)\odot \sigmad^{W/W^{\ggot_\chi},\varphi}\in \Ko_{G_\chi}(\T_{G_\chi}^*W^{gen})
$$
and we want to compute its image by $\bj$. 

We note that the minimal stabilizer $H_{min}\subset G$ for the $G$-action on $V$ is equal to the minimal stabilizer for the $G_\chi$-action on $W$. Let $s:=\dim G_\chi -\dim H_{min}$. A $(G_\chi,W)$-flag $\varphi$  corresponds to 
\begin{itemize}
\item a decomposition $W/W^{\ggot_\chi}=W^\varphi_1\oplus\cdots\oplus W^\varphi_s$ in complex $G$-subspaces
\item a decomposition $\ggot_\chi=\hgot_{min}\oplus\Rbb\beta^\varphi_1\oplus\cdots\Rbb\beta^\varphi_s$
\end{itemize} 
such that for any $1\leq k\leq s$, $\beta^\varphi_k$ acts trivially on  $W^\varphi_1\oplus\cdots\oplus W^\varphi_{k-1}$ and $\beta^\varphi_k$ acts bijectively on $W^\varphi_k$. The term $\sigmad^{W/W^{\ggot_\chi},\varphi}\in\Ko_{G_\chi}(\T_{G_\chi}^* (W/W^{\ggot_\chi})^{gen})$ is equal to the product of $\sigmad^{\varphi,k}\in \Ko_{G_\chi}(\T_{\Rbb\beta^\varphi_k}^* W^\varphi_{k})$, for $1\leq k\leq s$.

Since $V/V^\ggot\simeq W^{\ggot_\chi}/W^{\ggot}\oplus\Cbb_\chi\oplus W/W^{\ggot_\chi}$, we can define a $(V,G)$-flag $\psi$ as follows: 
\begin{itemize}
\item $V_1^\psi:= W^{\ggot_\chi}/W^{\ggot}\oplus\Cbb_\chi$ and $\beta^\psi_1=\beta$, 
\item  $V_k^\psi:= W^\varphi_{k-1}$ and $\beta_k^\psi:= \beta^\varphi_{k-1}$ for $2\leq k\leq s+1$.
\end{itemize} 

We note that the $G_\chi$-transversally symbols $\sigmad^{\varphi,k}\in \Ko_{G_\chi}(\T_{\Rbb\beta^\varphi_k}W^\varphi_{k})$ correspond to the restriction of the $G$-transversally symbols $\sigmad^{\psi,k+1}\in \Ko_{G}(\T_{\Rbb\beta^\psi_{k+1}}V^\psi_{k+1})$.

Finally, thanks to Lemma \ref{lem:J-sigma}  we have
\begin{eqnarray*}
\bj(\alpha)&=&
\sigmad^{\Cbb_\chi}\odot\bott(V^{\ggot}_\Cbb)\odot\bott((W^{\ggot_\chi}/W^{\ggot})_\Cbb)\odot\sigmad^{\psi,2}\odot\cdots\odot \sigmad^{\psi,s}\\
&=&\wedge^\bullet\overline{W^{\ggot_\chi}/W^{\ggot}}\otimes\bott(V^{\ggot}_\Cbb)\odot\sigmad^{\psi,1}\odot\sigmad^{\psi,2}\odot\cdots\odot \sigmad^{\psi,s}\\
&=&\wedge^\bullet\overline{W^{\ggot_\chi}/W^{\ggot}}\otimes \bott(V^{\ggot}_\Cbb)\odot\sigmad^{V/V^\ggot,\psi}.
\end{eqnarray*}
Here we use the  identity $\sigmad^{\Cbb_\chi}\odot\bott(W^{\ggot_\chi}/W^{\ggot}_\Cbb)=
\wedge^\bullet\overline{W^{\ggot_\chi}/W^{\ggot}}\otimes\sigmad^{\psi,1}$, 
valid in $\Ko_G(\T^*_{\Rbb\beta}(\Cbb_\chi\oplus W^{\ggot_\chi}/W^{\ggot}))$, which is proved in Proposition \ref{prop:sigmad-bott}. Since $R(G)\to R(G_\chi)$ is onto, we have proved that $\bj(A_{G_\chi}(W))\subset A_G(V)$. 

\medskip

\medskip

Suppose now that $V^{gen}\cap W\neq\emptyset$, and let us prove now that $A_G(W)\subset\br(A_{G}(V))$. Let $\varphi$ 
be a $(G,W)$-flag : let $W/W^\ggot=W^\varphi_1\oplus\cdots\oplus W^\varphi_s$ and
 $\ggot=\hgot_{min}\oplus\Rbb\beta^\varphi_1\oplus\cdots\oplus\Rbb\beta^\varphi_s$ be 
 the corresponding decompositions. The hypothesis $V^{gen}\cap W\neq\emptyset$ means that the minimal stabilizer 
 $\hgot_{min}$ for the $\ggot$-action in $W$ is contained in $\ggot_\chi$. Hence $\bar{\chi}$ does not belongs 
 to $(\Rbb\beta^\varphi_1\oplus\cdots\oplus\Rbb\beta^\varphi_s)^\perp$. Let 
$$
k=\max\{i\ \vert \ \langle\bar{\chi},\beta^\varphi_i\rangle\neq 0\}.
$$
Let $\psi$ be the $(G,V)$-flag defined as follows: 
\begin{itemize}
\item $V^\psi_i= W^\varphi_i$ is $i\neq k$, and $V^\psi_k= W^\varphi_k\oplus\Cbb_\chi$,
\item $\beta^\varphi_s=\beta^\varphi_s$ for $1\leq k\leq s$.
\end{itemize}

Then we have 
\begin{eqnarray*}
\br(\bott(V^\ggot_\Cbb)\odot\sigmad^{V/V^\ggot,\psi})&=&
\bott(W^\ggot_\Cbb)\odot\sigmad^{\varphi,1}\odot\cdots\br(\sigmad^{\psi,k})\cdots\odot \sigmad^{\varphi,s}\\
&=&\bott(W^\ggot_\Cbb)\odot\sigmad^{\varphi,1}\odot\cdots\sigmad^{\varphi,k}\cdots\odot \sigmad^{\varphi,s}\\
&=&\bott(W^\ggot_\Cbb)\odot\sigmad^{W/W^\ggot,\varphi}
\end{eqnarray*}
We use here the relation $\br(\sigmad^{\psi,k})=\sigmad^{\varphi,k}$ (see Proposition \ref{prop:R-thom-beta}). 
It proves that $A_G(W)\subset\br(A_{G}(V))$.
\end{proof}

%%%%%%%%%%%%%%%%%%%%%%%%%%%%%%%%%%%%%%%%%%%%%%%%%%%%%%%%%%%%%%%%%%%%%%%%
%%%%%%%%%%%%%%%%%%%%%%%%%%%%%%%%%%%%%%%%%%%%%%%%%%%%%%%%%%%%%%%%%%%%%%%%
\subsection{$\Ko_G(\T^*_G V)$ is isomorphic to $\Fcal_G(V)$}\label{sec:proof-F-CPV}
%%%%%%%%%%%%%%%%%%%%%%%%%%%%%%%%%%%%%%%%%%%%%%%%%%%%%%%%%%%%%%%%%%%%%%%%
%%%%%%%%%%%%%%%%%%%%%%%%%%%%%%%%%%%%%%%%%%%%%%%%%%%%%%%%%%%%%%%%%%%%%%%%

For any $G$-module $V$, we denote $\Fcal_G(V)'$ the image of $\Ko_G(\T^*_G V)$  by $\indice^G_V$. We know from Section \ref{subsec:proof-injective} that the index map $\indice^G_V$ is injective, hence $\Fcal_G(V)'\simeq\Ko_G(\T^*_G V)$. Let $\Fcal_G(V)$ be the generalized Dahmen-Michelli submodule defined in the introduction. We start with the following

\begin{lem}
We have  $\Fcal_G(V)'\subset \Fcal_G(V)$.
\end{lem}

\begin{proof} 
Let $\sigma\in \Ko_G(\T^*_G V)$ and let $\hgot\in\Delta_G(V)$. Since the vector space $V/V^\hgot$ carries an 
invariant complex structure we have a restriction morphism $\br_\hgot:\Ko_G(\T^*_G V)\to \Ko_G(\T^*_G V^\hgot)$. 
Let $i_!: \Ko_G(\T^*_G V^\hgot)\to \Ko_G(\T^*_G V)$ be the push-forward morphism associated to 
the inclusion $V^\hgot\croc V$. Thanks to Proposition \ref{prop:restriction} we know that $i_!\circ \br_\hgot(\sigma)= \sigma\otimes 
\wedge^\bullet\overline{V/V^\hgot}$, and then 
\begin{equation}\label{eq:restriction-prop}
\wedge^\bullet\overline{V/V^\hgot}\otimes \indice^G_V(\sigma)= \indice^G_{V^\hgot}(\br_\hgot(\sigma)).
\end{equation}
But since the action of $H$ is trivial on $V^\hgot$, we know that $\indice^G_{V^\hgot}(\br_\hgot(\sigma))\in \langle R^{-\infty}(G/H)\rangle$ (see Remark \ref{rem:indice-H-trivial}). The inclusion $\Fcal_G(V)'\subset \Fcal_G(V)$ is proved.
\end{proof}

\medskip

We will now prove by induction on $n\geq 0$ the following fact
$$
(\mathrm{H}''_n) \qquad \Fcal_G(V)'=\Fcal_G(V)\quad  \ \ \mathrm{if}\ \  \dim V/V^\ggot\leq n.
$$

\medskip

If $\dim V/V^\ggot=0$, we have $\T_G^* V=\T^* V$ and $\Delta_G(V)=\{\ggot\}$. In this situation, $\hgot_{min}=\ggot$ and $\langle R^{-\infty}(G/H_{min})\rangle= R(G)$. We have then $\Fcal_G(V)=R(G)$, and $(\mathrm{H}''_0)$ is a  direct consequence of the Bott isomorphism. 

Suppose now that $(\mathrm{H}''_n)$ and is true, and consider $G\circlearrowleft V$ 
such that $\dim V/V^\ggot= n+1$. We have a decomposition $V=W\oplus \Cbb_\chi$ with $\bar{\chi}\neq 0$.  
If we apply $(\mathrm{H}''_n)$ to  $G\circlearrowleft W$ and $G_\chi\circlearrowleft W$, we get $\Fcal_G(W)'=\Fcal_G(W)$ and $\Fcal_{G_\chi}(W)'=\Fcal_{G_\chi}(W)$. The following Lemma will be the key point of our induction.

\medskip

\begin{lem}\label{lem-G-a-dm}
$\bullet$ Let $H\subset G_\chi$ be a closed subgroup ($G$ is abelian). For any $\Phi\in R^{-\infty}(G_\chi)$, we have the equivalences
\begin{equation}\label{eq:induction-support}
\Phi\in \langle R^{-\infty}(G_\chi/H)\rangle\Longleftrightarrow \indGa(\Phi)\in  \langle R^{-\infty}(G/H)\rangle,
\end{equation}
\begin{equation}\label{eq:induction-F-G-V}
\Phi\in \Fcal_{G_\chi}(W)\Longleftrightarrow \indGa(\Phi)\in  \Fcal_{G}(V).
\end{equation}

$\bullet$ The exact sequence (\ref{eq:sequence-G-chi}) specializes in the exact sequence
\begin{equation}\label{eq:sequence-G-chi-F}
0\longrightarrow \Fcal_{G_\chi}(W)\stackrel{\indGa}{\longrightarrow}
\Fcal_G(V)\stackrel{\wedge^\bullet \overline{\Cbb_{\chi}}}{\longrightarrow}
\Fcal_G(W).
\end{equation}
\end{lem}

\medskip

\begin{proof} Let us consider the first point. For $\Phi:=\sum_{\mu\in\widehat{G_\chi}} m(\mu)\,\Cbb_\mu \in R^{-\infty}(G_\chi)$,  we have 
$\indGa(\Phi)= \sum_{\varphi\in\widehat{G}} m(\pi_{G_\chi}(\varphi))\, \Cbb_\varphi$, where 
$\pi_{G_\chi}: \widehat{G}\to \widehat{G_\chi}$. We see then that $\supp(\indGa(\Phi))= \pi_{G_\chi}^{-1}\left(\supp(\Phi)\right)$. If 
$\pi_H: \widehat{G}\to \widehat{H}$ and $\pi'_H: \widehat{G_\chi}\to \widehat{H}$ denote the projections, we have then the following relation
$$
\pi_H\left(\supp\left(\indGa(\Phi)\right)\right)=\pi_H'\left(\supp(\Phi)\right)
$$
that induces (\ref{eq:induction-support}).

For any $\Phi\in R^{-\infty}(G_\chi)$ and any subspace $\hgot\in\Delta_G(V)$, we consider the  expression $\Omega:=\wedge^\bullet\overline{V/V^\hgot}\otimes \indGa(\Phi)$. We have two cases:
\begin{itemize}
\item Either $\hgot \nsubseteq \ggot_\chi$ : here $\Cbb_\chi\subset V/V^\hgot$ and 
$\wedge^\bullet\overline{V/V^\hgot}=\wedge^\bullet\overline{\Cbb_\chi}\otimes \delta$. In this case, $\Omega=0$ because 
$\wedge^\bullet\overline{\Cbb_\chi}\circ \indGa=0$.

\item Or $\hgot \subset \ggot_\chi$ : here $\hgot\in\Delta_{G_\chi}(W)$  and $V/V^\hgot=W/W^\hgot$. In this case $\Omega=
\indGa(\wedge^\bullet\overline{W/W^\hgot}\otimes\Phi)$. 
\end{itemize}
It is then immediate that the equivalence (\ref{eq:induction-F-G-V}) follows from (\ref{eq:induction-support}). 

Thanks to (\ref{eq:induction-F-G-V}) it is an easy matter to check that the sequence (\ref{eq:sequence-G-chi-F}) is exact at 
$\Fcal_G(V)$. We leave to the reader the checking that $\wedge^\bullet \overline{\Cbb_{\chi}}\cdot \Fcal_G(V)\subset \Fcal_G(W)$.
So the second point is proved.
\end{proof}

\medskip

Let $I_\Fcal: \Fcal_G(V)'\croc\Fcal_G(V)$ be the inclusion.  Finally, we have the following commutative diagram, 
where all the horizontal sequences are exact :
$$
\xymatrix@C=10mm{ 
0\ar[r] & \Ko_{G_\chi}(\T_{G_\chi}^*W) \ar[d]\ar[r]^{\bj}&
\Ko_G(\T_G^* V)\ar[d]\ar[r]^{\br}&
\Ko_G(\T_G^* W)\ar[d] \ar[r] & 0\\
0\ar[r] & \Fcal_{G_\chi}(W)' \ar[d]\ar[r]^{\indGa} &
\Fcal_{G}(V)'\ar[d]^{I_\Fcal}\ar[r]^{\wedge^\bullet \overline{\Cbb_{\chi}}} &
\Fcal_{G}(W)' \ar[d]\ar[r] & 0\\
& \Fcal_{G_\chi}(W)\ar[r]^{\indGa} &
\Fcal_{G}(V) \ar[r]^{\wedge^\bullet \overline{\Cbb_{\chi}}} &
\Fcal_{G}(W). & 
}
$$
Except for $I_\Fcal$, we know that all the vertical arrows are isomorphism. It is an easy exercise to check that 
$I_\Fcal$ must be an isomorphism.

%%%%%%%%%%%%%%%%%%%%%%%%%%%%%%%%%%%%%%%%%%%%%%%%%%%%%%%%%%%%%%%%%%%%%%%%
%%%%%%%%%%%%%%%%%%%%%%%%%%%%%%%%%%%%%%%%%%%%%%%%%%%%%%%%%%%%%%%%%%%%%%%%
\subsection{$\Ko_G(\T^*_G V^{gen})$ is isomorphic to $\dm_G(V)$}\label{sec:proof-dm-CPV}
%%%%%%%%%%%%%%%%%%%%%%%%%%%%%%%%%%%%%%%%%%%%%%%%%%%%%%%%%%%%%%%%%%%%%%%%
%%%%%%%%%%%%%%%%%%%%%%%%%%%%%%%%%%%%%%%%%%%%%%%%%%%%%%%%%%%%%%%%%%%%%%%%

For any $G$-module $V$, we denote $\dm_G(V)'$ the image of $\Ko_G(\T^*_G V^{gen})$ by $\indice^G_V$. Since the maps 
$j_*:\Ko_G(\T^*_G V^{gen})\to \Ko_G(\T^*_G V)$ and $\indice^G_V$ are injective (see Remark \ref{rem:j-injective} and Section \ref{subsec:proof-injective}), we have 
$$
\dm_G(V)'\simeq\Ko_G(\T^*_G V^{gen}).
$$ 
for any $G$-module. Let $\dm_G(V)$ be the generalized Dahmen-Michelli submodules defined in the introduction. We start with the following 

\medskip

\begin{lem}
We have  $\dm_G(V)'\subset\dm_G(V)$.
\end{lem}

\begin{proof} 
Let $\tau\in  \Ko_G(\T^*_G V^{gen})$ and $j_*(\tau)\in \Ko_G(\T^*_G V)$. First we remark that 
$\indice^G_V(\tau)\in\langle R^{-\infty}(G/H_{min})\rangle$ since $H_{min}$ acts trivially on $V$ (see Remark \ref{rem:indice-H-trivial}).

 Let $\hgot\neq\hgot_{min}$ be a stabilizer in $\Delta_G(V)$. Since $V^\hgot\cap V^{gen}=\emptyset$ the composition 
$\br_\hgot\circ j_*$ is the zero map, and (\ref{eq:restriction-prop}) gives in this case that 
$\wedge^\bullet\overline{V/V^\hgot}\otimes \indice^G_V(j_*(\sigma))=0$. Since by definition $\indice^G_V(\tau)=\indice^G_V(j_*(\tau))$, the inclusion $\dm_G(V)'\subset\dm_G(V)$ is proved. 
\end{proof}

\medskip

We will now prove by induction on $n\geq 0$ the following fact
$$
(\mathrm{H}'''_n) \qquad \dm_G(V)'=\dm_G(V)
\ \ \mathrm{if}\ \  \dim V/V^\ggot\leq n.
$$

\medskip

If $\dim V/V^\ggot=0$, we have $\T_G^* V=\T^* V$, $V^{gen}=V$ and $\Delta_G(V)=\{\ggot\}$. In this situation, $\hgot_{min}=\ggot$ and $\langle R^{-\infty}(G/H_{min})\rangle= R(G)$. We have then $\dm_G(V)=R(G)$, and assertion $(\mathrm{H}'''_0)$ is a  direct consequence of the Bott isomorphism. 

Suppose now that $(\mathrm{H}'''_n)$ and is true, and consider $G\circlearrowleft V$ 
such that $\dim V/V^\ggot= n+1$. We have a decomposition $V=W\oplus \Cbb_\chi$ with $\bar{\chi}\neq 0$.  
If we apply $(\mathrm{H}'''_n)$ to  $G\circlearrowleft W$ and $G_\chi\circlearrowleft W$, we get  $\dm_G(W)'=\dm_G(W)$ and 
$\dm_{G_\chi}(W)'=\dm_{G_\chi}(W)$.

It works like in the previous section, apart for the dichotomy concerning $V^{gen}\cap W$.  We have the following 
\begin{lem}\label{lem-G-a-dm-2}
$\bullet$ Let $H\subset G_\chi$ be a closed subgroup ($G$ is abelian). For any $\Phi\in R^{-\infty}(G_\chi)$, we have the equivalences
\begin{equation}\label{eq:dm-induction}
\Phi\in \dm_{G_\chi}(W)\Longleftrightarrow \indGa(\Phi)\in  \dm_{G}(V).
\end{equation}
$\bullet$ If $V^{gen}\cap W\neq \emptyset$, the exact sequence (\ref{eq:sequence-G-chi}) specializes in the exact sequence
\begin{equation}\label{eq:sequence-G-chi-DM-1}
0\longrightarrow \dm_{G_\chi}(W)\stackrel{\indGa}{\longrightarrow}
\dm_G(V)\stackrel{\wedge^\bullet \overline{\Cbb_{\chi}}}{\longrightarrow}
\dm_G(W).
\end{equation}
$\bullet$ If $V^{gen}\cap W=\emptyset$, the exact sequence (\ref{eq:sequence-G-chi}) induces the isomorphism
\begin{equation}\label{eq:sequence-G-chi-DM-2}
\indGa: \dm_{G_\chi}(W)\stackrel{\sim}{\longrightarrow}
\dm_G(V).
\end{equation}
\end{lem}

\begin{proof} Let $\Phi\in R^{-\infty}(G_\chi)$ and $\hgot\in\Delta_G(V)$. We consider the term 
$\Omega:=\wedge^\bullet\overline{V/V^\hgot}\otimes \indGa(\Phi)$. Like in the proof of lemma \ref{lem-G-a-dm}, we have two cases:
\begin{itemize}
\item Either $\hgot \nsubseteq \ggot_\chi$ : in this case  $\Omega=0$.
\item Or $\hgot \subset \ggot_\chi$ : here $\hgot\in\Delta_{G_\chi}(W)$ and $V/V^\hgot=W/W^\hgot$. In this case 
$\Omega=\indGa(\eta)$ with $\eta=\wedge^\bullet\overline{W/W^\hgot}\otimes\Phi$. 
\end{itemize}

Since the minimal stabilizer\footnote{Says $\hgot_{min}$ wih corresponding group $H_{min}=\exp(\hgot_{min})$.} for the $G_\chi$ action on $W$ coincides with the minimal stabilizer for the $G$ action on $V$, the relation 
(\ref{eq:induction-support}) induces the equivalence $\Phi\in \langle R^{-\infty}(G_\chi/H_{min})\rangle$ $\Longleftrightarrow \indGa(\Phi)\in  \langle R^{-\infty}(G/H_{min})\rangle$. For the stabilizers $\hgot_{min}\varsubsetneq \hgot\subset \ggot_\chi$, 
using the fact that $\indGa$ is injective, we see that $\wedge^\bullet\overline{V/V^\hgot}\otimes \indGa(\Phi)=0$ if and only if $\wedge^\bullet\overline{W/W^\hgot}\otimes \Phi=0$. The first point follows.

Thanks to (\ref{eq:dm-induction}) it is an easy matter to check that the sequence (\ref{eq:sequence-G-chi}) specializes in the exact sequence
$0\to \dm_{G_\chi}(W)\stackrel{\alpha}{\longrightarrow} \dm_G(V)\stackrel{\beta}{\longrightarrow}
R^{-\infty}(G)$, where $\alpha=\indGa$ and $\beta= \wedge^\bullet \overline{\Cbb_{\chi}}$.  We can precise this sequence as follows. 

Let $\hgot_{min}(W),\hgot_{min}(V)$ be respectively the minimal infinitesimal stabilizer for the $G$-action on $W$ and $V$. We note that $V^{gen}\cap W\neq \emptyset\Longleftrightarrow \hgot_{min}(W)\subset \ggot_\chi\Longleftrightarrow \hgot_{min}(W)=\hgot_{min}(V)$.

Suppose that $V^{gen}\cap W\neq \emptyset$, and let us check that the image of $\beta$ is contained in $\dm_G(W)$. 
Take $\Phi\in\dm_G(V)$ and $\hgot\in\Delta_G(W)$. Let $\beta(\Phi)=\wedge^\bullet\overline{\Cbb_\chi}\otimes\Phi$.  We have to consider three cases :
\begin{enumerate}
\item If $\hgot=\hgot_{min}(W)$, then $\Phi$ and $\beta(\Phi)$ belong to $\langle R^{-\infty}(G/H_{min}(W))\rangle= \langle R^{-\infty}(G/H_{min}(V))\rangle$.
\item If $\hgot_{min}(W)\varsubsetneq \hgot\subset \ggot_\chi$, then $\wedge^\bullet\overline{V/V^\hgot}=\wedge^\bullet\overline{W/W^\hgot}$ and $\wedge^\bullet\overline{W/W^\hgot}\otimes \beta(\Phi)=\beta(\wedge^\bullet\overline{V/V^\hgot}\otimes \Phi)=0$.
\item If $\hgot\nsubseteq \ggot_\chi$, then $V/V^\hgot=W/W^\hgot\oplus\Cbb_\chi$. We get  then $\wedge^\bullet\overline{W/W^\hgot}\otimes \beta(\Phi)=\wedge^\bullet\overline{V/V^\hgot}\otimes \Phi=0$.
\end{enumerate}
We have proved that $\beta(\Phi)\in\dm_G(W)$.

Suppose now that $V^{gen}\cap W=\emptyset$, and let us check that $\beta$ is the zero map. Let 
$\hgot:=\hgot_{min}(W)\in\Delta_G(V)$.  We have $V/V^\hgot=\Cbb_\chi$ since $\hgot\nsubseteq\ggot_\chi$, and by definition we have 
$\beta(\Phi)=\wedge^\bullet\overline{\Cbb_\chi}\otimes \Phi=\wedge^\bullet\overline{V/V^\hgot}\otimes \Phi=0$ for any $\Phi\in\dm_G(V)$.
\end{proof}

Let $I_\dm: \dm_G(V)'\croc\dm_G(V)$ be the inclusion. If $V^{gen}\cap W\neq \emptyset$, we have the following commutative diagram
$$
\xymatrix@C=10mm{ 
0\ar[r] & \Ko_{G_\chi}(\T_{G_\chi}^*W^{gen,G_\chi}) \ar[d]\ar[r]^{\bj}& \Ko_G(\T_G^* V^{gen})\ar[d]\ar[r]^{\br}&
\Ko_G(\T_G^* W^{gen})\ar[d] \ar[r] & 0\\
0\ar[r] & \dm_{G_\chi}(W)' \ar[d]\ar[r]^{\indGa} &\dm_{G}(V)'\ar[d]^{I_\dm}\ar[r]^{\wedge^\bullet \overline{\Cbb_{\chi}}} &
\dm_{G}(W)' \ar[d]\ar[r] & 0\\
0\ar[r] & \dm_{G_\chi}(W)\ar[r]^{\indGa} &
\dm_{G}(V) \ar[r]^{\wedge^\bullet \overline{\Cbb_{\chi}}} &
\dm_{G}(W), & 
}
$$
and if $V^{gen}\cap W=\emptyset$, we have the other commutative diagram
$$
\xymatrix@C=10mm{ 
0\ar[r] & \Ko_{G_\chi}(\T_{G_\chi}^*W^{gen,G_\chi}) \ar[d]\ar[r]^{\bj}&\Ko_G(\T_G^* V^{gen})\ar[d] \ar[r] & 0\\
0\ar[r] & \dm_{G_\chi}(W)' \ar[d]\ar[r]^{\indGa}                                  & \dm_{G}(V)'\ar[d]^{I_\dm} \ar[r] & 0\\
0\ar[r]& \dm_{G_\chi}(W)\ar[r]^{\indGa} & \dm_{G}(V) \ar[r]&   0.
}
$$
In both diagrams,  all the horizontal sequences are exact, and except for $I_\dm$, we know that all the vertical arrows are isomorphisms. It is an easy exercise to check that in both cases $I_\dm$ must be an isomorphism.

%%%%%%%%%%%%%%%%%%%%%%%%%%%%%%%%%%%%%%%%%%%%%%%%%%%%%%%%%%%%%%%%%%%%%%%%
%%%%%%%%%%%%%%%%%%%%%%%%%%%%%%%%%%%%%%%%%%%%%%%%%%%%%%%%%%%%%%%%%%%%%%%%
\subsection{Decomposition of $\Ko_G(\T^*_G V)\simeq\dm_G(V)$}\label{sec:decomposition-V}
%%%%%%%%%%%%%%%%%%%%%%%%%%%%%%%%%%%%%%%%%%%%%%%%%%%%%%%%%%%%%%%%%%%%%%%%
%%%%%%%%%%%%%%%%%%%%%%%%%%%%%%%%%%%%%%%%%%%%%%%%%%%%%%%%%%%%%%%%%%%%%%%%

Let $V$ be a real $G$-module such that $V^\ggot=\{0\}$.  Let $J$ be an invariant complex structure 
on $V$. Let $\Wcal\subset \widehat{G}$ be the set of weights: $\chi\in\Wcal$ if 
$V_\eta:=\{v\in V \ \vert \ g\cdot v= \eta(g)v\}\neq \{0\}$. The differential of $\eta$ is denoted $i\bar{\eta}$ with $\bar{\eta}\in\ggot^*$.
Let $\overline{\Wcal}:=\{\bar{\eta}\ \vert\ \eta\in\Wcal\}$ : it is the set of infinitesimal weights for the action of $\ggot$ on $V$.

%Let $\hgot$ be the orthogonal of $\sum_{\bar{\chi}\in\overline{\Delta}}\Rbb\bar{\chi}\subset\ggot^*$~: it is the generic stabilizer 
%for the $\ggot$-action on $V$. Let $H$ be the corresponding connected subgroup. 

%
%Let $s:=\dim G -\dim H$. A linear subspace $\rgot\subset\ggot^*$ is called rational if $\rgot$ is generated by $\rgot\cap\overline{\Delta}$. A $V$-flag $\varphi$ of $\ggot^*$ is a family 
%$$
%\{0\}=\rgot_0\subset\rgot_1\subset\cdots\rgot_{s-1}\subset\rgot_s=\hgot^\perp
%$$

Let $\Delta_G(V)$ be the finite set formed by the infinitesimal stabilizer of points in $V$. For a subspace 
$\hgot\subset \ggot$, we see that $\hgot\in \Delta_G(V)$ if and only if $\hgot^\perp\subset \ggot^*$ is generated by 
$\hgot^\perp\cap \overline{\Wcal}$.

Any vector $v\in V$ decomposes as $v=\sum_{\eta} v_\eta$ with $v_\eta\in V_\eta$. The subalgebra $\ggot_v$ that stabilizes $v$ 
is equal to $\cap_{v_\eta\neq 0} \ker (\bar{\eta})=(\sum_{v_\eta\neq 0}\Rbb\bar{\eta})^\perp$. For a subspace $\hgot\subset \Delta_G(V)$, 
we see that the subspace $V^\hgot:=\{v\ \vert\  \hgot\subset \ggot_v\}$ is equal to $\oplus_{\bar{\eta}\in\hgot^\perp} V_\eta$ 
and $V_\hgot:=\{v\ \vert\ \hgot=\ggot_v\}$ is the subspace $(V^\hgot)^{gen}$ formed by the 
vectors $v:=\sum_{\bar{\eta}\in\hgot^\perp} v_\eta$ such that $\sum_{v_\eta\neq 0}\Rbb\bar{\eta}=\hgot^\perp$.

We have $V/V^\hgot\simeq\sum_{\bar{\eta}\notin\hgot^\perp} V_\eta$. Following Section \ref{sec:decomposition}, we consider a collection $\mathbf{\gamma}:=\{\gamma_\hgot\in\hgot, \hgot\in\Delta_G(V)\}$ such that $(V/V^\hgot)^{\gamma_\hgot}=\{0\}$. We look at the $H$-transversally elliptic symbol $\thom_{\gamma_\hgot}(V/V^\hgot)$ on $V/V^\hgot$. Since the action of $H$ is trivial on $V^\hgot$, 
the following map 
\begin{eqnarray*}
\Ko_G(\T_G^* (V^\hgot)^{gen})&\longrightarrow &\Ko_G(\T_G^* ((V^\hgot)^{gen}\times V/V^\hgot))\\
\sigma&\longrightarrow &\sigma\odot \thom_{\gamma_\rgot}(V/V^\hgot)
\end{eqnarray*}
is well defined. We can compose the previous map with the push-forward morphism $\Ko_G(\T_G^* ((V^\hgot)^{gen}\times V/V^\hgot))\to \Ko_G(\T_G^*V)$: let us denote $\bs^\hgot_\gamma$ the resulting map. 

We can now state Theorem \ref{theo:decomposition-K_G-T_G} in our linear setting.

\begin{theo}\label{theo:decomposition-V}
The map 
$$
\bs_\gamma:=\oplus_{\hgot} \bs^\hgot_\gamma : \bigoplus_{\hgot\in\Delta_G(V)}\Ko_G(\T_G^* (V^\hgot)^{gen})\longrightarrow\Ko_G(\T_G^*V)
$$
is an isomorphism of $R(G)$-modules.
\end{theo}

\medskip

Now we can translate the previous decomposition through the index map. For $\hgot\in\Delta_G(V)$, we consider the 
element $[\wedge^\bullet\overline{V/V^\hgot}]^{-1}_{\gamma_\hgot}\in R^{-\infty}(G)$ which is equal to the $G$-index of  
$\thom_{\gamma_\hgot}(V/V^\hgot)$ (see Definition \ref{def:wedge.inverse} and Proposition \ref{prop.index.Thom.beta.E}).

We need first the following

\begin{lem} The product by $[\wedge^\bullet\overline{V/V^\hgot}]^{-1}_{\gamma_\hgot}$ defines a map from 
$\dm_G(V^\hgot)$ into $\Fcal_G(V)$.
\end{lem}
\begin{proof}  Since the symbol $\thom_{\gamma_\hgot}(V/V^\hgot)$ is $H$-transversally elliptic, the projection $\pi_\hgot: \ggot^*\to\hgot^*$ is proper when restricted to the infinitesimal support  $\overline{\supp(\Omega)}$ of $\Omega:=[\wedge^\bullet\overline{V/V^\hgot}]^{-1}_{\gamma_\hgot}$. Let $\Phi\in\langle R^{-\infty}(G/H)\rangle$: the image of 
$\overline{\supp(\Phi)}$ by $\pi_\hgot$ is finite. It is now easy to check that for any $\chi \in \widehat{G}$ the set 
$\{(\chi_1,\chi_2)\in \supp(\Omega)\times \supp(\Phi)\ \vert \chi_1+\chi_2=\chi \}$ is finite: the product $[\wedge^\bullet\overline{V/V^\hgot}]^{-1}_{\gamma_\hgot}\otimes\Phi$ is 
well-defined.

Let $\Phi\in\langle R^{-\infty}(G/H)\rangle$. For any $\agot\in\Delta_G(V)$ we have the `mother" formula\footnote{See formula (2) in 
\cite{CPV-2010-1}.} 
\begin{equation}\label{eq:mother-formula}
\wedge^\bullet\overline{V/V^\agot}\otimes[\wedge^\bullet\overline{V/V^\hgot}]^{-1}_{\gamma_\hgot}\otimes\Phi=
\wedge^\bullet\overline{V^\hgot/V^{\hgot+\agot}}
\otimes[\wedge^\bullet\overline{V^\agot/V^{\hgot+\agot}}]^{-1}_{\gamma_\hgot}\otimes\Phi
\end{equation} 
which is due to the isomorphisms $V/V^{\agot}\simeq V/(V^\hgot+V^{\agot})\oplus V^\hgot/V^{\hgot+\agot}$, 
$V/V^{\hgot}\simeq$ \break $V/(V^\hgot+V^{\agot})\oplus V^\agot/V^{\hgot+\agot}$, and the relation 
$$
\wedge^\bullet W\otimes[\wedge^\bullet W]^{-1}_{\gamma}=1
$$
that holds for any $G$-module such that $W^\gamma=\{0\}$.

Note that for any $\agot,\hgot\in\Delta_G(V)$ we have the equivalence $V^{\hgot+\agot}=V^\hgot \Longleftrightarrow \agot\subset \hgot$. 
Suppose now that $\Phi\in\dm_G(V^\hgot)$ and consider the product  $\Omega:=[\wedge^\bullet\overline{V/V^\hgot}]^{-1}_{\gamma_\hgot}\otimes\Phi\in R^{-\infty}(G)$. If $\agot\subset \hgot$, we have 
$$
\wedge^\bullet\overline{V/V^\agot}\otimes\Omega=
[\wedge^\bullet\overline{V^\agot/V^{\hgot}}]^{-1}_{\gamma_\hgot}\otimes\Phi \in \langle R^{-\infty}(G/A)\rangle
$$
since $[\wedge^\bullet\overline{V^\agot/V^{\hgot}}]^{-1}_{\gamma_\hgot}\in \langle R^{-\infty}(G/A)\rangle$ and 
$\Phi\in \langle R^{-\infty}(G/H)\rangle \subset \langle R^{-\infty}(G/A)\rangle$. In the other hand, if $\agot\nsubseteq \hgot$, we have $\wedge^\bullet\overline{V/V^\agot}\otimes\Omega=0$ since $\wedge^\bullet\overline{V^\hgot/V^{\hgot+\agot}}\otimes \Phi=0$. 

We have proved that $\Omega=[\wedge^\bullet\overline{V/V^\hgot}]^{-1}_{\gamma_\hgot}\otimes\Phi$ belongs to $\Fcal_G(V)$.
\end{proof}

\medskip

The map 
$$
\mathbf{\Scal}_\gamma : \bigoplus_{\hgot\in\Delta_G(V)}\dm_G(V^\hgot)\longrightarrow \Fcal_G(V)
$$
defined by $\mathbf{\Scal}_\gamma(\oplus_\hgot \Phi_\hgot):=\sum_{\hgot\in\Delta_G(V)}
 [\wedge^\bullet\overline{V/V^\hgot}]^{-1}_{\gamma_\hgot}\otimes \Phi_\hgot$
satisfies the following commutative diagram
$$
\xymatrix@C=30mm{ 
\bigoplus_{\hgot}\Ko_G(\T^*_G(V^\hgot)^{gen})\ar[r]^{\bs_\gamma} \ar[d]^{\oplus_\hgot \indice^G_{V^\hgot}} & 
\Ko_{G}(\T_{G}^*V) \ar[d]^{\indice^G_V} \\
\bigoplus_{\hgot}\dm_G(V^\hgot)\ar[r]^{\mathbf{\Scal}_\gamma} & \Fcal_G(V).
}
$$

Since $\bs_\gamma$ and the index maps $\indice^G_{V^\hgot},\indice^G_{V}$ are isomorphisms we recover the following theorem of de Concini-Procesi-Vergne \cite{CPV-2010-1}.

\begin{theo}\label{theo:decomposition-dm-V}
The map $\mathbf{\Scal}_\gamma$ is an isomorphism of $R(G)$-modules.
\end{theo}

%%%%%%%%%%%%%%%%%%%%%%
%%%%%%%%%%%%%%%%%%%%%%
%%%%%%%%%%%%%%%%%%%%%%
\section{Appendix}
%%%%%%%%%%%%%%%%%%%%%%
%%%%%%%%%%%%%%%%%%%%%%
%%%%%%%%%%%%%%%%%%%%%%

%%%%%%%%%%%%%%%%%%%%%%%%%%%%%%%%%%%%%%%%%%%%
%%%%%%%%%%%%%%%%%%%%%%%%%%%%%%%%%%%%%%%%%%%%
\subsection{Appendix A}
%%%%%%%%%%%%%%%%%%%%%%%%%%%%%%%%%%%%%%%%%%%%
%%%%%%%%%%%%%%%%%%%%%%%%%%%%%%%%%%%%%%%%%%%%
Let $G$ be a compact abelian Lie group, and let $\chi : G\to U(1)$ be a surjective morphism. We want to prove that the sequence 
\begin{equation}\label{eq:sequence-G-chi-appendix}
0\longrightarrow R^{-\infty}(G_\chi)\stackrel{\indGa}{\longrightarrow}
R^{-\infty}(G)\stackrel{\wedge^\bullet \overline{\Cbb_{\chi}}}{\longrightarrow}
R^{-\infty}(G)
\end{equation}
is exact. Note that the induction map $\indGa: R^{-\infty}(G_\chi)\to R^{-\infty}(G)$ is the dual of the restriction morphism 
$R(G)\to R(G_\chi)$. Hence the injectivity of $\indGa$ will follows from the classical 
\begin{lem}\label{lem-appendix-1}
Let $H$ be a closed subgroup of  a compact abelian Lie group $G$.  The restriction $R(G)\to R(H)$ is onto. 
\end{lem}
\begin{proof}
Let $\theta$ be a character of $H$. For any $L^1$-function $\phi: G\to \Cbb$, we consider the average $\tilde{\phi}(g)=\int_H\phi(gh)\theta(h)^{-1}dh$~:  we have then 
\begin{equation}\label{eq:phi-gh}
\tilde{\phi}(gh)=\tilde{\phi}(g)\theta(h)\quad \mathrm{for\ any} \quad (g,h)\in G\times H.
\end{equation}
Let us choose $\phi$ such that $\tilde{\phi}\neq 0$. For any character $\chi: G\to\Cbb$, we consider the function
$$
\tilde{\phi}_\chi(t):=\int_G \tilde{\phi}(tg)\chi(g)^{-1}dg.
$$
We have $\tilde{\phi}_\chi= (\tilde{\phi},\chi)\chi$ where $(\tilde{\phi},\chi)=\int_G \tilde{\phi}(g)\chi(g)^{-1}dg\in \Cbb$. It is immediate that  (\ref{eq:phi-gh}) gives that 
$\tilde{\phi}_\chi(h)=(\tilde{\phi},\chi)\theta(h)$ for $h\in H$. Hence the restriction of $\chi$ to $H$ is equal to $\theta$ when $(\tilde{\phi},\chi)\neq 0$. By a density argument, we know that such $\chi$ exists.
\end{proof}

\medskip

Now we want to prove that ${\rm Image}(\indGa)=\ker (\wedge^\bullet \overline{\Cbb_{\chi}})$. The inclusion \break ${\rm Image}(\indGa)\subset\ker (\wedge^\bullet \overline{\Cbb_{\chi}})$ comes from the fact that $\wedge^\bullet \overline{\Cbb_{\chi}}=0$ in $R(G_\chi)$.

For the other inclusion, we consider $\Phi:=\sum_{\mu\in\widehat{G}}m(\mu)\Cbb_\mu\in \ker (\wedge^\bullet \overline{\Cbb_{\chi}})$. We have the relation $\Phi\otimes \Cbb_\chi=\Phi$,  which means that $m(\mu+\chi)=m(\mu)$ for all $\mu\in\widehat{G}$. Let $\pi: \widehat{G}\to \widehat{G_\chi}$ be the restriction morphism. Thanks to Lemma \ref{lem-appendix-1}, we know that $\pi$ is surjective, and we see that for $\theta\in \widehat{G_\chi}$, $\pi^{-1}(\theta)$ is of the form $\{k\chi +\theta', k\in\Zbb\}$. For $\theta\in \widehat{G_\chi}$, we denote $n(\theta)\in\Zbb$ the integer $m(\mu)$ for $\mu\in\pi^{-1}(\theta)$. We have then
\begin{eqnarray*}
\Phi=\sum_{\mu\in\widehat{G}}m(\mu)\Cbb_\mu
=\sum_{\theta\in\widehat{G_\chi}}\ \sum_{\mu\in\pi^{-1}(\theta)}m(\mu)\Cbb_\mu
&=&\sum_{\theta\in\widehat{G_\chi}} n_\theta\sum_{k\in\Zbb}\Cbb_{k\chi+\theta'}\\
&=&\indGa\left(\sum_{\theta\in\widehat{G_\chi}} n_\theta \Cbb_\theta\right).
\end{eqnarray*}

%%%%%%%%%%%%%%%%%%%%%%%%%%%%%%%%%%%%%%%%%%%%
%%%%%%%%%%%%%%%%%%%%%%%%%%%%%%%%%%%%%%%%%%%%
\subsection{Appendix B}
%%%%%%%%%%%%%%%%%%%%%%%%%%%%%%%%%%%%%%%%%%%%
%%%%%%%%%%%%%%%%%%%%%%%%%%%%%%%%%%%%%%%%%%%%

This section is devoted to the proof of Proposition \ref{prop:thom-pm}. Let $V$ be equipped with the complex structure $J:=J_\beta$. The class $\thom_{\pm\beta}(V)\in \Ko_G(\T_G^* V)$ are represented by the symbols $\Clif(\xi\pm\beta(x)):\wedge^+V\longrightarrow \wedge^- V$. Since $-\thom_{\pm\beta}(V)$ is represented by 
$-\Clif(\xi+\beta(x)):\wedge^-V\longrightarrow \wedge^+ V$, the class $\thom_{-\beta}(V)-\thom_{\beta}(V)$ is represented by the symbol
$$
\tau(x,\xi):\wedge^\bullet V \to \wedge^\bullet V
$$ 
defined by $\tau(x,\xi)=\Clif(\xi)\circ \epsilon -\Clif(\beta(x))$, where $\epsilon(w)=(-1)^{|w|}w$. We consider the family 
$\tau_s(x,\xi)=(s {\rm Id} +\Clif(\xi))\circ \epsilon -\Clif(\beta^s(x)),\quad s\in [0,1]$, where $\beta^s= s J + (1-s) \beta$. 
Note that $\beta^s$ is invertible for any $s\in [0,1]$.

\begin{lem}\label{lem:tau-s}
The family $\tau_s, s\in [0,1]$ is an homotopy of transversally elliptic symbols.
\end{lem}

Thanks to the last lemma, we know that $\tau=\tau_1$ in $\Ko_G(\T_G^* V)$. Since 
$\mathrm{Support}(\tau_1)\cap \T_G^* V\subset \T_G^*( V\setminus \{0\})$, the restriction 
$\tau':=\tau_1\vert_{V\setminus \{0\}}$ is a $G$-transversally elliptic symbol on $V\setminus \{0\}$, and 
the excision property tells us that $j_!(\tau')=\tau_1=\tau$ in $\Ko_G(\T_G^* V)$.

\medskip

For $(x,\xi)\in \T^*(V\setminus \{0\})$, the map 
$\tau'(x,\xi):\wedge^\bullet V\to\wedge^\bullet V$ is given by 
$$
\tau'(x,\xi)=({\rm Id}+\Clif(\xi))\circ \epsilon -\Clif(J x).
$$

Let $S$ be the sphere of radius one of $V$. We work with the isomorphism 
$S\times\Rbb\simeq V\setminus \{0\}, (y,t)\mapsto e^t y$. Let $\cu=S\times\Rbb\times \Cbb$ be the trivial complex vector bundle. Let $\Hcal\to S\times\Rbb$ be the 
vector bundle defined by $\Hcal_{(y,t)}:=(\Cbb y)^\perp\subset \T_y S$. We use the isomorphism of vector bundle
$$
\phi:\Hcal\oplus\cu\longrightarrow \T(S\times \Rbb)
$$
defined by $\phi_{(y,t)}(\xi'\oplus a+ib)=(\xi'+ b J(y),a)\in \T_y S\times \T_t\Rbb$. Through $\phi$ 
the bundle map $\Clif(\xi):\wedge^+ V\to \wedge^- V$ for $\xi\in\T_x V$ becomes
$$
\Clif_{(y,t)}(\xi'\oplus z):(\wedge \Hcal_y\otimes\wedge\Cbb)^+\to (\wedge \Hcal_y\otimes\wedge\Cbb)^-
$$

Through $\phi$, the vector field $x\mapsto J_\cgot x$ becomes the section of $\cu$ given by $(y,t)\mapsto e^t i$, and the morphism $\tau'$ is defined as follows: for $(y,t)\in S\times \Rbb$, and $\xi'\oplus z\in \Hcal_y\oplus \Cbb$, the map $\tau'_{(y,t)}(\xi'\oplus z):\wedge \Hcal_y\otimes\wedge\Cbb\to \wedge \Hcal_y\otimes\wedge\Cbb$ 
is defined by 
$$
\tau'_{(y,t)}(\xi'\oplus z)=({\rm Id}+\Clif(\xi'\oplus z))\circ \epsilon -e^t\Clif(i).
$$

Let $A_{z,\xi'}=\Clif(\xi'\oplus z)$ and $B=\Clif(i)$ be the maps
from $(\wedge \Hcal_y\otimes\wedge\cu)^+$ into  $(\wedge \Hcal_y\otimes\wedge\cu)^-$. The matrix of 
$\tau'_{(y,t)}(\xi'\oplus z)$ relatively to the grading of $\wedge \Hcal_y\otimes\wedge\Cbb$ is 
$$
\left(\begin{array}{cc} 
{\rm Id}& A_{z,\xi'}^*+ e^t B^* \\ 
A_{z,\xi'}-e^t B& -{\rm Id}
\end{array}\right).
$$

Let us consider the deformation of $\tau'$ in a family 
$$
\sigma_s:=\left(\begin{array}{cc} 
{\rm Id} & sA_{z,\xi'}^*+ e^t B^* \\ 
A_{z,\xi'}-e^t B& -{\rm Id}
\end{array}\right),\ s\in[0,1].
$$

\begin{lem}\label{lem:sigma-s}
The family $\sigma_s, s\in [0,1]$ is an homotopy of transversally elliptic symbols.
\end{lem}

The symbol
$\sigma_0:=\left(\begin{array}{cc} 
{\rm Id} & e^t B^* \\ 
A_{z,\xi'}-e^t B& -{\rm Id}
\end{array}\right)
$
 is clearly clearly homotopic to 
\begin{eqnarray*}
\sigma_2&:=&\left(\begin{array}{cc} {\rm Id} & 0 \\ e^{-t} B& {\rm Id}\end{array}\right)\sigma_0
\left(\begin{array}{cc} {\rm Id} & 0\\ -e^{-t} B& {\rm Id} \end{array}\right)\\
&=& \left(\begin{array}{cc} 0 & e^t B^*\\ A_{z,\xi'}-(e^t - e^{-t}) B& 0\end{array}\right)
\end{eqnarray*}

Since the morphism $e^t B: (\wedge \Hcal\otimes\wedge\cu)^-\to (\wedge \Hcal\otimes\wedge\cu)^+$ is always 
invertible, its class vanishes. Hence we have
$$
[\tau']=[\sigma_0]=[\sigma_2]=[A_{z,\xi'}-(e^t - e^{-t}) B]\quad {\rm in}\quad \Ko_G(\T_G^*(S\times \Rbb)).
$$
We are now working with the morphism $\sigma_3: (\wedge \Hcal\otimes\wedge\cu)^+\to (\wedge \Hcal\otimes\wedge\cu)^-$
defined by
$$
\sigma_3(\xi'\oplus z)= \Clif(\xi')+\Clif(z-(e^t-e^{-t})i).
$$
Since $\frac{e^t-e^{-t}}{t}>0$ on $\Rbb$, we can deform the term $z-(e^t-e^{-t})i$ in $t+i\mathrm{Re}(z)$ without 
changing the intersection of the support with $\T_G^*(S\times\Rbb)$.

Finally we have proved that $\thom_{-\beta}(V)-\thom_{\beta}(V)$ is represented on $S\times \Rbb$ 
by the morphism $\Clif(\xi')+\Clif(t+i\mathrm{Re}(z)): 
(\wedge \Hcal\otimes\wedge\cu)^+\to (\wedge \Hcal\otimes\wedge\cu)^-$
which is by definition equal to $\sigmad^V\odot \bott(\T\Rbb)=i_!(\sigmad^V)$.

\medskip

\medskip

We finish this section with the proofs of the deformation Lemmas. For the family $\tau_s(x,\xi)=(s {\rm Id} +\Clif(\xi))\circ \epsilon -\Clif(\beta^s(x))$, we have 
$$
(\tau_s(x,\xi))^*\tau_s(x,\xi)= 
\left(\begin{array}{cc} 
s^2+\|\xi-\beta^s(x)\|^2 & -2s\Clif(\beta^s(x))\\ 
2s\Clif(\beta^s(x))& s^2+\|\xi+\beta^s(x)\|^2
\end{array}\right)
$$
Then $\det(\tau_s(x,\xi))=0$ if and only if 
$$
(s^2+\|\xi-\beta^s(x)\|^2)(s^2+\|\xi+\beta^s(x)\|^2)= 4s^2\|\beta^s(x)\|^2
$$
which is equivalent to the equality $(s^2+\|\xi\|^2+\|\beta^s(x)\|^2)^2= 4s^2\|\beta^s(x)\|^2+ 4(\xi,\beta^s(x))^2$. 
If $\xi\notin\Rbb\beta^s(x)$, we have $(\xi,\beta^s(x))^2<\|\xi\|^2\|\beta^s(x)\|^2$, and then
$$
(s^2+\|\xi\|^2+\|\beta^s(x)\|^2)^2<4(s^2+\|\xi\|^2)\|\beta^s(x)\|^2
$$
which gives $(s^2+\|\xi\|^2-\|\beta^s(x)\|^2)^2<0$ which is contradictory. Then $\det(\tau_s(x,\xi))$ $=0$ if and only if 
$\xi\in\Rbb\beta^s(x)$ and $s^2+\|\xi\|^2-\|\beta^s(x)\|^2=0$. 
If furthermore $\xi\in\T_G^* V\vert_x$, then\footnote{It is due to the fact that $(\beta^s(x),\beta(x))> 0$ 
when $x\neq 0$.}  $\xi=0$.  We have proved that $\mathrm{Support}(\tau_s)\cap \T_G^* V$ is equal to the 
compact set $\{(x,\xi)\ \vert\ \xi=0\ \mathrm{and}\ s^2-\|\beta^s(x)\|^2=0\}$. So $s\in[0,1]\to \tau_s$ is an homotopy of transversally elliptic symbols. 

\medskip

%%Recall the following relations: $B^*B=BB^*= Id$, $A^*A=AA^*= (\|\xi'\|^2+\|z\|^2)Id$ and 
%%$A^*B+B^*A=AB^*+BA^*= 2\mathrm{Im}(z) Id$. We have 
%%\begin{eqnarray*}
%%(\sigma_s)^*\sigma_s &=&
%%\left(\begin{array}{cc} {\rm Id} & A^*- e^t B^* \\ sA+e^t B& - {\rm Id} \end{array}\right)
%%\left(\begin{array}{cc} {\rm Id} & sA^*+ e^t B^* \\ A-e^t B& - {\rm Id} \end{array}\right)\\
%%&=&
%%\left(\begin{array}{cc} \rho^+\,{\rm Id} & (1-s)A^*+ 2e^t B^* \\ (1-s)A+ 2e^t B & \rho^-(s)Id\end{array}\right)
%%\end{eqnarray*}

For the family\footnote{We write $A$ for $A_{z,\xi'}$.} $\sigma_s:=\left(\begin{array}{cc} 
{\rm Id} & sA^*+ e^t B^* \\ 
A-e^t B& -{\rm Id}
\end{array}\right),$ 
we have 
$$
(\sigma_s)^*\sigma_s =
\left(\begin{array}{cc} \rho^+\,{\rm Id} & (1-s)A^*+ 2e^t B^* \\ (1-s)A+ 2e^t B & \rho^-(s)Id\end{array}\right)
$$
with $\rho^+= 1+\|\xi'\|^2+\|z-e^t i\|^2$ and $\rho^-(s)=1+\|s\xi'\|^2+\|sz+ e^t i\|^2$. We check 
easily that $((1-s)A^*+ 2e^t B^*)((1-s)A+ 2e^t B)=\rho(s){\rm Id}$ with 
$$
\rho(s)=\|(s-1)\xi'\|^2+\|(s-1)z+ 2e^t i\|^2.
$$
Finally $\det(\sigma_s)=0$ if and only if $\rho(s)=\rho^-(s)\rho^+$. In other words, $(y,t;\xi'\oplus z)$ belongs to the 
support of $\sigma_s$ if and only if
$$
\|(s-1)\xi'\|^2+\|(s-1)z+ 2e^t i\|^2=\left(1+\|s\xi'\|^2+\|sz+ e^t i\|^2\right)\left(1+\|\xi'\|^2+\|z-e^t i\|^2\right).
$$

Let us suppose now that $\xi'\oplus z\in \T_G^*(S\times \Rbb)$. It imposes $\mathrm{Im}(z)=0$, and the last relation 
becomes $(s-1)^2\Theta+ 4e^{2t}=\left(1+e^{2t} + s^2\Theta  \right)\left(1+e^{2t}+\Theta\right)$ 
with $\Theta=\|\xi'\|^2+\|z\|^2$. It is easy to see that the last relation holds if and only if $t=\Theta=0$. Finally 
we have proved that
$$
\mathrm{Support}(\sigma_s)\cap \T_G^*(S\times \Rbb)=\left\{(y,t;\xi'\oplus z)\ \vert\  t=0,\ \xi'=0, \ z=0\right\}, 
$$
and then $\sigma_s, s\in[0,1]$ defines an homotopy of transversally elliptic symbols.

%%%%%%%%%%%%%%%%%%%%%%%%%%%%%%%%%%%%%%%%%%%%%%%%%%%%%%%%%%%%%%%%%%%%%%%%%%


{\small

\begin{thebibliography}{99}

\bibitem{Atiyah74} {\sc M.F. Atiyah}, Elliptic operators and
compact groups, Springer, 1974. Lecture notes in Mathematics, {\bf
401}.

\bibitem{Atiyah89}{\sc M.F. Atiyah}, K-theory. Advanced Book Classics, 2nd edn. Addison-Wesley Publishing Company
Advanced Book Program, Redwood City (1989). Notes by D.W. Anderson

\bibitem{Atiyah-Segal68} {\sc M.F. Atiyah} and {\sc G.B. Segal},
{\em The index of elliptic operators II}, Ann. Math. {\bf 87}, 1968,
p. 531-545.

\bibitem{Atiyah-Singer-1} {\sc M.F. Atiyah} and {\sc I.M. Singer},
{\em The index of elliptic operators I}, Ann. Math. {\bf 87}, 1968,
p. 484-530.

\bibitem{Atiyah-Singer-2} {\sc M.F. Atiyah} and {\sc I.M. Singer},
{\em The index of elliptic operators III}, Ann. Math. {\bf 87},
1968, p. 546-604.

\bibitem{Atiyah-Singer-3} {\sc M.F. Atiyah} and {\sc I.M. Singer},
{\em The index of elliptic operators IV}, Ann. Math. {\bf 93}, 1971,
p. 139-141.


\bibitem{B-V.inventiones.96.1} {\sc N. Berline} and {\sc M. Vergne},
The Chern character of a transversally elliptic symbol and the
equivariant index, {\em Invent. Math.}, {\bf 124}, 1996, p. 11-49.

\bibitem{B-V.inventiones.96.2} {\sc N. Berline} and {\sc M. Vergne},
L'indice \'equivariant des op\'erateurs transversalement
elliptiques, {\em Invent. Math.}, {\bf 124}, 1996, p. 51-101.





\bibitem{CPV-2010-1} {\sc C. De Concini}, {\sc C. Procesi} and {\sc M. Vergne}, 
Vector partition functions and generalized dahmen and micchelli spaces
{\em Transformation Groups}, Vol.15, No. 4, 2010, p. 751-773.

\bibitem{CPV-2010-2} {\sc C. De Concini}, {\sc C. Procesi} and {\sc M. Vergne}, 
Vector partition functions and index of transversally elliptic operators 
{\em Transformation Groups}, Vol.15, No. 4, 2010, p. 775-811.


\bibitem{CPV-2012} {\sc C. De Concini}, {\sc C. Procesi} and {\sc M. Vergne}, Box splines and the equivariant index theorem,
{\em J. Inst. Math. Jussieu}, 2012

\bibitem{pep-RR} {\sc P-E. Paradan}, Localization of the Riemann-Roch
character, {\em J. Funct. Anal.} {\bf 187}, 2001, p. 442-509.

\bibitem{pep-vergneIII} {\sc P-E. Paradan} and {\sc M. Vergne}, Equivariant Chern characters with generalized coefficients,
arXiv:0801.2822, math.DG.

\bibitem{pep-vergne:bismut} {\sc P-E. Paradan} and {\sc M. Vergne}, Index of transversally elliptic operators, 
{\em Ast\'erique, Soc. Math. Fr.}, {\bf 328}, 2009, p. 297-338.


\bibitem{Segal68} {\sc G. Segal}, Equivariant K-Theory, {\em Publ. Math. IHES}, {\bf 34}, 1968, p. 129-151.


\end{thebibliography}
}
\end{document}